\documentclass{amsart}

\usepackage{lmodern}
\usepackage{float,enumitem}
\usepackage[dvipsnames,svgnames,x11names,hyperref]{xcolor}
\usepackage{mathtools,url,graphicx,verbatim,amssymb,stmaryrd,nicefrac,amscd, eufrak}
\usepackage[pagebackref,colorlinks,citecolor=Mahogany,linkcolor=Mahogany,urlcolor=Mahogany,filecolor=Mahogany]{hyperref}
\usepackage{microtype}
\usepackage{xcolor, soul}
\usepackage[margin=1.5in]{geometry}
\usepackage{setspace}
\usepackage{tikz, tikz-cd}
\usetikzlibrary{matrix,calc,arrows,patterns,decorations.pathreplacing}
\usepackage{caption}
\graphicspath{ {./fig/} }

\newtheorem{theorem}{Theorem}[section]
\newtheorem{lemma}[theorem]{Lemma}
\newtheorem{proposition}[theorem]{Proposition} 
\newtheorem{corollary}[theorem]{Corollary}

\newtheorem*{corollary*}{Corollary}

\newtheorem{atheorem}{Theorem}

\newtheorem{acorollary}[atheorem]{Corollary}

\theoremstyle{definition}
\newtheorem{definition}[theorem]{Definition}
\newtheorem{example}[theorem]{Example} 
\newtheorem{remark}[theorem]{Remark}
\newtheorem{observation}[theorem]{Observation}
\newtheorem{notation}[theorem]{Notation}

\newtheorem{question}[theorem]{Question}

\renewcommand{\Bbb}{\mathbb}

\author{Balarka Sen}
\title{$h$--Principle for loose Legendrian embeddings}

\begin{document}

\begin{abstract}This article provides an exposition of Emmy Murphy's work on loose Legendrian embeddings. After a brief review of the rudiments of contact topology, we state and discuss some foundational results from the theory of $h$--principles, providing many relevant examples from contact topology on the way. We then proceed to prove Murphy's $h$--principle for loose Legendrian embeddings. We also provide an accessible exposition of some background material from microlocal sheaf theory. As applications, we demonstrate the existence of non-loose Legendrian embeddings, and prove a version of Gromov's nonsqueezing theorem for loose charts.
\end{abstract}

\maketitle

\tableofcontents

\section{Introduction}\label{sec-intro}

\noindent {\bf Context.} A contact structure $\xi$ on an odd-dimensional manifold $Y^{2n+1}$ is a hyperplane field on the manifold which is totally non-integrable. In particular, a contact manifold does not admit integral submanifolds of codimension one. In fact, much more is true: the integral submanifolds of $\xi$ are necessarily of dimension at most $n$. We call integral submanifolds of the critical dimension $n$ as Legendrian submanifolds of the contact manifold. In this article, we address the question of classification of Legendrian embeddings $f : \Lambda^n \to (Y^{2n+1}, \xi)$ of a fixed manifold into a fixed contact manifold, upto isotopy through Legendrian embeddings, for high dimensions ($2n + 1\geq 5$).

This question has a long history in the case $2n + 1 = 3$, which pertains to the theory of Legendrian knots and links in contact $3$--manifolds. Apart from the underlying smooth isotopy type, nullhomologous Legendrian knots have ``classical invariants", given by the Thurston-Bennequin number and the rotation number \cite[Chapter 3.5]{geigbook}. In general, the classical invariants and the underlying smooth isotopy type do not suffice to determine the Legendrian isotopy type. To distinguish knots with the same underlying smooth isotopy type as well as same classical invariants, finer and more complicated invariants coming from pseudoholomorphic curves \cite{chek} (in which case the invariant is a noncommutative dga) and more recently, microlocal sheaf theory \cite{stz} (in which case the invariant is a dg derived category) has been developed.

In higher dimensions, the analogue of the classical invariants is the ``formal Legendrian isotopy type" of a Legendrian submanifold, which is a purely homotopy-theoretic invariant of the Legendrian embedding. In \cite{mur}, Emmy Murphy defines a large class of Legendrian submanifolds, termed ``loose", and she proves that the formal Legendrian isotopy type determines the Legendrian isotopy type for such Legendrian submanifolds. This effectively establishes an $h$--principle, in the sense of Gromov \cite{grobook}, for loose Legendrian embeddings. Looseness of a given Legendrian submanifold depends on the existence of a very specific isocontactly embedded adapted chart, termed a ``loose chart", on the submanifold. Existence of such a chart has strong consequences for the geometry of the submanifold; for instance, the higher dimensional analogues of the Legendrian invariants coming from pseudoholomorphic curves and microlocal sheaf theory both vanish for loose Legendrian submanifolds. 

The proof of Murphy's theorem uses a number of foundational tools from the theory of $h$--principles developed by Gromov \cite{grobook} and Eliashberg-Mishachev \cite{embook, empaper}. The purpose of this article is to present an accessible exposition of the proof, with a focus on the necessary prerequisites, and providing many examples on the way.

\subsection{Statement of results}
Let $(Y^{2n+1}, \xi)$ be a contact manifold of dimension $2n+1 \geq 5$, and $\Lambda^n$ be a smooth $n$--dimensional manifold. Let $U \subset Y$ be a fixed Darboux chart, and $\phi : D^n \to U$ be a fixed Legendrian immersion of a fixed disk $D^n \subset \Lambda$ in $U$ parametrizing a fixed loose chart $(U, \phi(D^n))$ (Definition \ref{def-loose}). Let us denote $\mathrm{Emb}_{\rm{Leg}}(\Lambda, Y)$ and $\mathrm{Emb}_{\rm{Leg}, \ell}(\Lambda, Y; U)$ to be the space of Legendrian embeddings of $\Lambda$ in $Y$, and the space of loose Legendrian embeddings of $\Lambda$ in $Y$ with fixed loose chart $U$ (Definition \ref{def-looseLegsub}), respectively. Let us also denote $\mathrm{Emb}^f_{\rm{Leg}}(\Lambda, Y)$ and $\mathrm{Emb}_{\rm{Leg}, \ell}^f(\Lambda, Y; U)$ to be the space of formal Legendrian embeddings of $\Lambda$ in $Y$ (Definition \ref{def-flegemb}), and the space of formal Legendrian embeddings of $\Lambda$ in $Y$ with a fixed loose chart $U$ (Definition \ref{def-fLegloose}), respectively.

The upshot of the article are the following two results by Murphy \cite[Theorem 1.3., Theorem 1.2.]{mur}, discussed in detail in the article in Theorem \ref{thm-hloose} and Corollary \ref{cor-hloosepi0}, respectively.

\begin{atheorem}\label{thm-main}
Given a $D^d$--parametric family of formal Legendrian embeddings of $\Lambda$ in $Y$ with fixed loose chart $U$ which restricts to a $\partial D^d$--parametric family of holonomic Legendrian embeddings, there is a homotopy rel $\partial D^d$ of $D^d$--parametric families of formal Legendrian embeddings (not necessarily with fixed loose chart $U$) starting at the given family and ending at a $D^d$--parametric family of holonomic Legendrian embeddings of $\Lambda$ in $Y$.

In other words, the inclusion map of pairs,
$$(\mathrm{Emb}^f_{\rm{Leg}, \ell}(\Lambda, Y; U), \mathrm{Emb}_{\rm{Leg}, \ell}(\Lambda, Y; U)) \hookrightarrow (\mathrm{Emb}^f_{\rm{Leg}}(\Lambda, Y), \mathrm{Emb}_{\rm{Leg}}(\Lambda, Y))$$
induces the zero map on all relative homotopy groups.
\end{atheorem}

\begin{acorollary}\label{cor-main}Let $f_0, f_1 : \Lambda \to (Y, \xi)$ be a pair of loose Legendrian embeddings which are isotopic through formal Legendrian embeddings. Then they are also Legendrian isotopic.
\end{acorollary}

In Section \ref{sec-microsheaf}, we discuss some foundational results in microlocal sheaf theory, which we combine with Theorem \ref{thm-main} and Corollary \ref{cor-main} to obtain the following applications. Both the results are well-known to experts, but complete proofs may be slightly difficult to locate in the existing literature. These are discussed in detail in the article in Proposition \ref{prop-saucernloose} and Proposition \ref{prop-loosenonsq}, respectively.

\begin{atheorem}The Legendrian flying saucer $\Lambda \subset (\Bbb R^{2n+1}, \xi_{\rm{std}})$ is not loose.\end{atheorem}

\begin{acorollary}A loose chart does not admit a contact relative embedding in a pseudo-loose chart. More precisely, let $(C' \times B_{\rho'}, \mathcal{Z}_{a'} \times J_{\rho'})$ with size parameter $\rho'^2/a' > 1/2$ be a loose chart and $(C \times B_{\rho}, \mathcal{Z}_a \times J_{\rho})$ with size parameter $\rho^2/a < 1/2$ be a pseudo-loose chart. Then, there does not exist a contact embedding 
$$\varphi : C' \times B_{\rho'} \to C \times B_\rho,$$
such that $\varphi(\mathcal{Z}_{a'} \times J_{\rho'}) \subset \mathcal{Z}_a \times J_\rho$.\end{acorollary}

\subsection{Outline of the article} In Section \ref{sec-prel} we provide a lightning tour of the foundations of contact and symplectic topology. The key takeways are Example \ref{eg-unic}, which appears throughout the article as a guiding example, and Moser's lemma (Lemma \ref{lem-moser}), which is used in subsequent sections to prove local integrability and microflexibility of various differential relations in contact topology. In Section \ref{sec-legknots}, we discuss the $h$--principle for approximation of a smooth knot by Legendrian knots. Here we introduce the Legendrian zig-zag (Observation \ref{obs-intpzigzag}) as an object playing a crucial role in establishing the $h$--principle, which appears later in the $h$--principle for wrinkled Legendrians (Theorem \ref{thm-hwrinkleg}), as well as in the very definition of loose charts (Definition \ref{def-loose}) themselves. This also sets up the groundwork for the $h$--principle for general Legendrian immersions discussed later in Section \ref{sec-hLegimm}.

Section \ref{sec-holapp} discusses the two cornerstone results in the theory of $h$--principles, namely the Holonomic Approximation Theorem (Theorem \ref{thm-holapp}) and the Convex Integration Theorem (Theorem \ref{thm-hample}). In Section \ref{sec-hLegimm}, we prove the $h$--principle for Legendrian immersions as a corollary of the Holonomic Approximation Theorem, and pose the key question of classification of Legendrian immersions upto Legendrian isotopies (see, Question \ref{que-htpylegemb}). Section \ref{sec-dirimmemb} uses the Convex Integration Theorem to deduce the $h$--principle for embeddings and immersions directed by ample differential relations. This is later used in Section \ref{sec-flegemb} to isotope formal Legendrian embeddings to more geometrically tractable embeddings, called $\varepsilon$--Legendrian embeddings (see, Definition \ref{defn-epsLeg} and Proposition \ref{prop-epsLeg}). 

In Section \ref{sec-wrinkemb} we discuss the $h$--principle for wrinkled embeddings (Theorem \ref{thm-hwrink}), and in Section \ref{sec-flegemb} we construct fronts for formal Legendrian embeddings over which they are graphical (Theorem \ref{thm-graphfLeg}). In Section \ref{sec-wrinkleg} we combine techniques from Section \ref{sec-legknots}, Section \ref{sec-flegemb} and Section \ref{sec-wrinkemb} to prove that $h$--principle for Legendrian embeddings holds with a caveat: one must allow mild singularities, given by Legendrian lifts of wrinkle singularities appearing in the front projection (Theorem \ref{thm-hwrinkleg}, see also Definition \ref{def-wrinkleg}). Finally, we introduce loose Legendrians in Section \ref{sec-loose} and prove the ``without caveats" version of the $h$--principle for loose Legendrian embeddings in Theorem \ref{thm-main}. This is done by applying Theorem \ref{thm-hwrinkleg} and removing the wrinkle singularities using the loose charts. Corollary \ref{cor-main} is proved as a consequence.

In Section \ref{sec-microsheaf}, we discuss the problem of detecting non-loose Legendrian submanifolds and give examples of such as well. We focus on Legendrian isotopy invariants coming from microlocal sheaf theory, explicating a comment in \cite{mur} observing that loose Legendrians submanifolds in cosphere bundles are not microsupport of constructible sheaves. The exposition is carefully crafted with an audience of geometrically-minded topologists in mind. As an application, we prove a contact analogue of Gromov's nonsqueezing theorem for loose charts (Proposition \ref{prop-loosenonsq}). \newline

\noindent {\bf Acknowledgements.} I would like to thank Professor Mahan Mj for his guidance during this project as well as his interest regarding $h$--principles in general, and Professor Mahuya Datta for introducing me to Murphy's paper and discussing it with me in detail during a summer project in 2022. I would also like to thank Professor Emmy Murphy and Professor Dishant Pancholi for their kind and generous feedback on a draft of this article. Finally, I would like to thank my colleagues and friends Biswajit Nag, Ritwik Chakraborty, Supravat Sarkar, Shivansh Tomar, Ritvik Radhakrishnan, Abhishek Khannur, Bhaswar Bhattacharya, Bhishek Garg, as well as the participants of the TIFR Mathematics Students’ Seminar for patiently and critically listening to my thoughts on contact topology and microlocal sheaf theory.

\section{Preliminaries from contact and symplectic topology}\label{sec-prel}

In this section, we review some basic notions from symplectic and contact topology. For a detailed introduction, we refer the reader to \cite{msbook} and \cite{geigbook}. A concise discussion leading up to Moser's lemma can be found in \cite[Chapter 9]{embook}.

\begin{definition}\label{def-contmfld}A \emph{contact manifold} is a pair $(Y, \xi)$ consisting of a smooth manifold $Y$ of odd dimension $\dim Y = 2n+1$, and a smooth, codimension--$1$ vector subbundle $\xi \subset TY$ which is \emph{totally non--integrable}, in the following sense: for any point $p \in Y$, and a choice of a $1$--form $\alpha$ on a neighborhood $U \subset Y$ of $p$ such that $\ker \alpha = \xi|_U$, the $(2n+1)$--form 
$$\alpha \wedge (d\alpha)^{\wedge n} = \alpha \wedge d\alpha \wedge \cdots \wedge d\alpha$$
is nowhere--vanishing on $U$. We shall refer to the hyperplane field $\xi \subset TY$ as a \emph{contact structure} on $Y$. A contact structure is said to be \emph{co--orientable} if there exists a globally defined $1$--form $\alpha$ on $Y$ such that $\ker \alpha = \xi$. Henceforth, we shall implicitly assume all the relevant contact structures are co--orientable. \end{definition}

We start with a physically motivating example of a contact manifold.

\begin{example}[Phase space of a unicycle]\label{eg-unic}Consider a unicycle on a $2$--dimensional cartesian plane. We will choose coordinates $(x, y)$ for the position of the unicycle, and angle $\theta$ which the steering wheel makes with the $x$--axis. In other words, our model of the unicycle has three degrees of freedom: two of them given by position in $\mathbb{R}^2$ and the other given by the angle of the steering wheel in $S^1$. Thus, the configuration space of the unicycle is $Y = \mathbb{R}^2 \times S^1$. 

The trajectory of such a unicycle in motion $\sigma(t) = (x(t), y(t))$ on $\mathbb{R}^2$ must have tangent vector $\sigma'(t) = (\dot{x}(t), \dot{y}(t))$ parallel to the direction vector of the steering wheel $(\cos \theta(t), \sin \theta(t))$, at all times. In other words, the motion of the unicycle is constrained by the following ``differential relation" on $Y$, which is presently an underdetermined ODE:
$$\dot{x} \cos(\theta) - \dot{y} \sin(\theta) = 0.$$
The relation can be described as a rank $2$ subbundle $\xi \subset TY$ given by 
$$\xi = \ker \alpha,$$ 
where $\alpha = \cos(\theta) dx - \sin(\theta) dy$ is a differential $1$--form on $Y$. Trajectories of motion of the unicycle are therefore paths $\gamma$ in $Y$ which are parallel to the distribution $\xi$, i.e., $\gamma' \in \xi$.

$(Y, \xi)$ is an example of a contact manifold. Indeed, 
\begin{align*}\alpha \wedge d\alpha 
&= (\cos(\theta) dx - \sin(\theta) dy) \wedge (-\sin(\theta) d\theta \wedge dx - \cos(\theta) d\theta \wedge dy)\\
&= \sin^2(\theta) dy \wedge d\theta \wedge dx - \cos^2(\theta) dx \wedge d\theta \wedge dy\\
&= dx \wedge dy \wedge d\theta,\end{align*}
which is a nonzero volume form on $Y$.
\end{example}

Contact manifolds are odd dimensional analogues of \emph{symplectic manifolds}. Recall that a pair $(V, \omega)$ consisting of a vector space $V$ of even dimension $\dim V = 2n$, equipped with an alternating bilinear form $\omega \in \Lambda^2 V^*$ is a \emph{symplectic vector space} if $\omega$ is nondegenerate, or equivalently, if $\omega^{\wedge n} \neq 0$.

\begin{definition}\label{def-sympmfldbdl}A real vector bundle $(E, B, \pi)$ of even rank $\mathrm{rk}(E) = 2n$, equipped with a fiberwise $2$--form $\omega \in \Gamma(B; \Lambda^2 E^*)$ is a \emph{symplectic vector bundle} if for all $p \in B$, the fiber $E_p := \pi^{-1}(p)$ equipped with the alternating bilinear form $\omega_p$ is a symplectic vector space.

A \emph{symplectic manifold} is a pair $(X, \omega)$ consisting of a smooth manifold $X$ of even dimension $\dim X = 2n$, and a $2$--form $\omega \in \Omega^2(X)$, such that
\begin{enumerate}
    \item The tangent bundle $(TX, X, \pi)$ equipped with the fiberwise symplectic form $\omega \in \Omega^2(X) = \Gamma(X; \Lambda^2 T^*X)$ is a symplectic vector bundle, and
    \item $\omega$ is a closed form.
\end{enumerate}
\end{definition}

Suppose $(Y^{2n+1}, \xi)$ is a contact manifold. Let $\alpha$ be a $1$--form such that $\ker \alpha = \xi$. Then the subbundle $\xi \subset TY$ equipped with the $2$--form $\omega := d\alpha$ is a symplectic vector bundle. Indeed, for all $p \in Y$, $\omega^{\wedge n} = (d\alpha)^{\wedge n}$ must be nonzero on $\xi_p$, as otherwise $\alpha \wedge (d\alpha)^{\wedge n}$ would vanish on $T_p Y \cong \xi_p \oplus T_pY/\xi_p$. Note that while there is no canonical choice for such a $1$--form $\alpha$, for any other $\alpha'$ satisfying $\ker \alpha' = \xi$, we must have $\alpha' = f \alpha$ for some nowhere--vanishing smooth function $f : Y \to \mathbb{R}$. Then the corresponding fiberwise symplectic form on $\xi$ are related by
$$d\alpha' = df \wedge \alpha + f d\alpha = f d\alpha.$$
Here, the first term vanishes as $\alpha \equiv 0$ on $\xi$. Hence, there is a well--defined conformal class of a fiberwise symplectic structure on the hyperplane field $\xi$, which we denote as $\mathrm{CS}(\xi)$. 

\begin{definition}\label{def-sycomor}\leavevmode
\begin{enumerate}
    	\item Given a pair of symplectic manifolds $(X, \omega), (X', \omega')$, an immersion $f : X \to X'$ is called \emph{isosymplectic} if $f^*\omega' = \omega$. A diffeomorphism $f : X \to X'$ that is also an isosymplectic immersion is called a \emph{symplectomorphism}.
    	\item Given a pair of contact manifolds $(Y, \xi), (Y', \xi')$ with $\xi = \ker \alpha$ and $\xi' = \ker \alpha'$, an immersion $f : Y \to Y'$ is called \emph{isocontact} if $f^*\alpha' = h \alpha$ for some positive function $h > 0$ on $Y$. A diffeomorphism $f : Y \to Y'$ that is also an isocontact immersion is called a \emph{contactomorphism}.
\end{enumerate}
\end{definition}

\begin{remark}\label{rmk-sycomor}Note that in Definition \ref{def-sycomor}, a contactomorphism $f : Y \to Y'$ may be equivalently defined as a diffeomorphism satisfying $f^*\xi' = \xi$. This suggests that while studying contact manifolds, the hyperplane distribution should be prioritized over any particular $1$--form that may have it as its kernel. This is in contrast to the case of symplectic manifolds, where the symplectic $2$--form is synonymous with the symplectic structure. These two notions are related by the following observation. For a contactomorphism $f : (Y, \xi) \to (Y', \xi')$, we have $f^* \mathrm{CS}(\xi') = \mathrm{CS}(\xi)$. Hence, the conformal symplectic structures on the respective hyperplane fields are preserved under contactomorphisms.
\end{remark}

For any contact manifold $(Y, \xi)$, there is a canonical choice of a vector field transverse to the contact distribution $\xi \subset TY$.

\begin{definition}\label{def-reeb}Let $(Y, \xi)$ be a contact manifold, and suppose $\xi = \ker \alpha$. We define the \emph{Reeb vector field} $R$ on $Y$ to be the unique vector field satisfying $\alpha(R) = 1$ and $i_R d\alpha = 0$. \end{definition}

\begin{remark}\label{rmk-reeb}We note that the Reeb vector field $R$ is uniquely determined by the choice of the contact form $\alpha$, therefore regardless of choices the \emph{Reeb direction} spanned by $R$ only depends on the contact structure $\xi$. Observe also that, by Cartan's magic formula, 
$$\mathcal{L}_R\alpha = i_R d\alpha + d i_R \alpha = 0$$
Let $\{\phi^t_R : t \geq 0\}$ be the flow defined by the vector field $R$, known as the \emph{Reeb flow}. The Lie derivative computation above shows $(\phi^t_R)^*\alpha = \alpha$, hence $\{\phi^t_R : t \geq 0\}$ is a flow of contactomorphisms. In general, vector fields on a contact manifold generating a flow of contactomorphisms are known as \emph{contact vector fields}.\end{remark}

The canonical examples of symplectic and contact manifolds are the so-called standard symplectic and contact structures on the Euclidean space, respectively:

\begin{example}\label{eg-stansympcont}\leavevmode
\begin{enumerate}
	\item The \emph{standard symplectic structure} on $\Bbb R^{2n}$ is given by the $2$--form: 
			$$\omega_{\rm{std}} := \sum_{i = 1}^n dx_i \wedge dy_i.$$ 
		Here, $(x_1, \cdots, x_n, y_1, \cdots, y_n)$ are a chosen coordinate system for $\Bbb R^{2n}$.
	\item The \emph{standard contact structure} on $\Bbb R^{2n+1}$ is given by the distribution:
			\begin{gather*}
			\xi_{\rm{std}} = \ker \alpha_{\rm{std}},\\
			\alpha_{\rm{std}} := dz - \sum_{i = 1}^n y_i dx_i.
			\end{gather*}
    		Here, $(x_1, \cdots, x_n, y_1, \cdots, y_n, z)$ are a chosen coordinate system for $\Bbb R^{2n+1}$.
\end{enumerate}
\end{example}

Both of these examples are special cases of \emph{phase spaces} appearing in physics. Let $Q$ be an $n$--dimensional smooth manifold, called the \emph{configuration space}. Let $\pi : T^* Q \to Q$ denote the cotangent bundle projection, and $(q_1, \cdots, q_n)$ be a choice of a local coordinate system on $Q$. The associated holonomic coordinate system on $T^*Q $ is $(q_1, \cdots, q_n, p^1, \cdots, p^n)$ where $p^i = \partial/\partial q_i$. In classical mechanics, these are known as the generalized position-momentum coordinates.

\begin{example}\label{eg-phase}\leavevmode
\begin{enumerate}
	\item The \emph{symplectic phase space} is the symplectic manifold $(T^* Q, \omega)$ where $\omega = -d\theta$ is exterior derivative of the \emph{tautological $1$--form} $\theta$ on $T^*Q$, defined as follows. For any $(q, \zeta) \in T^* Q$, $v \in T_{(q, \zeta)}T^*Q$, we define 
$$\theta(v) = \zeta(\pi_*(v)).$$
In the generalized position-momentum coordinate system, we obtain:
\begin{gather*}\theta = \sum_{i = 1}^n p^i dq_i, \;
\omega = \sum_{i = 1}^n dq_i \wedge dp^i.
\end{gather*}
	\item The \emph{contact phase space} is the contact manifold $(\Bbb R \times T^* Q, \xi)$ where $\xi = \ker(\alpha)$, and $\alpha$ is a $1$--form on $\Bbb R \times T^* Q$ defined as follows. Let $z$ denote the coordinate for the first factor $\Bbb R$. Then we define,
$$\alpha = dz - \theta,$$
where $\theta$ is the tautological $1$--form as defined above. In the generalized position-momentum coordinate system, we obtain:
$$\alpha = dz - \sum_{i = 1}^n p^i dq_i$$
\end{enumerate}
\end{example}

\begin{remark}Example \ref{eg-stansympcont} is the special case of Example \ref{eg-phase} with $Q = \Bbb R^n$.\end{remark}

Recall that, given a symplectic vector space $(V, \omega)$ and a subspace $W \subset V$, the \emph{symplectic complement} is defined by:
$$W^{\perp \omega} := \{v \in V : \omega(v, w) = 0, \ \forall w \in W\}.$$
The following terminologies are standard:
\begin{enumerate}
	\item $W$ is called \emph{isotropic} if $W \subset W^{\perp \omega}$,
	\item $W$ is called \emph{co-isotropic} if $W^{\perp \omega} \subset W$,
	\item $W$ is called \emph{Lagrangian} if it is isotropic and co-isotropic, i.e., $W = W^{\perp \omega}$,
	\item $W$ is called \emph{symplectic} if $(W, \omega|_W)$ is a symplectic vector space.
\end{enumerate}
Observe that we have a short exact sequence:
$$0 \to W^{\perp \omega} \to V \stackrel{\phi}{\to} W^* \to 0,$$
where $\phi(v) = \omega(v, -)|_W$. The kernel of $\phi$ is exactly $W^{\perp \omega}$ by definition. From this, we deduce that if $W$ is a isotropic subspace, $\dim W \leq \dim V/2$. Here, the equality holds if and only if $W$ is a Lagrangian subspace. 

We record a lemma from linear algebra for future use.

\begin{lemma}\label{lem-lagcomp}Let $(V, \omega)$ be a symplectic vector space of rank $n$, and $L \subset V$ be a Lagrangian subspace. Then there exists a Lagrangian subspace $S \subset V$ such that $S \oplus L = V$. Moreover, $S$ is canonically isomorphic to $L^*$.
\end{lemma}

\begin{proof}
For any such subspace $S$, we have a map $\phi : S \to L^*$, $\phi(v) = \omega(v, -)|_L$. If this is not an isomorphism, then $\phi$ must have kernel as $\dim S = \dim L^*$. Thus, there exists $v \in S$ such that $\omega(v, w) = 0$ for all $w \in L$. But, as $S \subset V$ is Lagrangian, $\omega(v, w) = 0$ for all $w \in S$ as well. Therefore, $\omega(v, w) = 0$ for all $w \in S \oplus W = V$, contradicting nondegeneracy of $\omega$. Thus, $\phi$ is necessarily an isomorphism, and so $S \cong L^*$.

We now show existence of such a subspace. Choose an almost complex structure $J$ on $V$ \emph{compatible} with $\omega$, i.e., an isomorphism $J : V \to V$ such that,
\begin{enumerate}
\item $J^2 = -\mathrm{id}$, 
\item $\omega(Jv, Jw) = \omega(v, w)$ for all $v, w \in V$, and 
\item $\omega(v, Jv) \geq 0$ for all $v \in V$, with equality iff $v = 0$. 
\end{enumerate}
See \cite[Section 2.5]{msbook} for the proof of existence of such an almost complex structure. We claim $S = JL$ is the desired subspace. Indeed, for all $v, w \in S$, $\omega(v, w) = \omega(Jv, Jw) = 0$ as $Jv, Jw \in L$ and $L$ is Lagrangian. Moreover, if $v \in S \cap L$, then $v \in L$ as well as $v = Jw$ for some $w \in L$. Hence, $\omega(w, Jw) = \omega(v, w) = 0$. This forces $w = 0$ and thus $v = 0$, by compatibility of $J$ and $\omega$. Therefore, $S \cap L = \{0\}$. As $S, L$ are both Lagrangian, they are half-dimensional subspaces of $V$. Since they intersect trivially, $S \oplus L$ is forced to be $V$. This concludes the proof of the lemma.
\end{proof}

\begin{definition}\label{def-lagleg}\leavevmode
\begin{enumerate}
	\item Let $(X, \omega)$ be a symplectic manifold. An immersion $f : L \to (X, \omega)$ is said to be \emph{isotropic} (resp. \emph{Lagrangian}) if for every $p \in L$ with $q := f(p) \in X$, $f_*(T_p L) \subset T_q X$ is an isotropic (resp. Lagrangian) subspace.
	\item Let $(Y, \xi)$ be a contact manifold. An immersion $f : \Lambda \to (Y, \xi)$ is said to be \emph{isotropic} (resp. \emph{Legendrian}) if for every $p \in \Lambda$ with $q := f(p) \in Y$, $f_*(T_p \Lambda) \subset \xi_q$ and $f_*(T_p \Lambda)$ is an isotropic (resp. Lagrangian) subspace of $(\xi_q, \mathrm{CS}(\xi)_q)$. 
\end{enumerate}
We say a submanifold $L \subset X$ (resp. $\Lambda \subset Y)$ is a Lagrangian (resp. Legendrian) submanifold if the inclusion map is a Lagrangian (resp. Legendrian) immersion.
\end{definition}

\begin{example}Consider the trajectories of motion $\gamma : I \to (Y = \Bbb R^2 \times S^1, \xi)$ of a unicycle in Example \ref{eg-unic}. Since $\gamma' \in \xi$, $\gamma$ is isotropic. Moreover, as $\dim Y = 3 = 2 \dim I + 1$, $\gamma$ must be a Legendrian immersion.
\end{example}

\begin{example}\label{eg-phaselag}
In Example \ref{eg-phase}, let $f : Q \to \Bbb R$ be a smooth function. Consider the embedding, 
$$df : Q \to (T^* Q, \omega),$$
defined by $df(x) := (x, df_x)$. This is an isotropic immersion, since:
$$(df)^*\omega = \sum_{i = 1}^n dq_i \wedge d \left ( \frac{\partial f}{\partial q_i} \right ) = \sum_{i = 1}^n \sum_{j = 1}^n \frac{\partial^2 f}{\partial q_i \partial q_j} dq_i \wedge dq_j = 0.$$
As $\dim Q = (\dim T^*Q)/2$, $df$ is in fact a Lagrangian immersion.
Likewise, the embedding 
$$j^1 f : Q \to (\Bbb R \times T^* Q, \xi),$$
defined by $j^1 f(x) := (f(x), x, df_x)$ is an example of a Legendrian immersion. Indeed,
$$(df)^*\alpha = df - \sum_{i = 1}^n \frac{\partial f}{\partial q_i} dq_i = 0.$$
\end{example}

The importance of Example \ref{eg-stansympcont} and Example \ref{eg-phase} lie in the following foundational theorems in symplectic and contact topology. The first theorem states that symplectic and contact manifolds are modelled on Example \ref{eg-stansympcont} in a neighborhood of any point, and the second theorem states that they are modelled on Example \ref{eg-phase} in a neighborhood of any Lagrangian (resp. Legendrian) submanifold. This stands in stark contrast to, e.g., Riemannian geometry: certainly not all Riemannian manifolds are isometric to the flat Euclidean space near a point; indeed, the local obstruction is given by the Riemann curvature tensor.

\begin{theorem}[Darboux's Theorem]\label{thm-darboux}
Let $(X^{2n}, \omega)$ be a symplectic manifold and $(Y^{2n+1}, \xi)$ be a contact manifold. Let $x \in X$ and $y \in Y$ be a point chosen from each.
\begin{enumerate}[label=(\arabic*), font=\normalfont]
\item There exists a neighborhood $U \subset X$ of $x$ and a map $\varphi : (U, \omega) \to (\Bbb R^{2n}, \omega_{\rm{std}})$ such that $\varphi$ is a symplectomorphism onto image, and $\varphi(x) = 0$.
\item There exists a neighborhood $V \subset Y$ of $y$ and a map  $\psi : (V, \xi) \to (\Bbb R^{2n+1}, \xi_{\rm{std}})$ such that $\psi$ is a contactomorphism onto image, and $\psi(y) = 0$.
\end{enumerate}
\end{theorem}

\begin{theorem}[Weinstein's Tubular Neighborhood Theorem]\label{thm-weinstein}
Let $(X^{2n}, \omega)$ be a symplectic manifold and $(Y^{2n+1}, \xi)$ be a contact manifold. Let $L \subset X$ be a Lagrangian submanifold and $\Lambda \subset Y$ be a Legendrian submanifold. Let $\omega_0, \xi_0$ denote the symplectic and contact structures on the phase spaces constructed in Example \ref{eg-phase}.
\begin{enumerate}[label=(\arabic*), font=\normalfont]
\item There exists a neighborhood $U \subset X$ of $L$ and a map $\varphi : (U, \omega) \to (T^*L, \omega_0)$ such that $\varphi$ is a symplectomorphism onto image, and $\varphi(L) = 0_L \subset T^*L$ is the zero section.
\item There exists a neighborhood $V \subset Y$ of $\Lambda$ and a map $\psi : (V, \xi) \to (\Bbb R \times T^*L, \xi_0)$ such that $\psi$ is a contactomorphism onto image, and $\psi(L) = \{0\} \times 0_L \subset T^*L$.
\end{enumerate}
\end{theorem}

The cornerstone result that is going to be used to prove these theorems is stated in the following lemma, colloquially known as \emph{Moser's trick}. The lemma enables us to extend certain homotopies of symplectic and contact structures to ambient isotopies of the manifold. Let $X, Y$ continue to denote manifolds of even and odd dimensions, respectively.

\begin{lemma}[Moser's trick]\label{lem-moser}\leavevmode
\begin{enumerate}[label=(\arabic*), font=\normalfont, leftmargin=*]
\item Let $K \subset X$ be a compact subset, and $U \subset X$ be an open neighborhood of $K$. Let 
$$\{\omega_t = \omega_0 + d\alpha_t : 0 \leq t \leq 1\}$$
be a homotopy of symplectic forms on $U$ such that $\omega_t = \omega_0$ on $TX|_K$. Then there exists an isotopy $\varphi_t : U \to X$ such that $\varphi_t(x) = x$, $(d\varphi_t)_x= \mathrm{id}$ for all $x \in K$, and $\varphi_t^*\omega_t = \omega_0$.
\item Let $K \subset Y$ be a compact subset, and $U \subset Y$ be an open neighborhood of $K$. Let 
$$\{\xi_t : 0 \leq t \leq 1\}$$ 
be a homotopy of contact structures on $U$ such that $\xi_t = \xi_0$ on $TY|_K$. Then there exists an isotopy $\psi_t : U \to Y$ such that $\psi_t(x) = x$, $(d\psi_t)_x = \mathrm{id}$ for all $x \in K$, and $\psi_t^*\xi_t = \xi_0$.
\end{enumerate}
\end{lemma}

\begin{proof}\leavevmode
\begin{enumerate}[leftmargin=*]
\item We wish to find an isotopy of diffeomorphisms $\{\varphi_t : U \to X : 0 \leq t \leq 1\}$, such that $\varphi_t^* \omega_t = \omega_0$. By differentiating both sides with respect to $t$, one obtains,
$$\mathcal{L}_{V_t} \omega_t = \frac{d\omega_t}{dt},$$
where $V_t$ is the $1$--parameter family of vector fields generating $\varphi_t$. By Cartan's magic formula, this equation is equivalent to,
$$d i_{V_t} \omega_t = \frac{d\omega_t}{dt}.$$
Since $d\omega_t/dt = d(\omega_0 + td\alpha)/dt = d\alpha$, it suffices to solve for $i_{V_t}\omega_t = \alpha$. But $\omega_t$ is a nondegenerate $2$--form on $U$, hence there exists such a solution $V_t$ for every $0 \leq t \leq 1$. Now, by solving the nonautonomous initial value problem
$$\frac{d\varphi_t}{dt} = V_t \circ \varphi_t, \; \varphi_0 = \mathrm{id},$$
we obtain the desired isotopy. Notice that $V_t = 0$ on $K$ by construction, as $\omega_t = \omega_0$ along $K$ for all $0 \leq t \leq 1$. In summary, we have $\varphi_t^*\omega_t = \omega_0$ and for all $x \in K$, $\varphi_t(x) = x$ and $(d\varphi_t)_x = \mathrm{id}$.   

\item We may choose a homotopy of $1$--forms $\{\alpha_t : 0 \leq t \leq 1\}$ such that $\xi_t = \ker \alpha_t$. As before, we wish to find an isotopy of diffeomorphisms $\{\psi_t : U \to Y : 0 \leq t \leq 1\}$, such that $\psi_t^* \alpha_t = e^{f_t} \alpha_0$. Note that in this case we allow a conformal scale by a time-dependent family of positive functions in virtue of Remark \ref{rmk-sycomor}. By differentiating both sides with respect to $t$, one obtains,
$$\mathcal{L}_{V_t} \alpha_t = \frac{d\alpha_t}{dt} + f_t\alpha_t,$$
where $V_t$ is the $1$--parameter family of vector fields generating $\psi_t$. By Cartan's magic formula, this equation is equivalent to,
$$i_{V_t} d\alpha_t + d i_{V_t} \alpha_t = \frac{d\alpha_t}{dt} + f_t \alpha_t.$$
Let us choose $V_t$ such that $i_{V_t} d\alpha_t = (d\alpha_t/dt)|_{\xi_t}$. This is possible, as $d\alpha_t$ defines a fiberwise nondegenerate $2$--form on $\xi_t$. Thus, $V_t \in \xi_t$ lies along the contact distribution, hence $i_{V_t}\alpha_t = 0$. Consequently, we have
$$\left (\mathcal{L}_{V_t} \alpha_t - \frac{d\alpha_t}{dt}\right)\bigg{\vert}_{\xi_t} = \left (i_{V_t} d\alpha_t - \frac{d\alpha_t}{dt}\right )\bigg{\vert}_{\xi_t} = 0.$$
Hence, there exists $f_t$ such that,
$$\mathcal{L}_{V_t} \alpha_t - \frac{d\alpha_t}{dt} = f_t \alpha_t.$$
Once again, by solving the nonautonomous initial value problem,
$$\frac{d\psi_t}{dt} = V_t \circ \psi_t, \ \psi_0 = \mathrm{id},$$
we obtain the desired isotopy. We have $V_t = 0$ on $K$ by construction, as $\xi_t = \xi_0$ along $K$ for all $0 \leq t \leq 1$. In summary, we have $\psi_t^*\xi_t = \xi_0$, and for all $x \in K$, $\psi_t(x) = x$ and $(d\psi_t)_x = \mathrm{id}$.\qedhere
\end{enumerate}
\end{proof}

Finally, we give a proof of Darboux's Theorem \ref{thm-darboux} and Weinstein's Theorem \ref{thm-weinstein} using Moser's trick.

\begin{proof}[Proof of Theorem \ref{thm-darboux}]\leavevmode
\begin{enumerate}[leftmargin=*]
\item Let $(X^{2n}, \omega)$ be a symplectic manifold and $x \in X$ be a point. The symplectic vector space $(T_x X, \omega_x)$ admits a linear symplectic basis, which we may extend to coordinates $(q_1, p^1, \cdots, q_n, p^n)$ on a chart $U \cong \Bbb R^{2n}$ around $x$. Let $\omega_{\rm{std}}$ denote the standard symplectic form on $U$, as in Example \ref{eg-stansympcont}, with respect to these coordinates. 
Consider the homotopy of closed $2$--forms,
$$\omega_t := t\omega + (1- t)\omega_{\rm{std}}, \ 0 \leq t \leq 1.$$
As $\omega_{\rm{std}}, \omega$ agree at $x$, we have $\omega_t = \omega_{\rm{std}}$ at $x$ for all $0 \leq t \leq 1$. Since nondegeneracy of $2$--forms is an open condition, we may shrink $U$ further to assume $\omega_t$ is symplectic on $U$ for all $t$. Thus, we may apply Lemma \ref{lem-moser} with $K = \{x\}$ to furnish the desired symplectomorphism.
\item Let $(Y^{2n+1}, \xi)$ be a contact manifold with $\xi = \ker \alpha$ and $y \in X$ be a point. Consider a coordinate system $(q_1, p^1, \cdots, q_n, p^n, z)$ on a chart $U \cong \Bbb R^{2n+1}$ around $y$ such that $\xi_y = \{t = 0\}$, and $(d\alpha)_y = dq_1 \wedge dp^1 + \cdots + dq_n \wedge dp^n$. Let $\xi_{\rm{std}} = \ker \alpha_{\rm{std}}$ denote the standard symplectic form on $U$, as in Example \ref{eg-stansympcont}, with respect to these coordinates. We may consider the homotopy of $1$--forms,
$$\alpha_t := t\alpha + (1- t)\alpha_{\rm{std}}, \ 0 \leq t \leq 1.$$
Observe, $\alpha, d\alpha$ agrees with $\alpha_{\rm{std}}, d\alpha_{\rm{std}}$, respectiely, at $y$. Thus, we have $\alpha_t = \alpha_{\rm{std}}$ and $d\alpha_t = d\alpha_{\rm{std}}$ at $y$, for all $0 \leq t \leq 1$. Since complete nonintegrability of $1$--forms is an open condition, we may shrink $U$ further to assume $\alpha_t$ is completely nonintegrable on $U$ for all $t$. Thus, we may apply Lemma \ref{lem-moser} with $K = \{y\}$ to furnish the desired contactomorphism.\qedhere
\end{enumerate}
\end{proof}

\begin{proof}[Proof of Theorem \ref{thm-weinstein}]\leavevmode
\begin{enumerate}[leftmargin=*]
\item Let $L \subset (X^{2n}, \omega)$ be a Lagrangian submanifold. Thus, we have a short exact sequence 
$$0 \to TL \to TX \to T^*L \to 0,$$
Therefore, $T^*L \cong TX/TL \cong \nu_L$ is the normal bundle of $L$ in $X$. By the tubular neighborhood theorem, there is a diffeomorphism onto image $f : T^*L \cong \nu_L \to X$ such that $f(0_L) = L$. The differential of $f$ along the zero section $0_L$ maps $T(T^*L)|0_L \cong TL \oplus T^*L$ isomorphically to $TX$, hence $f$ is a symplectomorphism along $L$. Let $\omega_1 := f^*\omega$. Since $\omega_1, \omega_0$ agree along $0_L$, and the embedding $0_L \to T^*L$ is an isomorphism in de Rham cohomology, these $2$--forms must be cohomologous. Thus, there is a $1$--form $\alpha$ on $T^*L$ such that $\omega_1 = \omega_0 + d\alpha$. We define,
$$\omega_t := \omega_0 + d(t\alpha).$$
Evidently, $\omega_t$ agrees with the symplectic form $\omega_0$ along the tangent spaces to $0_L$, for all $t$. Therefore, after further shrinking $U$, we may assyme $\omega_t$ is symplectic on $U$ for all $t$. Thus, we may apply Lemma \ref{lem-moser} with $K = L$ to furnish the desired symplectomorphism.
\item Let $\Lambda \subset (Y^{2n+1}, \xi)$ be a Legendrian submanifold. In this case, we have a short exact sequence,
$$0 \to T\Lambda \to \xi \to T^*\Lambda \to 0.$$
Therefore, $T^*\Lambda \cong \xi/T\Lambda \subset TY/T\Lambda \cong \nu_f$ is a subbundle of the normal bundle of $L$ in $X$, of co-rank $1$. As $\xi$ is co-oriented, we may choose a complementary trivial bundle to construct a bundle-isomorphism $f : \Bbb R \times T^*\Lambda \to \nu_f$. Once again, we let $\xi_1 = f^*\xi$. By an analogous argument using Moser's trick as in Part $(1)$, and also Part $(2)$ of Proof of Theorem \ref{thm-darboux}, we may furnish the desired contactomorphism between $\xi_1$ and $\xi_0$.\qedhere
\end{enumerate}
\end{proof}

\section{Introduction to Legendrian knots in three dimensions}\label{sec-legknots}

Here we review the local theory of Legendrian knots in contact $3$--manifolds, and refer the reader to \cite[Chapter 3]{geigbook} for further discussions.

\begin{definition}Let $(Y, \xi)$ be a contact $3$--manifold. A \emph{Legendrian knot} in $Y$ is a Legendrian embedding $\gamma : S^1 \to (Y, \xi)$. By a small abuse of notation, we shall also call the submanifold $\Lambda = \gamma(S^1)$, a Legendrian knot in $Y$. \end{definition}

The local theory of Legendrian knots is the same as that of proper Legendrian arcs in $(\Bbb R^3, \xi_{\rm{std}})$, by Darboux's theorem \ref{thm-darboux}. Thus, we begin by studying Legendrian knots with $(Y, \xi) = (\Bbb R^3, \xi_{\rm{std}})$. Let us choose standard coordinates $(x, y, z)$ on $\Bbb R^3$, so that $\xi = \ker(\alpha)$, where $\alpha = dz - ydx$ is the standard contact form.

\begin{definition}\label{def-legprojknot}Let $\pi : \Bbb R^3 \to \Bbb R^2$, $\pi(x, y, z) := (x, z)$. For a Legendrian knot $\Lambda \subset \Bbb R^3$, we shall call the knot diagram $\pi(\Lambda) \subset \Bbb R^2_{xz}$, the \emph{front projection} or the \emph{Legendrian projection} of the knot.
\end{definition}

\begin{example}\label{eg-cusp}Let $\gamma : \Bbb R \to (\Bbb R^3, \xi)$, 
$$\gamma(t) := \left(\frac{t^2}{2}, \ t,\ \frac{t^3}{3}\right ).$$
Then $\gamma$ is a Legendrian submanifold, as 
$$\gamma^*(\alpha) = d(t^3/3) - t d(t^2/2) = (t^2 - t^2)dt = 0.$$
The front projection of $\gamma$ is $\pi \circ \gamma : \Bbb R \to \Bbb R^2$, $\gamma(t) = (t^2/2, t^3/3)$, which is a \emph{semicubical cusp}.
\end{example}

\begin{proposition}\label{prop-legfr}A \emph{generic} front projection of a Legendrian knot is a knot diagram $K \subset \Bbb R^3_{xz}$ characterized by the following properties:
\begin{enumerate}[label=(\arabic*), font=\normalfont]
\item $K$ does not contain any vertical tangencies, i.e., points $p \in K$ such that, for any parametrization $(x(t), y(t))$ of some arc of $K$ passing through that point $(x(t_0), y(t_0)) = p$, $\dot{x}(t_0) = 0, \dot{z}(t_0) \neq 0$.
\item The only singular points of $K$ (i.e., points where $K$ is not a submanifold of $\Bbb R^2$) are either transverse double crossings, or semi-cubical cusps.
\item At each crossing the slope of the overcrossing arc is less than that of the undercrossing arc.
\end{enumerate}
\end{proposition}

\begin{proof}First, we begin by showing that a front projection of a generic Legendrian knot is a knot diagram satisfying the properties listed above. Let $\gamma(t) = (x(t), y(t), z(t))$ be a Legendrian knot, with image $\Lambda$. Let $K = \pi(\Lambda)$ be the front projection.

\begin{enumerate}[label=(\arabic*), font=\normalfont, leftmargin=*]
\item By definition, $\gamma' \in \xi$, which is to say $\gamma$ must satisfy the differential equation:
$$\dot{z}(t) = y(t) \dot{x}(t).$$
Thus, for any $t_0 \in S^1$, $\dot{x}(t_0) = 0$ forces $\dot{z}(t_0) = 0$. Hence, there are no vertical tangencies. 

\item For a generic Legendrian knot $\gamma$, the coordinate function $x : S^1 \to \Bbb R$ is a Morse function. This implies $\dot{x} = 0$ only at an isolated set of nondegenerate critical points of $S^1$. Thus, the front projection $\pi \circ \gamma : \Bbb R \to \Bbb R^2$ has an isolated set of singular points as well, away from which it is an immersion and thus $K = \pi(\Lambda)$ is an immersed curve at such points. Again, by genericity of the Legendrian knot, we may assume that for all points $p \in K$ over the preimage of which $\pi \circ \gamma$ is an immersion, the cardinality of $(\pi \circ \gamma)^{-1}(p)$ is at most $2$. This ensures that all the crossings are double crossings.

Now assume by translation and a reparametrization that a singular point of $\pi \circ \gamma$ is attained at 
$$(x(0), y(0), z(0)) = (0, 0, 0).$$
At this point, $\pi \circ \gamma$ fails to be an immersion, hence $\dot{x}(0) = \dot{z}(0) = 0$. As $\gamma$ is a smooth embedding, we must necessarily have $\dot{y}(0) \neq 0$. We may choose coordinates around $0 \in S^1$ so that $y(t) = t$ on a small arc around $0$. As $0 \in S^1$ is a Morse singularity of $x$, we may choose coordinates around $0 \in S^1$ so that $x(t) = t^2/2$. We can achieve this simultaneous change of coordinates by considering the jointly defined map $(x, y) : S^1 \to \Bbb R^2$ and changing coordinates on the target $\Bbb R^2$ near $(0, 0)$. Finally, this uniquely specifies the $z$-coordinate  near $0 \in S^1$ as,
$$z(t) = \int_{0}^t y(t) \dot{x}(t) dt = \int_{0}^t t^2 dt = t^3/3.$$
This gives us the required cuspidal normal form near $0$.

\item Note that standard orientation $\{x, y, z\}$ on $\Bbb R^3$ is equivalent to that of $\{x, z, -y\}$. Thus, at a crossing in the knot diagram $K \subset \Bbb R^2_{xz}$, the overcrossing has a \emph{smaller} $y$--value than the undercrossing. Therefore,
$$\left ( \frac{dz}{dx} \right )_{\rm{over}} = y_{\rm{over}} \leq y_{\rm{under}} = \left ( \frac{dz}{dx} \right )_{\rm{under}}$$
This proves that the slope of the overcrossing is less than that of the undercrossing.
\end{enumerate}

Conversely, given a knot diagram drawn on the $xz$-plane satisfying the three properties listed above, we may choose a Legendrian lift by defining $y(t) = \dot{z}(t)/\dot{x}(t)$ for all $t \in S^1$, whenever the right hand side of the expression is well-defined. At the points where it is not defined, we may use Example \ref{eg-cusp} as a local model for a lift. Since Condition $(2)$ ensures that the knot diagram does not contain any self-tangencies, this lift is in fact an embedding.\end{proof}

\begin{theorem}[$C^0$--dense $h$--principle for Legendrian knots]\label{thm-hprinclegknots}
Let $\kappa \subset \Bbb R^3$ be a smooth knot, and $\varepsilon > 0$. There exists a Legendrian knot $\Lambda \subset (\Bbb R^3, \xi_{\rm{std}})$ such that $\mathrm{dist}_{C^0}(\kappa, \Lambda) < \varepsilon$. Moreover, given finitely many marked points on $\kappa$, we can demand $\Lambda$ to pass through these points.
\end{theorem}

\begin{proof}
Let $N > 1$ be a sufficiently large natural number. Choose points $\{p_0, \cdots, p_N\}$ on $\kappa \subset \Bbb R^3$ with spacing (in the Euclidean norm) less than $1/N$. Let $s_i$ be the value of the $y$--coordinate of $\kappa$ at $p_i$, for all $0 \leq i \leq N$. We will construct a front diagram $K \subset \Bbb R^2_{xz}$ which interpolates between the points $\pi(p_0), \cdots, \pi(p_N)$, such that moreover the slope of $K$ interpolates between the values $s_0, \cdots, s_N$. The basic local model we shall use for this is the following:

\begin{observation}[Interpolation zig-zags]\label{obs-intpzigzag} For any $\varepsilon > 0$, an embedded arc $\sigma : [0, 1] \to \Bbb R^2$, and a slope-value $s > 0$, there exists a topologically embedded arc $\gamma = (\gamma_1, \gamma_2) : [0, 1] \to \Bbb R^2$ with no vertical tangencies and only semicubical cusp singularities such that $\gamma(0) = \sigma(0)$, $\gamma(1) = \sigma(1)$, $\mathrm{dist}_{C^0}(\gamma, \sigma) < \varepsilon$ and $\|\gamma_2'(t)/\gamma_1'(t) - s\| < \varepsilon$ for all $t \in [0, 1]$.
\end{observation}

\begin{figure}[h]
\centering
\includegraphics[scale=0.6]{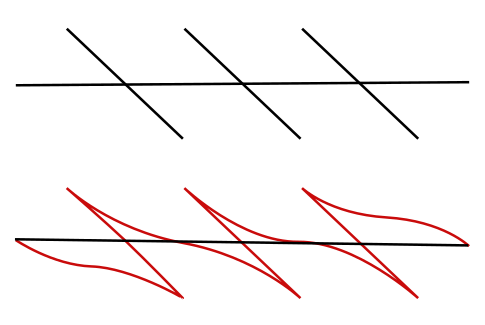}
\caption{The arc $\sigma$ with the slope-values drawn as a line-field along $\sigma$ (top). The interpolation zig-zags are constructed by using the line-field as `guides' (bottom).}
\label{fig-intpzigzag}
\end{figure}

By considering interpolating zig-zags for the arcs of $\pi(\kappa)$ joining $\pi(p_i)$ and $\pi(p_{i+1})$ for $0 \leq i \leq N-1$, and composing these in sequence, we obtain a front diagram. Lifting it to a Legendrian knot $\Lambda$ by Proposition \ref{prop-legfr} gives the desired conclusion. The statement of the theorem also holds true relative to finitely many marked points on $\kappa$ by including these to be in the set of points $\{p_0, \cdots, p_N\}$ chosen in the beginning of the proof.
\end{proof}

\begin{remark}Theorem \ref{thm-hprinclegknots} has the following physically motivating corollary. We saw in Example \ref{eg-unic} that unicycle trajectories gives rise to Legendrian arcs in $(\Bbb R^2 \times S^1, \xi)$. By covering $\Bbb R^2 \times S^1$ with finitely many Darboux charts and using Theorem \ref{thm-hprinclegknots}, we conclude that every embedded path in $\Bbb R^2 \times S^1$ can be $C^0$--approximated by unicycle trajectories with the same initial and terminal points. Thus, a unicycle can manuevre through any arbitrarily contrieved obstacle course. This is known, in a related model where the unicycle is replaced by a car and the obstacle course is replaced by a saturated parking lot, as the \emph{car-parking problem}.
\end{remark}

\begin{definition}\label{def-legisotopy}Two Legendrian knots $\Lambda_0, \Lambda_1 \subset (Y, \xi)$ in a contact $3$--manifold are \emph{Legendrian isotopic} if there is a smooth isotopy between them through Legendrian knots $\{\Lambda_t \subset (Y, \xi): 0 \leq t \leq 1\}$. We shall call such an isotopy a \emph{Legendrian isotopy}.
\end{definition}

\begin{example}[Legendrian Reidemeister moves]\label{eg-reidmoves}There are Legendrian isotopies in $(\Bbb R^3, \xi_{\mathrm{std}})$ which relates each of front diagrams depicted in the left column of Figure \ref{fig-reid}, to the corresponding front diagram on the right column. 

\begin{figure}[h]
\centering
\includegraphics[scale=0.4]{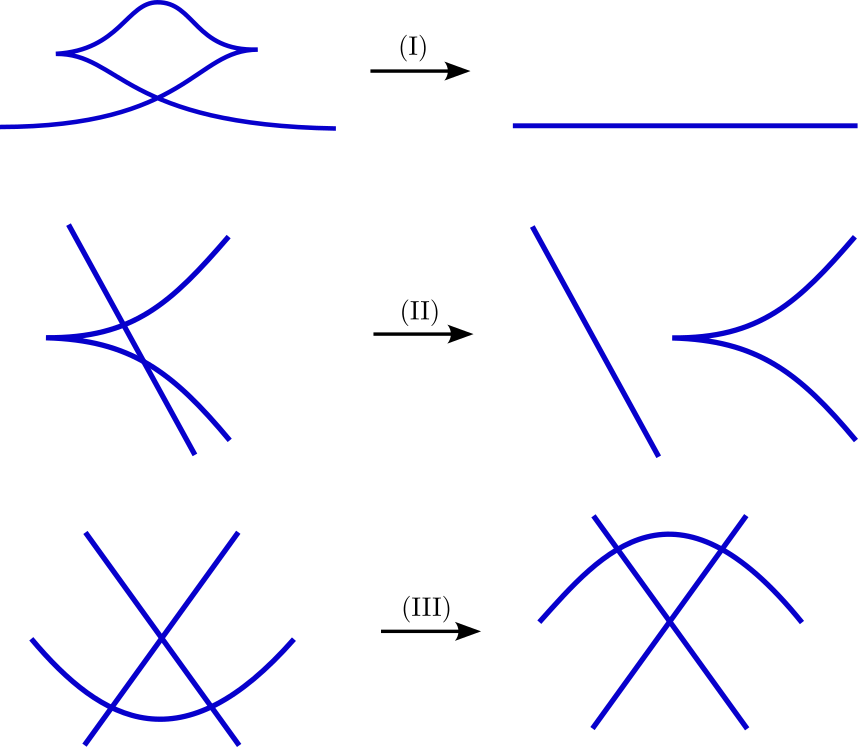}
\caption{Reidemeister moves $\rm{I}, \rm{II}$ and $\rm{III}$.}
\label{fig-reid}
\end{figure}

\noindent
Moves $\rm{II}$ and $\rm{III}$ are straightforward to verify, as they can be achieved by a movie of front diagrams $\{\pi(\Lambda_t) : t \neq t_0\}$ in $\Bbb R^2_{xz}$ which are legitimate except at a single exceptional time $t = t_0$: for move $\rm{II}$, this occurs when the cusp crosses the strand, and for move $\rm{III}$, this occurs at the triple point. The slopes at the crossings are chosen in a way that the Legendrian lift of the movie $\{\Lambda_t : t \neq t_0\}$ define an isotopy which extends to $t = t_0$ as well, i.e. the Legendrian isotopies $\{t < t_0\}$ and $\{t > t_0\}$ glue at $t = t_0$.

To see Move $\rm{I}$, let us consider $(\Bbb R^3, \xi_{\mathrm{std}})$ as a Darboux chart in the phase space of a unicycle $(\Bbb R^2 \times S^1, \xi)$ where $\xi = \cos(\theta)dx - \sin(\theta) dy$, as in Example \ref{eg-unic}. Then the Legendrian front $\Bbb R^2_{xz}$ is simply the configuration space $\Bbb R^2_{xy}$ of unicycle trajectories. It is a straightforward calculation to verify that the vector field,
$$R = \cos(\theta) \frac{\partial}{\partial x} - \sin(\theta) \frac{\partial}{\partial y},$$ 
defines a Reeb vector field (cf. Remark \ref{rmk-reeb}) for $(\Bbb R^2 \times S^1, \xi)$. Along any particular Legendrian curve (or \emph{on-shell}) in $\Bbb R^2 \times S^1$, given by a lift of a unicycle trajectory $(x(t), y(t))$ in $\Bbb R^2_{xz}$,
\begin{align*}
\cos(\theta) &= \frac{\dot{x}}{\|(\dot{x}, \dot{y})\|},\\
\sin(\theta) &= \frac{\dot{y}}{\|(\dot{x}, \dot{y})\|}.
\end{align*}
Thus, along any particular Legendrian curve in $\Bbb R^2 \times S^1$, $R$ restricts to the normalization of the vector field $\dot{x} \partial_x - \dot{y} \partial_y$. In the configuration space $\Bbb R^2_{xy}$, this corresponds to the unit normal vector field to the unicycle trajectory. The unicycle trajectory need not be smooth, but generically the singularities are cuspidal, hence the unit normal vector field admits a canonical extension to these singular points as well. 

Thus, flowing a particular front diagram or unicycle trajectory $C \subset \Bbb R^2_{xy}$ along the Reeb flow then produces the family of parallel curves $\{C + t\mathbf{n} : t \in \Bbb R\}$ where $\mathbf{n}$ denotes the unit normal to $C$. We demonstrate Move $\rm{I}$ as a Reeb flow in Figure \ref{fig-reidflow}

\begin{figure}[h]
\centering
\includegraphics[scale=0.4]{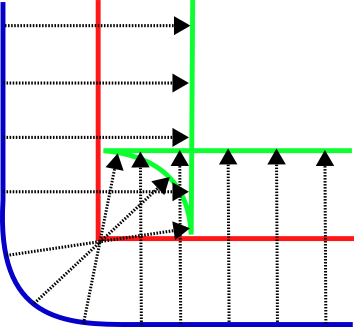}
\caption{Legendrian Reidemeister move $\rm{I}$ as a Reeb flow from the blue curve to the green curve.}
\label{fig-reidflow}
\end{figure}
\end{example}

\section{Holonomic approximation and convex integration}\label{sec-holapp}

In this section, we review fundamentals of the theory of $h$--principles developed by Gromov \cite{grobook} and Eliashberg-Mishachev \cite{embook}. As the nature of the material presented is quite general and abstract, we elucidate the ideas by providing many examples, primarily from contact topology. 

\subsection{Preliminaries on jet bundles}

Let us begin by introducing a convenient notation introduced by Gromov \cite[Section 1.4.1]{grobook}, which we shall use subsequently throughout.

\begin{notation}[Germinal neighborhoods]\label{not-op}Let $M$ be a smooth manifold. Given a subset $K \subset M$, an unspecified open neighborhood of $K$ will be denoted as $\mathrm{Op}_M(K)$, or simply $\mathrm{Op}(K)$ when the ambient manifold $M$ is understood. Given another manifold $N$, germs of mappings from $M$ to $N$ defined on some open neighborhood of $K$ will be denoted as $f : \mathrm{Op}(K) \to N$. Thus, $\mathrm{Op}(K)$ shall stand as a placeholder for an unspecified open neighborhood of $K$ in $M$ which might shrink in the size during the course of a proof or discussion.\end{notation}

\begin{definition}[$r$--jet equivalence]\label{def-rjeteq}Let $E, B$ be smooth manifolds, and let $\pi : E \to B$ be a smooth fiber bundle. Let $p \in B$ be a point, and $r \geq 1$ be a natural number. Two germs of sections $s_1, s_2 : \mathrm{Op}(\{p\}) \to E$ are said to be \emph{$r$--jet equivalent at $p$} if all the derivatives of $s_1$ and $s_2$ of order upto $r$ agree at $p$. More precisely, let $(x^1, \cdots, x^n)$ be a local coordinate system around $p$ on $B$. Then sections $s_1, s_2$ are $r$--jet equivalent at $p$ if for all multi-sets $I = \{i_1, \cdots, i_k\} \subset \{1, \cdots, n\}$ with $|I| \leq r$, 
$$\frac{\partial^{|I|} s_1(p)}{\partial x^{i_1} \cdots \partial x^{i_k}} = \frac{\partial^{|I|} s_2(p)}{\partial x^{i_1} \cdots \partial x^{i_k}}.$$
This notion is independent of the choice of local coordinates by the chain-rule. We shall denote the $r$--jet equivalence class of a section $s : \mathrm{Op}(\{p\}) \to E$ as $j^r_p s$. \end{definition}

Let us denote the set of $r$--jet equivalence classes of sections at an arbitrary point in $B$ as:
$$J^r E := \{j^r_p s :  p \in B, \ s : \mathrm{Op}(\{p\}) \to E, \ \pi \circ s = \mathrm{id}\}$$

\begin{example}[Trivial Euclidean bundles]\label{eg-trivjet}
Consider the case $B = \Bbb R^n, E = \Bbb R^n \times \Bbb R^k$ and $\pi : E \to B$ given by projection to the first $n$ coordinates. Let $\mathcal{P}_r(n, k)$ denote the space of $k$--tuples of multivariate polynomials with $n$ variables, each polynomial of total degree at most $r$. We have a set-theoretic bijection,
\begin{gather*}\varphi : J^r E \to B \times \mathcal{P}_r(n, k),\\
\varphi(j^r_p s) = (p, (T_r s)_p).\end{gather*}
Here, $(T_r s)_p$ denotes the $r$--th order Taylor expansion of $s$ around $p$. We equip $J^r E$ with the topology given by pulling back the Euclidean topology on 
$$B \times \mathcal{P}_r(n, k) \cong \Bbb R^n \times \Bbb R^{k \binom{n+r}{r}} = \Bbb R^{n+k\binom{n+r}{r}}.$$ 
\end{example}

As any smooth fiber bundle of smooth manifolds $(E, B, \pi)$ is locally isomorphic to $(\Bbb R^n \times \Bbb R^k, \Bbb R^n, \pi)$ by choice of adapted coordinate charts, we may use the Example \ref{eg-trivjet} to give a basis for a topology on $J^r E$ in general. This topology is evidently locally Euclidean, and in fact forms a smooth atlas for $J^r E$. 

More precisely, given $\pi$--adapted coordinate chart $U \subset E$ with coordinate system $(x^1, \cdots, x^n, y^1, \cdots, y^k)$, we define the coordinate chart $J^r U \subset J^r E$ with coordinate system $(x^1, \cdots, x^n, y^1, \cdots, y^k, z^{j, I})$, where for any $j \in \{1, \cdots, k\}$ and multiset $I = \{i_1, \cdots, i_\ell\}\subset \{1, \cdots, n\}$ of size $|I| \leq r$, the coordinates 
$$z^{j, I} := \frac{\partial^{|I|}}{\partial x^{i_1}\cdots\partial x^{i_\ell}} \circ y^j$$ 
indicate the mixed partial derivatives of order $I$, of the $j$--th coordinate. Note that if we denote $\pi^{(r)} : J^r E \to B$ to be the natural projection $\pi^{(r)}(j^r_p s) := p$, then the coordinate chart $\{J^r U, (x^1, \cdots, x^n, y^1, \cdots, y^k, z^{j, I})\}$ on $E$ is $\pi^{(r)}$--adapted.

We consolidate this discussion in the following definition:

\begin{definition}[Jet bundle]\label{def-jetbdl}Let $E, B$ be smooth manifolds, and let $\pi : E \to B$ be a smooth fiber bundle. We define the associated \emph{$r$--jet bundle} $(J^r E, B, \pi^{(r)})$ by:
\begin{gather*}J^r E := \{j^r_p s :  p \in B, \ s : \mathrm{Op}(\{p\}) \to E, \ \pi \circ s = \mathrm{id}\},\\
\pi^{(r)}(j^r_p s) := p,
\end{gather*}
where $J^r E$ is equipped with the smooth atlas described above.
\end{definition}

Observe that for any section $s$ of $(E, B, \pi)$, one obtains a natural section $j^r s$ of the jet bundle $(J^r E, B, \pi^{(r)})$, defined by $(j^r s)(p) := j^r_p s$. We call $j^r s$ the \emph{$r$--jet prolongation} of $s$. Not all sections of $J^r E$ are $r$--jet prolongations of sections of $E$. Thus, to distinguish these, we call a general section of $J^r E$ a \emph{formal section} and an $r$--jet prolongation of a section of $E$ a \emph{holonomic section}. Nevertheless, there is a natural section of $E$ underlying every formal section of $J^r E$. Consider the map, 
$$\pi_{\rm{front}} : J^r E \to E, \; \pi_{\rm{front}}(j^r_p s) := s(p).$$
Then, for any formal section $f : B \to J^r E$ of $(J^r E, B, \pi^{(r)})$, we define the \emph{base} $\mathrm{bs}(f)$, which is a section of $(E, B, \pi)$ defined by $\mathrm{bs}(f) := \pi_{\rm{front}} \circ f$.

\begin{example}[Contact phase space revisited]\label{eg-1jet}Let $Q^n$ be a smooth manifold, and $\pi : Q \times \Bbb R \to Q$ be the trivial line bundle over $Q$. Note that sections of $\pi$ are of the form $s : Q \to Q \times \Bbb R$, $s(q) = (q, f(q))$ for smooth functions $f : Q \to \Bbb R$. Therefore, under the adapted coordinate charts described above,
$$j^1_q s = \left (q, f(q) + \frac{\partial f(q)}{\partial x^1} z^1 + \cdots + \frac{\partial f(q)}{\partial x^n} z^n \right ).$$
Note that the coordinates $z^i = \partial/\partial x^i$ on the associated $1$--jet bundle transform as pullback of the $1$--forms $dx^i$ under $\pi^{(1)}$, in virtue of the chain rule. Therefore, we may equivalently write,
$$j^1_q s = (q, f(q), df_q),$$
where $df_q$ is the exact $1$--form $df$ restricted to $T_q Q$. This gives an identification between the $1$--jet bundle of $(Q \times \Bbb R, Q, \pi)$ and $\Bbb R \times T^*Q$, which is the contact phase space from Example \ref{eg-phase}. We shall henceforth use the notation $J^1 Q$ to denote this particular jet bundle \end{example}

\begin{remark}[Jet-intepretation of the contact structure]\label{rmk-1jetcont}Example \ref{eg-phase} gives a contact structure $\xi$ on the total space $J^1 Q$. From Example \ref{eg-phaselag}, we obtain that a holonomic section of $(J^1 Q, Q, \pi^{(1)})$ is a Legendrian submanifold with respect to $\xi$. Moreover, for any $p \in J^1 Q$, the contact plane $\xi_p$ is spanned by the collection of tangent spaces at $p$ to all possible germs of holonomic sections of $(J^1 Q, Q, \pi^{(1)})$ passing through $p$.\end{remark}

\begin{example}[Jets of smooth maps]\label{eg-mapjet}
We record a slight generalization of Example \ref{eg-1jet} that we shall frequently encounter. Let $M$, $N$ be smooth manifolds, and $\pi : M \times N \to M$ be the trivial $N$--bundle over $M$. We shall henceforth denote the $1$--jet bundle of $(M \times N, M, \pi)$ as $J^1(M, N)$. In this case, the map
$$\pi_{\rm{front}} : J^1(M, N) \to M \times N,$$
is a fiber bundle with fiber over $(p, q) \in M \times N$ being $\mathrm{Hom}(T_p M, T_q N)$. Thus, a formal section of the $1$--jet bundle $(J^1(M, N), M, \pi^{(1)})$ can be equivalently described as a pair of maps $(F, f)$ where $F : TM \to TN$ is smooth map that is fiberwise linear and $f : M \to N$ is a smooth map such that the following diagram commutes:
\begin{equation*}
\begin{CD}
TM @>F>> TN \\
@VVV @VVV \\
M @>f>> N
\end{CD}
\end{equation*}
where the vertical maps are tangent bundle projections. We will use the terminology ``\emph{$F$ covers $f$}" to indicate the commutativity of this square.
We shall often call formal sections of $J^1(M, N)$, presented as a pair of maps $(F, f) : (TM, M) \to (TN, N)$ where $F$ covers $f$, as \emph{$1$--jets of maps} from $M$ to $N$.\end{example}

\subsection{Preliminaries on differential relations}

Let us begin by recalling Example \ref{eg-unic}, where we had the configuration space of a unicycle with its canonical contact structure
$$Y = \Bbb R^2 \times S^1, \; \xi = \ker(\cos(\theta) dx - \sin(\theta) dy).$$ 
We had described the equation for a trajectory of motion of a unicycle $\gamma : I \to Y$, given by $\gamma' \in \xi$, as a differential relation. In general, we shall define a ``differential relation" on sections of a bundle as any arbitrary constraint on the $r$--jets of sections on that bundle. We give a precise definition below, following Eliashberg-Mishachev \cite[Chapter 5]{embook}.

\begin{definition}[Partial differential relations]\label{def-diffrel}Given a smooth fiber bundle of smooth manifolds $(E, B, \pi)$, a subset $\mathcal{R} \subset J^r E$ of the total space of the associated $r$--jet bundle is called a \emph{partial differential relation of order $r$}, or simply a \emph{differential relation of order $r$}.

A formal section $\sigma : B \to J^r E$ of $(J^r E, B, \pi^{(r)})$ will be called a \emph{formal section} or \emph{formal solution of $\mathcal{R}$} if $\sigma(B) \subset \mathcal{R}$. If a formal section $\sigma$ of $\mathcal{R}$ can be written as an $r$--jet prolongation $\sigma = j^r s$ of some sections $s$ of $E$, we say $\sigma$ is a \emph{holonomic section} or \emph{holonomic solution of $\mathcal{R}$}.
\end{definition}

\begin{example}[Differential relation of Legendrian immersions]\label{eg-legimmrel}Let $\Lambda, Y$ be smooth manifolds and $\xi$ be a contact structure on $Y$. Let $\mathcal{R}_{\rm{Leg}} \subset J^1(\Lambda, Y)$ denote the subset consisting of $1$--jets of (graphs of) Legendrian immersions $f : \mathrm{Op}(\{p\}) \to (Y, \xi)$ near points $p \in \Lambda$. We call $\mathcal{R}_{\rm{Leg}}$ the \emph{differential relation of Legendrian immersions}.\end{example}

It is evident that the $\mathcal{R}_{\rm{Leg}} \subset J^1(\Bbb R, Y)$ is the relevant differential relation defining the trajectories of motions of a unicycle in Example \ref{eg-unic}. Example \ref{eg-legimmrel} will be the central object of study in Section \ref{sec-hLegimm}.

Let $(E, B, \pi)$ be a smooth fiber bundle of smooth manifolds, $(J^r E, B, \pi^{(r)})$ be the associated $r$--jet bundle and $\mathcal{R} \subset J^r E$ be a differential relation. We introduce the following definition for ease of notation.

\begin{definition}[Continuity of families]\label{def-conti}Suppose $\phi : A \times I^d \to B$ is a $d$--parameter homotopy through embeddings. Let us denote $A_t := \phi(A \times \{t\}) \subset B$. A family of formal or holonomic sections $\{f_t : A_t \to \mathcal{R} : t \in I^d\}$ shall be said to be \emph{continuous} if there exists a section $F : \mathrm{Op}(\phi(A \times I^d)) \to J^r E$ such that $F|{A_t} = f_t$. Likewise, a family of germs of formal or holonomic sections $\{f_t : \mathrm{Op}(A_t) \to \mathcal{R}\}$ shall be said to be \emph{continuous} if there exists a section $F : \mathrm{Op}(\phi(A \times I^d)) \to J^r E$ such that $F|\mathrm{Op}(A_t) = f_t$. \end{definition}

In the following, we introduce three important kinds of differential relations:

\begin{definition}[Locally integrable differential relations]\label{def-locintrel}
Suppose we are given,
\begin{enumerate}
\item A continuous map $\phi : I^d \to B$,
\item A continuous family of formal sections $\{f_t : \{\phi(t)\} \to \mathcal{R} : t \in I^d\}$,
\item A continuous family of germs of holonomic sections $\{\widetilde{f}_t : \mathrm{Op}(\{\phi(t)\}) \to \mathcal{R} : t \in \partial I^d\}$ such that $\widetilde{f}_t$ extends $f_t$ for all $t \in \partial I^d$, i.e., $\widetilde{f}_t(\phi(t)) = f_t(\phi(t))$ for all $t \in \partial I^d$.
\end{enumerate}
We say $\mathcal{R}$ is \emph{(parametrically) locally integrable} if there exists a continuous family of holonomic sections $\{\widetilde{f}_t : \mathrm{Op}(\phi(t)) \to E : t \in I^k\}$ agreeing with the family in $(3)$ for $t \in \partial I^k$ and extending the family in $(2)$, i.e., $\widetilde{f}_t(\phi(t)) = f_t(\phi(t))$ for all $t \in \partial I^d$. 
\end{definition}

Note that for $d = 0$, this simply says every jet in $\mathcal{R}$ lying over $p \in B$ is the value of some holonomic section of $\mathcal{R}$ defined near $p$. For the next definition, let $H := D^m \times D^{n-m}$ denote an $m$--handle and let $C := D^m \times \{0\} \subset H$ denote the core.

\begin{definition}[Microflexible differential relations]\label{def-micflexrel}
Suppose we are given,
\begin{enumerate}
\item An isotopy $\phi : H \times I^d \to B$ of handles $H_t := \phi(H \times \{t\})$ with core $C_t := \phi(C \times \{t\})$,
\item A continuous family of germs of holonomic sections $\{f_t : \mathrm{Op}(H_t) \to \mathcal{R} : t \in I^d\}$,
\item A continuous family-homotopy of germs of holonomic sections 
$$\{f_{t, s} : \mathrm{Op}(\partial H_t \cup C_t) \to \mathcal{R} : t \in I^d, s \in I\},$$
such that $f_{t, 0}|\mathrm{Op}(H_t) = f_t$ for all $t \in I^d$ and $f_{t, s}|\mathrm{Op}(\partial H_t)$ is constant in $s$.
\end{enumerate}
We say $\mathcal{R}$ is \emph{(parametrically) microflexible} if there exists $\varepsilon=\varepsilon(\phi, \{f_t\}, \{f_{t, s}\}) > 0$ and a continuous family-homotopy of germs of holonomic sections,
$$\{\widetilde{f}_{t, s} : \mathrm{Op}(H_t) \to \mathcal{R} : t \in I^d, s \in [0, \varepsilon]\},$$
such that $\widetilde{f}_{t, s} = f_{t, s}|\mathrm{Op}(\partial H_t \cup C_t)$ for all $t \in I^d, s \in [0, \varepsilon]$, and $\widetilde{f}_{t, s}|\mathrm{Op}(\partial H_t)$ is constant in $s$. We call such a family-homotopy $\{\widetilde{f}_{t, s} : t \in I^d, s \in [0, \varepsilon]\}$ a \emph{microextension} of the family of germs of formal sections $\{f_t : t \in I^d\}$.
\end{definition}

We illustrate the microflexibility condition for $d = 0$. Let $H$ be an $m$--handle and $C \subset H$ be the core. Suppose $\psi : \mathrm{Op}(H) \to \mathcal{R}$ is a holonomic section. Let $\{\psi_s : \mathrm{Op}(\partial H \cup C) \to \mathcal{R}\}$ be a homotopy of holonomic sections starting at $\psi_0 = \psi|\mathrm{Op}(\partial H \cup C)$, which is compactly supported on the interior of a neighborhood $\mathrm{Op}(C)$ of the core. We give a schematic of the image of $\psi_s$ in Figure \ref{fig-microflex} below. The condition of microflexibility of $\mathcal{R}$ demands the existence of some $\varepsilon > 0$ such that for all $s \leq \varepsilon$, the unshaded regions in Figure \ref{fig-microflex} can be ``filled in" to obtain a holonomic section of $\mathcal{R}$ over $\mathrm{Op}(H)$. 

\begin{figure}[h]
\centering
\includegraphics[scale=0.2]{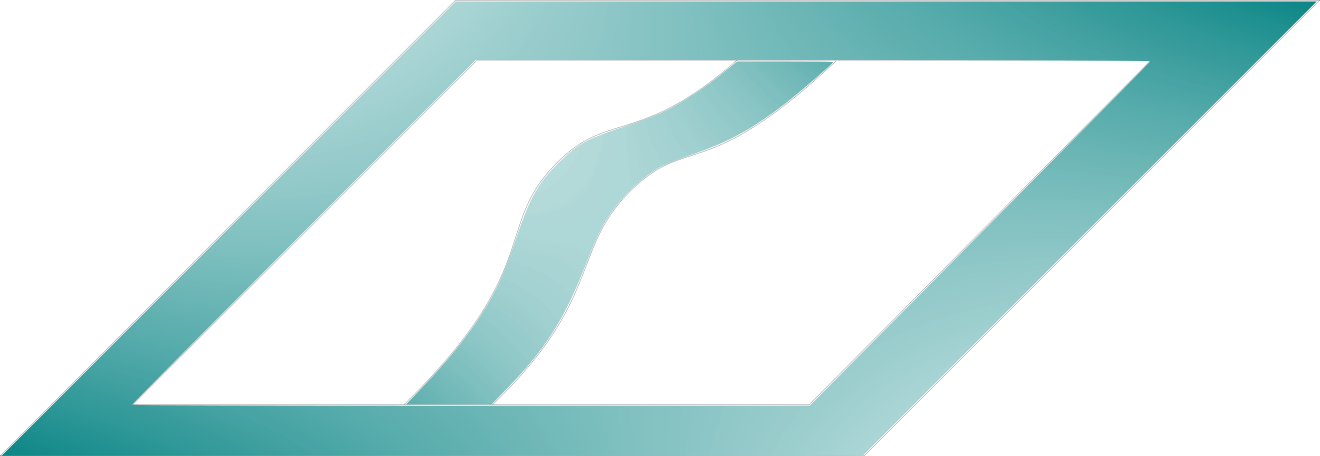}
\caption{Shaded region indicating the graph of $\psi_s$ over $\mathrm{Op}(\partial H \cup C)$.}
\label{fig-microflex}
\end{figure}

\begin{example}[Open implies locally integrable and microflexible]\label{eg-openrel}Let $\mathcal{R} \subset J^r E$ be an \emph{open} differential relation, i.e., $\mathcal{R}$ is an open subset of $J^r E$. Then $\mathcal{R}$ is both parametrically locally integrable and parametrically microflexible.\end{example}

We shall study another example which will be of importance in later sections. First, let us briefly recall Example \ref{eg-mapjet}. Let $M, N$ be smooth manifolds, $E = M \times N$ be the trivial $N$--bundle over $M$, and $J^1(M, N)$ be the associated $1$--jet bundle. We have the map,
$$\pi_{\rm{front}} : J^1(M, N) \to M \times N, \ \pi_f([j^1_p s]) = s(p).$$ 
Then $\pi_{\rm{front}}$ defines a fiber bundle, with fiber over $(p, q) \in M \times N$ being $\mathrm{Hom}(T_pM, T_q N)$.

\begin{example}[Differential relation of isocontact immersions]\label{eg-isoconrel}
Let $(Y, \xi)$, $(Y', \xi')$ be contact manifolds. Consider the differential relation $\mathcal{R}_{\rm{isocon}} \subset J^1(Y, Y')$ defined as follows: for every $y \in Y, y' \in Y'$, let $\pi_{\rm{front}}^{-1}(y, y') \cap \mathcal{R}_{\rm{isocon}}$ consist of monomorphisms $\ell \in \mathrm{Hom}(T_y Y, T_{y'} Y')$ such that $\ell^{-1}(\xi'_{y'}) = \xi_y$, and $\ell^*\mathrm{CS}(\xi')_{y'} = \mathrm{CS}(\xi)_y$. We observe that the holonomic sections of $\mathcal{R}_{\rm{isocon}}$ consist of $1$--jets of isocontact immersions $f : Y \to Y'$. 
\begin{enumerate}[leftmargin=*]
\item $\mathcal{R}_{\rm{isocon}}$ is locally integrable. Indeed, choose any jet in $\mathcal{R}_{\rm{isocon}}$ lying over $y \in Y$. This is simply a monomorphism $\ell \in \mathrm{Hom}(T_y Y, T_{y'} Y')$ such that $\ell^{-1}(\xi'_{y'}) = \xi_y$, and $\ell^*\mathrm{CS}(\xi')_{y'} = \mathrm{CS}(\xi)_y$. We choose ambient Riemannian metrics on $Y, Y'$. Consider the map,
$$g := \mathrm{exp}_{y'} \circ \ell \circ \mathrm{exp}_y^{-1} : \mathrm{Op}(y) \to Y', $$ 
where $\mathrm{exp}_y : T_y Y \to Y$ and $\mathrm{exp}_{y'} : T_{y'} Y' \to Y'$ are the Riemannian exponential maps, which are local diffeomorphisms in a neighborhood of the origin.
Let $\alpha_1 := g^* \alpha'$. By hypothesis, at the point $y$, the forms $\alpha_1, d\alpha_1$ are positive multiples of the forms $\alpha, d\alpha$, respectively. Thus, $\alpha_1 \wedge (d\alpha_1)^{\wedge n} > 0$ at $y$ and hence on $\mathrm{Op}(y)$ by continuity. Therefore, $\xi_1 := \ker \alpha_1$ is a contact structure on $\mathrm{Op}(y)$ agreeing with $\xi$ at $y$. Using Moser's trick (Lemma \ref{lem-moser}) and following the proof of Darboux's theorem (Theorem \ref{thm-darboux}), we obtain a germinal contactomorphism $h : (\mathrm{Op}(y), \xi) \to (\mathrm{Op}(y), \xi_1)$ with $h(y) = y$, $dh_y = \mathrm{id}$. Then,
$$f := g \circ h : \mathrm{Op}(y) \to Y'$$
gives an isocontact immersion with $j^1_y f = \phi$, proving local integrability. The same argument goes through mutatis mutandis for the parametric case as well.
\item $\mathcal{R}_{\rm{isocon}}$ is microflexible. To see this, start with an embedded handle $H \subset Y$ with core $C$, an isocontact embedding $f : (\mathrm{Op}(H), \xi) \to (Y, \xi)$ and an isotopy through isocontact embeddings $f_s : (\mathrm{Op}(\partial H \cup C), \xi) \to (Y', \xi')$ fixed near $\mathrm{Op}(\partial H)$ such that $f_0 = f$. We define a vector field on $f_s(\mathrm{Op}(\partial H \cup C))$ by,
$$V_s(f_s(p)) := \frac{df_s(p)}{ds}, \ p \in \mathrm{Op}(\partial H \cup C).$$
As $f_s^*\alpha' = h_s \alpha$ for some $1$--parameter family of positive functions $h_s > 0$, differentiating with respect to $s$, we get $\mathcal{L}_{V_s} \alpha' = \lambda_s \alpha$ for some $1$--parameter family of functions $\lambda_s$. Hence, $V_s$ is a contact vector field (see, \cite[Definition 1.5.7]{geigbook}) on its domain of definition. Cutting-off by an appropriate bump function, we may extend $V_s$ to a globally defined contact vector field on $Y'$. The nonautonomous flow defined by $V_s$ gives an isotopy through contactomorphisms $\psi_s : Y' \to Y'$ such that $\psi_s \circ f_0 = f_s$ on $\mathrm{Op}(\partial H \cup C)$. Thus, $\widetilde{f}_s := \psi_s \circ f : \mathrm{Op}(\partial H) \to Y'$ provides the desired (micro)extension. Once again, the same argument goes through parametrically. 
\end{enumerate}
We shall call $\mathcal{R}_{\rm{isocon}}$ the \emph{differential relation of isocontact immersions}. Note that by Example \ref{eg-mapjet}, a formal section of $\mathcal{R}_{\rm{isocon}}$ may be equivalently described as a pair of maps $(F, f)$, where $f : Y \to Y'$ is a smooth map and $F : TY \to TY'$ is a bundle-monomorphism, such that the following diagram commutes:
\begin{equation*}
\begin{CD}
TY @>F>> TY' \\
@VVV @VVV \\
Y @>f>> Y'
\end{CD}
\end{equation*}
and for all $p \in Y$, $q = f(p) \in Y'$, we have $F(\xi_p) \subseteq \xi'_q$ and $F^*\mathrm{CS}(\xi')_q = \mathrm{CS}(\xi)_p$. Such pairs of maps $(F, f) : (TY, Y) \to (TY', Y')$ shall be called \emph{formal isocontact immersions}. Such a pair $(F, f)$ corresponds to a holonomic section of $\mathcal{R}_{\rm{isocon}}$ if and only if $F = df$. In this situation, we will call the pair $(df, f)$ a \emph{holonomic isocontact immersion}.
\end{example}

\begin{remark}In literature, Property $(1)$ in Example \ref{eg-isoconrel} is known as \emph{Gray's stability theorem}, while Property $(2)$ is a version of the \emph{contact isotopy extension theorem}. We refer the reader to \cite[Theorem 2.2.2]{geigbook} and \cite[Theorem 2.6.12]{geigbook} for further details.\end{remark}

For the final definition, we introduce a new notion. A smooth fiber bundle of smooth manifolds $(E, B, \pi)$ is called \emph{natural} (see, \cite[Section 7.1]{embook}) if the action of the diffeomorphism group $\mathrm{Diff}(B)$ on $B$ lifts to an action on $E$ by fiber-preserving diffeomorphisms.

\begin{example}\label{eg-natural}The following are examples of natural bundles:
\begin{enumerate}
\item The trivial bundle $E = B \times F$. We define the lift of the diffeomorphism $f : B \to B$ to the fiber-preserving diffeomorphism $f \times \mathrm{id} : E \to E$. 
\item The tangent bundle $E = TB$. We define the lift of the diffeomorphism $f : B \to B$ to the fiber-preserving diffeomorphism $df : E \to E$. 
\item If $(E, B, \pi), (E', B', \pi')$ are natural bundles, then their direct sum $E \oplus E'$, tensor product $E \otimes E'$, exterior power $\Lambda^k E$ and jet bundle $J^r E$ are all natural.
\end{enumerate}
\end{example}

Suppose $(E, B, \pi)$ is natural, then from the above example, we have naturality of $(J^r E, B, \pi^{(r)})$ as well. In this case, we shall denote the action of a diffeomorphism $f \in \mathrm{Diff}(B)$ on an element $s \in J^r E$ by $f_* s$. 

$\mathrm{Diff}(B)$ admits the structure of an (infinite-dimensional) Fr\'echet Lie group. Let $\mathfrak{A} \subset \mathrm{Diff}(B)$ be a Lie subgroup, and let $\mathfrak{a} \subset \mathfrak{X}(M)$ be the corresponding Lie subalgebra.

\begin{definition}[Invariant differential relations]\label{def-diffinvrel}Let $(E, B, \pi)$ be a natural bundle. A differential relation $\mathcal{R} \subset J^r E$ is called \emph{$\mathfrak{A}$--invariant}, if for every $s \in \mathcal{R}$ and $f \in \mathfrak{A}$, $f_* s \in \mathcal{R}$ as well.\end{definition}

\begin{definition}[Capacious subgroups]\label{def-capac}A Lie subgroup $\mathfrak{A} \subset \mathrm{Diff}(B)$ with Lie algebra $\mathfrak{a} \subset \mathfrak{X}(M)$ is \emph{capacious} if 
\begin{enumerate}
\item For any vector field $V \in \mathfrak{a}$, any compact subset $K \subset B$ and any neighborhood $U \supset K$, there exists a vector field $\widetilde{V}_{K, U} \in \mathfrak{a}$ supported in $U$ and coinciding with $V$ on $K$.
\item For any point $x \in B$ and any tangent hyperplane $\tau \subset T_x B$, there exists a vector field $V \in \mathfrak{a}$ transverse to $\tau$ at $x$.
\end{enumerate}
\end{definition}

\begin{example}\label{def-contcap}Let $(Y, \xi)$ be a contact manifold. Due to existence of contact Hamiltonian vector fields (see, \cite[Section 2.3]{geigbook}), the subgroup $\mathrm{Cont}_0(Y, \xi) \subset \mathrm{Diff}(Y)$ of contactomorphisms isotopic to identity, is capacious. 
\end{example}

\subsection{The holonomic approximation theorem}

We are now prepared to discuss the holonomic approximation theorem for locally integrable and microflexible differential relations, due to Eliashberg and Mishachev \cite[Theorem 13.4.1]{embook}.

\begin{theorem}[Holonomic approximation theorem]\label{thm-holapp}
Let $E, B$ be smooth manifolds equipped with ambient Riemannian metrics and $(E, B, \pi)$ be a smooth fiber bundle. Let $\mathcal{R} \subset J^r E$ be a locally integrable and microflexible differential relation. Let $K \subset B$ be a embedded simplicial complex of positive codimension, and $f : \mathrm{Op}(K) \to \mathcal{R}$ be germ of a formal section. Then, for arbitrarily small $\delta, \varepsilon > 0$, there exists
\begin{enumerate}[label=(\arabic*), font=\normalfont, leftmargin=*]
\item A diffeotopy $\{h_t : B \to B : t \in I\}$ such that $h_0 = \mathrm{id}$ and $\mathrm{dist}_{C^0}(h_t, h_0) < \varepsilon$ for all $t \in I$,
\item A holonomic section $\widetilde{f} : \mathrm{Op}(h_1(K)) \to \mathcal{R}$ such that $\mathrm{dist}_{C^0}(\widetilde{f}, f) < \varepsilon$ on $\mathrm{Op}(h_1(K)) \subset \mathrm{Op}(K)$. 
\end{enumerate}
 Moreover, the result is true \emph{parametrically} and \emph{relatively}, that is, suppose
$$f_s : \mathrm{Op}(K) \to \mathcal{R}, \ s \in I^d,$$ 
is a continuous family of germs of formal sections such that $f_s$ is holonomic for all $s \in \partial I^d$. Then for arbitrarily small $\delta, \varepsilon > 0$, there exists
\begin{enumerate}[label=(\arabic*), font=\normalfont, leftmargin=*]
\item A diffeotopy $\{h_{s, t} : B \to B : s \in I^d, t \in I\}$ such that $h_{s, 0} = \mathrm{id}$, $h_{s, t} = \mathrm{id}$ for all $s \in \partial I^d$ and $\mathrm{dist}_{C^0}(h_{s, t}, h_{s, 0}) < \delta$ for all $s \in I^d, t \in I$,
\item A continuous family of holonomic sections 
$$\widetilde{f}_s : \mathrm{Op}(h_{s, 1}(K)) \to \mathcal{R}, \ s \in I^d$$ 
such that $\widetilde{f}_s = f_s$ for all $s \in \partial I^d$, and $\mathrm{dist}_{C^0}(\widetilde{f}_s, f_s) < \varepsilon$ on $\mathrm{Op}(h_{s, 1}(K)) \subset \mathrm{Op}(K)$. 
\end{enumerate}
\end{theorem}

\noindent
A consequence of Theorem \ref{thm-holapp} is the following \cite[Theorem 13.5.1, Theorem 15.2.1]{embook}:

\begin{theorem}[Local $h$--principle for locally integrable, microflexible, invariant relations]\label{thm-lochprinc}Let $(E, B, \pi)$ be a natural bundle {such that the $\mathrm{Diff}(B)$--action on $E$ is continuous with respect to the $C^0$--topology on $\mathrm{Diff}(B)$}. Let $K \subset B$ be a simplicial complex of positive codimension, $\mathfrak{A} \subset \mathrm{Diff}(B)$ be a capacious subgroup and $\mathcal{R} \subset J^r E$ be a locally integrable, microflexible, $\mathfrak{A}$--invariant differential relation. Suppose,
$$\{f_s : \mathrm{Op}(K) \to \mathcal{R} : s \in I^d\}$$
is a continuous family of germs of formal sections such that $f_s$ is holonomic for all $s \in \partial I^d$. Then for arbitrarily small $\varepsilon > 0$, there exists a continuous family of holonomic sections,
$$\{\widehat{f}_s : \mathrm{Op}(K) \to \mathcal{R} : s \in I^d\}$$
such that $\widehat{f}_s = f_s$ for all $s \in \partial I^d$, and $\mathrm{dist}_{C^0}(\mathrm{bs}(\widehat{f}_s), \mathrm{bs}(f_s)) < \varepsilon$ for all $s \in I^d$. Recall that $\mathrm{bs}(f)$ denotes the \emph{base} of a formal section $f$, as in the discussion below Definition \ref{def-jetbdl}.
\end{theorem}

\begin{proof}
For simplicity, consider the case $\mathfrak{A} = \mathrm{Diff}(B)$. Let 
$$\{f_s : \mathrm{Op}(K) \to \mathcal{R} : s \in I^d\}$$ 
be a continuous family of formal sections such that $f_s$ is holonomic for all $s \in \partial I^d$. By Theorem \ref{thm-holapp}, we can find a $C^0$--small parametric diffeotopy $\{h_{s, t} : s \in I^d, t \in I\}$ of $B$ starting at the identity $h_{s, 0} = \mathrm{id}$, and a continuous family of holonomic sections 
$$\{\widetilde{f}_s : \mathrm{Op}(h_{s,1}(K)) \to \mathcal{R} : s \in I^d\}$$
such that $\widetilde{f}_s$ agrees with $f_s$ for all $s \in \partial I^d$, and $\widetilde{f}_s$ is $\varepsilon$--close to $f_s$ on $\mathrm{Op}(h_{s, 1}(K)) \subset \mathrm{Op}(K)$. Using $\mathrm{Diff}(B)$--the invariance of $\mathcal{R}$, we obtain a continuous family of holonomic sections,
$$\{\widehat{f}_s := (h_{s, 1}^{-1})_* \widetilde{f}_s : \mathrm{Op}(K) \to \mathcal{R} : s \in I^d\}.$$
Certainly, $\widehat{f}_s = f_s$ for all $s \in \partial I^d$. Since $\widetilde{f}_s$ and $f_s$ were $\varepsilon$--close, we have,
$$\mathrm{dist}_{C^0}(\widetilde{f}_s, f_s)< \varepsilon.$$
The base section of $\widehat{f}_s$ is obtained from $\mathrm{bs}(\widetilde{f}_s)$ after acting by $h_{s, 1}^{-1}$. Now, the action of $\mathrm{Diff}(B)$ on $E$ is continuous with respect to the $C^0$--topology on $\mathrm{Diff}(B)$. Therefore, as $h_{s, 1}$ is a $C^0$--small diffeomorphism, 
$$\mathrm{dist}_{C^0}(\mathrm{bs}(\widehat{f}_s), \mathrm{bs}(\widetilde{f}_s)) < \varepsilon.$$
By triangle inequality, we obtain $\mathrm{bs}(\widehat{f}_s)$ and $\mathrm{bs}(f_s)$ are $2\varepsilon$--close in $C^0$--topology, as desired. 

Now, let us suppose $\mathfrak{A}$ is a general capcious subgroup. The argument above goes through if the diffeotopy $\{h_{s, t}\}$ can be chosen to belong to $\mathfrak{A}$. This can be achieved by a detailed analysis of the nature of the diffeotopy $\{h_{s, t}\}$ arising from the proof of Theorem \ref{thm-holapp}. We refer the read to \cite[Theorem 15.2.1, p. 134]{embook} for details.
\end{proof}

\begin{remark}In Theorem \ref{thm-lochprinc}, the conclusion $\mathrm{dist}_{C^0}(\mathrm{bs}(\widehat{f}_s), \mathrm{bs}(f_s)) < \varepsilon$ cannot be improved to $\mathrm{dist}_{C^0}(\widehat{f}_s, f_s) < \varepsilon$. This is because in the proof we crucially used that the $\mathrm{Diff}(B)$--action on $E$ is continuous with respect to the $C^0$--topology on $\mathrm{Diff}(B)$. For instance, same is not true for the $\mathrm{Diff}(B)$--action on $J^r E$ for $r > 0$, since a $C^0$--small diffeomorphism can uncontrollably increase the $C^0$--norms of higher order jets. 
\end{remark}

\begin{corollary}\label{cor-heqholapp}
Let $(E, B, \pi)$ be a natural bundle, $K \subset B$ be a simplicial complex of positive codimension, $\mathfrak{A} \subset \mathrm{Diff}(B)$ be a capacious subgroup and $\mathcal{R} \subset J^r E$ be a locally integrable, microflexible, $\mathfrak{A}$--invariant differential relation. Let $\mathrm{Sec}(\mathrm{Op}(K); \mathcal{R}) \subset \Gamma(B; J^r E)$ denote the space of germs of formal sections of $\mathcal{R}$ near $K$, and let $\mathrm{Hol}(\mathrm{Op}(K); \mathcal{R}) \subset \mathrm{Sec}(\mathrm{Op}(K); \mathcal{R})$ denote the subspace of germs of holonomic sections of $\mathcal{R}$ near $K$. The inclusion map 
$$\mathrm{Hol}(\mathrm{Op}(K); \mathcal{R}) \to \mathrm{Sec}(\mathrm{Op}(K); \mathcal{R}),$$
is a weak homotopy equivalence.
\end{corollary}

\begin{proof}
Let us begin with a family of formal sections
$$\{f_s : \mathrm{Op}(K) \to \mathcal{R} : s \in I^d\}$$
such that $f_s$ is holonomic for all $s \in \partial I^d$. From the proof of Theorem \ref{thm-lochprinc}, we obtain a family of holonomic sections,
$$\{\widehat{f}_s = (h_{s, 1}^{-1})_* \widetilde{f}_s : \mathrm{Op}(K) \to \mathcal{R} : s \in I^d\}$$
agreeing with $f_s$ for all $s \in \partial I^d$. We may linearly interpolate $f_s$ and $\widetilde{f}_s$ on $\mathrm{Op}(\varphi_1(K)) \subset \mathrm{Op}(K)$ to get a family-homotopy
$$\{F_{s, t} : \mathrm{Op}(h_{s, 1}(K)) \to \mathcal{R} : s \in I^d, 0 \leq t \leq 1\}$$
of formal sections. We now interpolate the two families $\{f_s : s \in I^d\}$ and $\{\widehat{f}_s : s \in I^d\}$ by a family-homotopy of formal sections, by concaternating the following family-homotopies:
\begin{enumerate}
\item $\{(h_{s, t}^{-1})_* f_s : s \in I^d\}$ for $0 \leq t \leq 1$, interpolating between $\{f_s\}$ and $\{(h_{s, 1}^{-1})_* f_s\}$,
\item $\{(h_{s, 1}^{-1})_* F_{s, t} : s \in I^d\}$ for $0 \leq t \leq 1$, interpolating between $\{(h_{s, 1}^{-1})_* f_s\}$ and $\{\widehat{f}_s\}$.
\end{enumerate}
This gives a new family-homotopy $\{H_{s, t} : \mathrm{Op}(K) \to \mathcal{R}\}$ through formal sections such that $H_{s, 0} = f_s$, $H_{s, 1} = \widetilde{f}_s$ for all $s \in I^d$, and $H_{s, t}$ is constant in $t$ for $s \in \partial I^d$. Thus, for all $d \geq 0$,
$$\pi_d(\mathrm{Sec}(\mathrm{Op}(K); \mathcal{R}), \mathrm{Hol}(\mathrm{Op}(K); \mathcal{R})) = 0$$
The relative homotopy long exact sequence concludes the proof of the corollary.
\end{proof}

\subsection{The convex integration theorem}

In this section, we state a result due to Gromov \cite[Chapter 2.4]{grobook} (see also, \cite[Part 4]{embook}) that can be considered as an ally of the holonomic approximation theorem (Theorem \ref{thm-holapp}), in that it also ensures an $h$--principle for differential relations under certain conditions. However, unlike holonomic approximation, convex integration yields a \emph{global} $h$--principle, for sections defined over the entirety of the base manifold. To proceed, we first introduce some notation.

Let $(E, B, \pi)$ be a smooth fiber bundle of smooth manifolds. For a point $p \in B$, let us denote $E_p := \pi^{-1}(p)$ to be the fiber over $p$, and let $q \in E_p$ be a point in this fiber.

\begin{definition}[Principal subspaces]\label{def-princsub}Let $\tau \subset T_p B$ be a subspace of corank $1$, and $\ell \in \mathrm{Hom}(\tau, T_q E_p)$ be a linear map. The \emph{principal subspace} associated to $\tau, \ell$ is defined as,
$$P(\tau, \ell) = \{j^1_p s \in J^1 E : s(p) = q, df_q|\tau = \ell\}.$$
If $\tau \subset T_p B$ is a coordinate hyperplane with respect to some chosen local coordinate system near $p \in B$, we call $P(\tau, \ell)$ the \emph{coordinate principal subspaces}.
\end{definition}

To illustrate the geometric meaning behind this definition, let us choose a coordinate system $(x^1, \cdots, x^n)$ near $p \in B$, and let $\tau = \{\partial/\partial x^1 = 0\} \subset T_p B$. Let us also choose an adapted coordinates system $(x^1, \cdots, x^n, y^1, \cdots, y^n)$ near $q \in E$. This gives rise to an adapted coordinate system for $J^1 E$,
$$(x^1, \cdots, x^n, y^1, \cdots, y^k, z^{j, i}),$$
as in the discussion above Definition \ref{def-jetbdl}. Here, $z^{j, i} := \partial/\partial x^i \circ y^j$ denotes the $i$--th "formal derivative" of the $j$--th component of the section. Then, the principal coordinate hyperplane $P(\tau, \ell)$ consists of $1$--jets of germs of sections of $E$, taking value $q$ at $p$, such that all the except the first formal derivative components $z^{j, 1}$ $(1 \leq j \leq k)$ of the jet is fixed by $\ell$. 

\begin{remark}[Jet bundle is affine]\label{rmk-jetaffine}We remark that the bundle $(J^1 E, B, \pi^{(1)})$ has the structure of an affine bundle: the fibers admit a natural vector space structure upto a choice of the origin, and there is no canonical choice of origin in general. Thus, the subspace $P(\tau, \ell) \subset (\pi^{(1)})^{-1}(p)$ also admits a structure of an affine subspace.\end{remark}

Next, we shall introduce a new property of differential relations. We begin with a definition from convex geometry.

\begin{definition}[Ample subsets]\label{def-ampleset}
Given an affine space $V$, a subset $\Omega \subset V$, and a point $v \in \Omega$, we shall denote by $\mathrm{Conn}_v(\Omega)$ the path-component of $\Omega$ containing $v$ and by $\mathrm{Conv}_v(\Omega)$ the convex hull of $\mathrm{Conn}_v(\Omega)$. A subset $\Omega \subset V$ shall be called \emph{ample} if $\mathrm{Conv}_v(\Omega) = V$ for all $v \in \Omega$.\end{definition}

\begin{definition}[Ample differential relations]\label{def-amplerel}A \emph{first order} differential relation $\mathcal{R} \subset J^1 E$ will be called \emph{ample} if for all points $p \in B$, $q \in E$ such that $\pi(q) = p$, and any corank $1$ subspace $\tau \subset T_p B$ and linear map $\ell \in \mathrm{Hom}(\tau, T_q E_p)$, $\mathcal{R} \cap P(\tau, \ell)$ is an ample subset of $P(\tau, \ell)$. \end{definition}

We record a useful criterion for verifying ampleness of a differential relation.

\begin{proposition}\label{prop-coordample}Let $(E, B, \pi)$ be a natural bundle, and $\mathcal{R} \subset J^1 E$ be a $\mathrm{Diff}(B)$--invariant first order differential relation. If $\mathcal{R} \cap P(\tau, \ell)$ is ample in $P(\tau, \ell)$ for all coordinate principal subspaces, then $\mathcal{R}$ is an ample differential relation.\end{proposition}

\begin{proof}Let $(x^1, \cdots, x^n)$ be a choice of local coordinates around a point $p \in B$. If $\tau \subset T_p B$ is an arbitrary hyperplane, we can find a germ of a diffeomorphism $f$ of $\mathrm{Op}(\{p\})$ with $f(p) = p$, such that $f_*(\tau_0) = \tau$ where $\tau_0 := \{\partial/\partial x^1 = 0\}$ is a coordinate hyperplane. Let $\ell_0 := \ell \circ (f_*)^{-1}$. Since $f_* : T_p B \to T_p B$ is a vector space isomorphism taking $\tau_0$ to $\tau$, ampleness of $P(\tau_0, \ell_0) \cap \mathcal{R}$ in $P(\tau_0, \ell_0)$ implies ampleness of $P(\tau, \ell) \cap f_*\mathcal{R}$ in $P(\tau, \ell)$. By invariance, we have $f_*\mathcal{R} = \mathcal{R}$. This proves the result.\end{proof}

We are now ready to state the convex integration theorem \cite[Theorem A, Section 2.4.3]{grobook} (see also, \cite[Theorem 18.4.1]{embook}) for first order, open, ample differential relations.

\begin{theorem}[Convex integration theorem]\label{thm-hample}Let $(E, B, \pi)$ be a natural bundle, $\mathcal{R} \subset J^1 E$ be an open ample differential relation, and $f : B \to \mathcal{R}$ be a formal section of $\mathcal{R}$. Then, for arbitrarily small $\varepsilon > 0$, there exists
\begin{enumerate}[label=(\arabic*), font=\normalfont, leftmargin=*]
\item A homotopy of formal sections,
$\{f_t : B \to \mathcal{R} : 0 \leq t \leq 1\},$
such that $f_0 = f$ and $f_1$ is holonomic, and
\item $\mathrm{dist}_{C^0}(\mathrm{bs}(f_t), \mathrm{bs}(f)) < \varepsilon$ for all $0 \leq t \leq 1$.
\end{enumerate}
Moreover, the result is also true parametrically and relatively, that is, suppose
$$\{f_s : V \to \mathcal{R} : s \in I^d\},$$
is a continuous family of formal sections such that $f_s$ is holonomic for all $s \in \partial I^d$. Then for arbitrarily small $\varepsilon > 0$, there exists 
\begin{enumerate}[label=(\arabic*), font=\normalfont, leftmargin=*]
\item A family-homotopy of formal sections,
$$\{f_{s, t} : B \to \mathcal{R} : s \in I^d, 0 \leq t \leq 1\},$$
such that $f_{s, 0} = f_s$ and $f_{s, 1}$ is holonomic for all $s \in I^d$ and $f_{s, t}$ is constant in $t$ for all $s \in \partial I^d$,
\item $\mathrm{dist}_{C^0}(f_{s, t}, f_{s, 0}) < \varepsilon$ for all $0 \leq t \leq 1.$
\end{enumerate}
\end{theorem}

\section{The $h$--principle for Legendrian immersions}\label{sec-hLegimm}

In this section, we generalize Theorem \ref{thm-hprinclegknots} to Legendrian immersions in contact manifolds of arbitrary dimensions. We begin by stating a definition that gives a framework for speaking of the auxilliary data of slopes appearing in the proof of Theorem \ref{thm-hprinclegknots}, in higher dimensions (see, Figure \ref{fig-intpzigzag}).

\begin{definition}\label{def-flegimm}Let $(Y^{2n+1}, \xi)$ be a contact manifold, and $\Lambda^n$ be a smooth manifold. Let $F : T\Lambda \to TY$ and $f : \Lambda \to Y$ be a pair of smooth maps such that the following diagram commutes:
\begin{equation*}
\begin{CD}
T\Lambda @>F>> TY \\
@VVV @VVV \\
\Lambda @>f>> Y
\end{CD}
\end{equation*}
The pair $(F, f)$ is called a \emph{formal Legendrian immersion} if $F$ is a fiberwise monomorphism, and $F(T_p\Lambda) \subseteq (\xi_{q}, \mathrm{CS}(\xi_q))$ is a Lagrangian subspace, for all $p \in \Lambda, q := f(p) \in Y$. We say a formal Legendrian immersion $(F, f)$ is \emph{holonomic} if $F = df$.\end{definition}

Let us equip the set of smooth mappings $C^\infty(\Lambda, Y)$ and $C^\infty(T\Lambda, TY)$ with the weak Whitney topology. We equip the set of holonomic Legendrian immersions $\mathrm{Imm}_{\rm{Leg}}(\Lambda, Y)$ with the induced topology as a subspace of $C^\infty(\Lambda, Y)$, and the set of formal Legendrian immersions $\mathrm{Imm}^f_{\rm{Leg}}(\Lambda, Y)$ with the induced topology as a subspace of $C^\infty(\Lambda, Y) \times C^\infty(T\Lambda, TY)$. Henceforth, homotopy of formal Legendrian immersions will mean paths in $\mathrm{Imm}^f_{\rm{Leg}}(\Lambda, Y)$.

\begin{theorem}[Parametric relative $C^0$--dense $h$--principle for Legendrian immersions]\label{thm-hprincleg}Let $(F, f) : (T\Lambda, \Lambda) \to (TY, Y)$ be a formal Legendrian immersion, and let $\varepsilon > 0$. There is a homotopy through formal Legendrian immersions, 
$$(F_t, f_t) : (T\Lambda, \Lambda) \to (TY, Y), \; 0 \leq t \leq 1,$$
to a holonomic Legendrian immersion $(F_1 = df_1, f_1)$ such that $\mathrm{dist}_{C^0}(f_t, f_0) < \varepsilon$ for all $0 \leq t \leq 1$. Moreover, the result is true \emph{parametrically} and \emph{relatively}, that is, given a smooth family of formal Legendrian immersions,
$$(F_s, f_s) : (T\Lambda, \Lambda) \to (TY, Y), \;\;  s \in D^n,$$ 
such that for all $s \in \partial D^n$, $(F_s, f_s)$ is holonomic restricted to an open set $U \subset \Lambda$, there exists a homotopy through formal Legendrian immersions,
$$(F_{s, t}, f_{s, t}) : (T\Lambda, \Lambda) \to (TY, Y), \;\; 0 \leq t \leq 1, s \in D^n,$$ 
such that 
\begin{enumerate}[label=(\arabic*), font=\normalfont]
\item $F_{s, 1} = df_{s, 1}$, for all $s \in D^n$,
\item $\mathrm{dist}_{C^0}(f_{s, t}, f_{s, 0}) < \varepsilon$, for all $s \in D^n$ and $0 \leq t \leq 1$,
\item $F_{s, t} = F_{s, 0}$ and $f_{s, t} = f_{s, 0}$ restricted to $U$, for all $s \in \partial D^n$ and $0 \leq t \leq 1$.
\end{enumerate}
\end{theorem}

An immediate corollary of the theorem above is the following.

\begin{corollary}\label{thm-heqLeg}The map $\Phi : \mathrm{Imm}_{\rm{Leg}}(\Lambda, Y) \to \mathrm{Imm}^f_{\rm{Leg}}(\Lambda, Y)$ given by $\Phi(f) = (df, f)$ is a weak homotopy equivalence. 
\end{corollary}

\begin{proof}[Proof of Theorem \ref{thm-hprincleg}]
Let $(F, f) : (T\Lambda, \Lambda) \to (TY, Y)$ be a formal Legendrian immersion. Then $F(T\Lambda) \subset (\xi, \mathrm{CS}(\xi))$ is a fiberwise Lagrangian subbundle. Let us choose a fiberwise almost-complex structure $J$ on $\xi$, compatible with $\mathrm{CS}(\xi)$. Then, by the proof of Lemma \ref{lem-lagcomp}, $J(F(T\Lambda)) \subset (\xi, \mathrm{CS}(\xi))$ is another Lagrangian subbundle, such that $F(T\Lambda) \oplus E = \xi$. By co-orientability of $\xi$, we may pick a trivial line bundle $ \underline{\Bbb R} \subset TY$ complementary to $\xi$. Thus, we define a bundle-morphism covering $f$,
$$\Phi : \underline{\Bbb R} \oplus T^*\Lambda \oplus T\Lambda \to TY,$$
where $\Phi$ maps the first component to the chosen trivial line bundle, restricts to $J \circ F$ on the second component, and $F$ on the third component. 

At this point, recall the contact phase space (Example \ref{eg-phase}) $J^1 \Lambda = \Bbb R \times T^*\Lambda$, equipped with the fiber projection $\pi^{(1)} : J^1 \Lambda \to \Lambda$. We choose an Ehressmann connection on $(J^1\Lambda, \Lambda, \pi^{(1)})$ compatible with the contact structure, in the sense that the holonomy preserves the contact hyperplanes of $J^1\Lambda$ equipped with the fiberwise conformal symplectic structures. Then, the tangent bundle to the total space $J^1 \Lambda$ decomposes as,
$$T J^1 \Lambda \cong H \oplus V,$$
where $V = (\pi^{(1)})^* (\underline{\Bbb R} \oplus T^*\Lambda)$ is the \emph{vertical subbundle}, given by the pullback of $\pi^{(1)}$ over itself, and $H = (\pi^{(1)})^*T\Lambda$ is the \emph{horizontal subbundle}, given by the pullback of the tangent bundle of the zero section $0_{\Lambda} \cong \Lambda$ over $\pi^{(1)}$. The definition of pullback gives us two bundle maps $H \to T\Lambda$ and $V \to \underline{\Bbb R} \oplus T^*\Lambda$ covering $\pi^{(1)}$, which we consolidate into a bundle map 
$$\Pi : TJ^1{\Lambda} \to \underline{\Bbb R} \oplus T^*\Lambda \oplus T\Lambda.$$ 
We then define a bundle morphism $G : TJ^1\Lambda \to TY$ by setting $G = \Phi \circ \Pi$. Moreover, let $g : J^1\Lambda \to Y$ be defined by $g = f \circ \pi^{(1)}$. Notice that the bundle morphism $G$ covers the map $g$ on the base spaces. Since the Ehresmann connection was chosen to be compatible with the contact structure on $J^1 \Lambda$, $G$ preserves the fiberwise contact structures equipped with the fiberwise conformal symplectic structures. Therefore, $(G, g)$ is a formal isocontact immersion, as defined in Example \ref{eg-isoconrel}.

Thus, we may treat the pair $(G, g)$ as a formal section of the differential relation $\mathcal{R} \subset J^1(J^1\Lambda, Y)$ of isocontact immersions described in Example \ref{eg-isoconrel}. We have already proved $\mathcal{R}$ is locally integrable and microflexible in Example \ref{eg-isoconrel}. If $\mathfrak{A} = \mathrm{Cont}_0(Y, \xi)$ denotes the connected component of identity in the group of contactomorphisms, we know $\mathfrak{A} \subset \mathrm{Diff}(Y)$ is a capacious subgroup and $\mathcal{R}$ is certainly $\mathfrak{A}$--invariant. Thus, we may now use the local $h$--principle (Corollary \ref{thm-lochprinc}) near the positive codimensional subcomplex $K = 0_{\Lambda} \subset J^1\Lambda$. This yields a holonomic isocontact immersion $g_1 : \mathrm{Op}(0_{\Lambda}) \to Y$ such that,
\begin{enumerate}
\item $(dg_1, g_1)$ is homotopic to $(G, g)$ through formal isocontact immersions $(G_t, g_t)$ of $\mathrm{Op}(0_{\Lambda})$ in $Y$. 
\item $\{g_t : 0 \leq t \leq 1\}$ is a $C^0$--small homotopy between $g$ and $g_1$.
\end{enumerate}
We may recover $(F, f)$ by restricting $(G, g)$ to the zero section $(T0_{\Lambda}, 0_{\Lambda}) \subset (TJ^1\Lambda, J^1\Lambda)$ of the $1$--jet bundle. Thus, restricting the homotopy $(G_t, g_t)$, we obtain a homotopy 
$$\{(F_t, f_t) : (T\Lambda, \Lambda) \to (TY, Y) : 0 \leq t \leq 1\}$$
through formal Legendrian immersions, such that $F_0 = f, f_1 = f$ and $F_1 = df_1$, where $f_1 = g_1|0_{\Lambda}$ is a genuine Legendrian immersion. Moreover, the homotopy $\{f_t : 0 \leq t \leq 1\}$, being a restriction of the $C^0$--small homotopy $\{g_t : 0 \leq t \leq 1\}$, is also $C^0$--small.

The parametric argument is completely analogous, with some extra bookkeeping.\end{proof}

Note that a generic Legendrian immersion $f : \Lambda \to (Y, \xi)$ is in fact an embedding. This follows from general position arguments, as $2 \dim \Lambda < \dim Y$. However, Theorem \ref{thm-hprincleg} does \emph{not} provide any insight into the homotopy type of the space $\mathrm{Emb}_{\rm{Leg}}(\Lambda, Y)$ of Legendrian embeddings. As an illustration, suppose $(F, f) : (T\Lambda, \Lambda) \to (TY, Y)$ is a formal Legendrian immersion, where $f : \Lambda \to Y$ is an embedding. While Theorem \ref{thm-hprincleg} shows $(F, f)$ is isotopic through formal Legendrian immersions to a Legendrian embedding $(dg, g)$, the homotopy between $f, g : \Lambda \to Y$ need not be a smooth isotopy through embeddings. 

\begin{question}\label{que-htpylegemb}What is the homotopy type of $\mathrm{Emb}_{\rm{Leg}}(\Lambda, Y)$?\end{question}

The subsequent sections will develop the necessary tools to study Question \ref{que-htpylegemb}.

\section{The $h$--principle for directed immersions and embeddings}\label{sec-dirimmemb}

Let $M^m, N^n$ be smooth manifolds, and $f : M \to N$ be an immersion. Let us recall the \emph{Grassmann bundle} $\mathrm{Gr}_m(N)$ over $N$, which is a smooth fiber bundle, with fiber over $q \in N$ being the Grassmannian of $m$--dimensional subspaces of $T_q N$. We define the \emph{tangential Gauss map}, or simply \emph{Gauss map} of $f$, to be the map 
$$G(df) : M \to \mathrm{Gr}_m(N),$$ 
where $G(df)(p)$ is defined to be the point in the Grassmann bundle corresponding to the subspace $f_*(T_p M) \subseteq T_q N$, $q := f(p)$.

\begin{definition}[Differential relation of $A$--directed immersions]\label{def-adirrel}Let $M$ be a smooth manifold, and $A \subset \mathrm{Gr}_m(N)$ be an arbitrary subset. Let $\mathcal{R}_A \subset J^1(M, N)$ denote the subset consisting of $1$--jets of (sections which are graphs of) immersions $f : M \to N$ at points of $M$, such that the tangential Gauss map of $f$ has image contained in $A$, i.e., $G(df)(M) \subseteq A$. We call $\mathcal{R}_A$ the \emph{differential relation of $A$--directed immersions}.\end{definition}

In the following definition, we give an alternative description of formal and holonomic sections of the differential relation $\mathcal{R}_A$. First, let us introduce a notational convention: for a subset $A \subset \mathrm{Gr}_m(N)$ and a point $q \in N$, we shall denote $A_q := A \cap \mathrm{Gr}_n(T_q N)$. 

\begin{definition}[Formal and holonomic $A$--directed immersions]\label{def-adirimm}Let $M^m, N^n$ be smooth manifolds, and $A \subset \mathrm{Gr}_m(N)$ be a subset. Let $F : TM \to TN$ be a smooth, fiber-preserving, fiberwise-linear map and $f : M \to N$ be a smooth maps such that the following diagram commutes: 
\begin{equation*}
\begin{CD}
TM @>F>> TN \\
@VVV @VVV \\
M @>f>> N
\end{CD}
\end{equation*}
The pair $(F, f)$ is called a \emph{formal $A$--directed immersion} if $F$ is a fiberwise monomorphism, and the subspace $F(T_p M) \subset T_q N$ corresponds to a point in $A_q$, $q := f(p)$. We say a formal $A$--directed immersion $(F, f)$ is \emph{holonomic} if $F = df$. For a holonomic $A$--directed immersion $(F = df, f)$, the map on the base spaces $f$ is simply said to be an \emph{$A$--directed immersion}. \end{definition}

\begin{definition}\label{def-ample}Let $A \subset \mathrm{Gr}_m(N)$ be an open subset. We say $A$ is \emph{ample} if for all $q \in N$, for any $m$--dimensional subspace $P \subset T_q N$ such that $P \in A_q$, and for any subspace $Q \subset P$ of codimension $1$, the set
$$\Omega_Q := \{v \in T_q N : \mathrm{span}(Q, v) \in A_q\},$$ 
is an ample subset of $T_q N$ in the sense of Definition \ref{def-ampleset}.
\end{definition}

\begin{theorem}[$C^0$--dense $h$--principle for $A$--directed immersions]\label{thm-adirimm}Let $A \subset \mathrm{Gr}_m(N)$ be an open ample set. Then every formal $A$--directed immersion $(F, f) : (TM, M) \to (TN, N)$ is homotopic through formal $A$--directed immersions $\{(F_t, f_t) : 0 \leq t \leq 1\}$ to a holonomic $A$--directed immersion $(df_1, f_1)$ such that $\{f_t : M \to N : 0 \leq t \leq 1\}$ is a $C^0$--small homotopy. Moreover, this result is also true parametrically and relatively.\end{theorem}

Here, we suppressed the elaborate meaning of ``$C^0$--small homotopy" and the theorem being true ``parametrically and relatively", as the usage is identical to that of  Theorem \ref{thm-hprincleg}. Let us denote $\mathrm{Imm}_A(M, N)$ and $\mathrm{Imm}_A^f(M, N)$ to be the spaces of $A$--directed immersions and formal $A$--directed immersions, respectively, topologized by the subspace topology inherited from $\mathrm{Imm}(M, N)$ and $\mathrm{Imm}^f(M, N)$ as in the discussion below Definition \ref{def-flegimm}. The following corollary is immediate from Theorem \ref{thm-adirimm}.

\begin{corollary}For an open ample set $A \subset \mathrm{Gr}_m(N)$, the inclusion map $\mathrm{Imm}_A(M, N) \to \mathrm{Imm}^f_A(M, N)$ is a weak homotopy equivalence.\end{corollary}

We shall derive Theorem \ref{thm-adirimm} from the Convex Integration Theorem (Theorem \ref{thm-hample}).

\begin{proof}[Proof of Theorem \ref{thm-adirimm}]In light of Theorem \ref{thm-hample}, we need only verify that the differential relation $\mathcal{R}_A \subset J^1(M, N)$ is open (Example \ref{eg-openrel}) and ample (Definition \ref{def-amplerel}). The former of these is evident, as $A \subset \mathrm{Gr}_m(M)$ is an open subset. Thus, we need only prove $\mathcal{R}_A$ is ample. 

Let $\tau \subset T_p M$ be a subspace of corank $1$, and $\ell \in \mathrm{Hom}(T_p M, T_q N)$ be a linear map. Let $P(\tau, \ell)$ denote the corresponding principal subspace (Definition \ref{def-princsub}). If $\ell$ is not injective, $P(\tau, \ell) \cap \mathcal{R}_A$ is empty as $\mathcal{R}_A \subset J^1(M, N)$ consist only of $1$--jets of immersions. Thus, in this case, $\mathcal{R}_A \cap P(\tau, \ell)$ is vacuously ample in $P(\tau, \ell)$. Otherwise, $P(\tau, \ell) \cap \mathcal{R}_A$ consists of monomorphisms $\phi : T_p M \to T_q N$ whose image corresponds to a point in $A \subset \mathrm{Gr}_m(T_p M)$, and such that $\phi|\tau = \ell$. These are determined by a choice of a vector $v \in T_q N$ such that $\mathrm{span}(Q, v) \in A_q$, where $Q = \ell(\tau) \subset T_q N$. Thus, $P(\tau, \ell) \cap \mathcal{R}_A$ is mapped to $\Omega_Q$ under the affine equivalence $P(\tau, \ell) \cong T_q N$. As $A \subset \mathrm{Gr}_m(M)$ is ample in the sense of Definition \ref{def-ample}, $\Omega_Q$ is an ample subset of $T_q N$ in the sense of Definition \ref{def-ampleset}. Consequently, $P(\tau, \ell) \cap \mathcal{R}_A$ is an ample subset of $P(\tau, \ell)$. Therefore, $\mathcal{R}_A$ is an ample differential relation.\end{proof}

Remarkably, an $h$--principle for $A$--directed \emph{embeddings} is also true, with the same hypothesis on $A$ as in Theorem \ref{thm-adirimm}. We begin with a definition.

\begin{definition}[Formal and holonomic $A$--directed embeddings]\label{def-adiremb}Let $M, N$ be smooth manifolds, and $A \subset \mathrm{Gr}_n(N)$ be a subset. Let 
$$\{F_s : TM \to TN : 0 \leq s \leq 1\},$$ 
be a family of smooth, fiber-preserving, fiberwise-linear maps and $f : M \to N$ be a smooth map such that the following diagram commutes for every $0 \leq s \leq 1$:
\begin{equation*}
\begin{CD}
TM @>F_s>> TN \\
@VVV @VVV \\
M @>f>> N
\end{CD}
\end{equation*}
The pair $(\{F_s : 0 \leq s \leq 1\}, f)$, more succintly denoted as $(F_s, f)$, is called a \emph{formal $A$--directed embedding} if $F_s$ is a homotopy of fiberwise monomorphisms, $f$ is a smooth embedding, $F_0 = df$ and $(F_1, f)$ is an $A$--directed immersion. If $F_s = df$ for all $0 \leq s \leq 1$, we call such a pair a \emph{holonomic $A$--directed embedding}. In that case, $f$ is called an \emph{$A$--directed embedding}. \end{definition}

\begin{remark}[Homotopy of formal $A$--directed embeddings]We remark that a $d$--parameter homotopy of formal $A$--directed embeddings is actually of the form 
$$\{(F_{s, t}, f_t) : (TM, M) \to (TM, M) : s \in I, t \in I^d\},$$
where $(F_{s, t}, f_t)$ is a formal $A$--directed embedding for every $t \in I^d$. \end{remark}

\begin{theorem}\label{thm-adiremb}Let $A \subset \mathrm{Gr}_m(N)$ be an open, ample set. Any formal $A$--directed embedding $(F_s, f) : (TM, M) \to (TN, N)$ can be homotoped to a holonomic $A$--directed embedding by a homotopy $(F_{s, t}, f_t)$ through $A$--directed embeddings, where $F_{s, 0} = F, f_0 = f$, $F_{s, 1} = df_1$, $f_1 : M \to N$ is an $A$--directed embedding, and the homotopy $\{f_t\}$ is $C^0$--small. The result is also true parametrically and relatively.\end{theorem}

As before, we define the space of $A$--directed embeddings $\mathrm{Emb}_A(M, N)$, which inherits a topology as a subspace of $\mathrm{Imm}_A(M, N)$. We also define a space of formal $A$--directed embeddings $\mathrm{Emb}_A^f(M, N)$, which inherits a topology as a subspace of the slightly more complicated topological space $\mathrm{Maps}(I, C^\infty(TM, TN)) \times C^\infty(M, N)$.  Once again, we have as an immediate consequence of Theorem \ref{thm-adiremb},

\begin{corollary}\label{cor-hadiremb}For an open ample set $A \subset \mathrm{Gr}_m(N)$, the inclusion map 
$$\mathrm{Emb}_A(M, N) \to \mathrm{Emb}_A^f(M, N)$$
is a weak homotopy equivalence.\end{corollary}

\begin{proof}[Proof of Theorem \ref{thm-adiremb}]
Let us identify $M$ with the embedded submanifold $f(M) \subset N$. Let $(E, M, \pi)$ denote the normal bundle to $M \cong f(M) \subset N$, and let $\iota : E \to N$ denote the embedding obtained from the tubular neighborhood theorem. We also identify $E$ with $\iota(E) \subset N$. 

Choose an Ehresmann connection on $E$, so that $TE \cong \pi^*TM \oplus \pi^*E$. We may extend the homotopy of fiberwise monomorphisms $\{F_s : TM \to TN : s \in I\}$ to a homotopy of fiberwise isomorphisms $\{G_s : TE \to TN : s \in I\}$ by defining $G_s$ as follows:
\begin{enumerate}
\item On the first factor of $TE \cong \pi^*TM \oplus \pi^*E$, $G_s|\pi^*TM = F_s \circ \pi$.
\item On the second factor of $TE \cong \pi^*TM \oplus \pi^*E$, $G_s|\pi^*E = \iota \circ \pi$.
\end{enumerate}
We equip $N$ with an ambient Riemannian metric, and let $\varepsilon > 0$ be sufficiently small. First, let us assume that $\{G_s : s \in I\}$ is a \emph{small} homotopy, in the following sense: for every rank $m$ subspace $P \in \mathrm{Gr}_m(E) \subset \mathrm{Gr}_m(N)$ of some tangent space to the total space of $E$, the subspaces $(G_s)_* P$ and $P$ are $\varepsilon$--close in $\mathrm{Gr}_m(N)$. For any $p \in M$, $g_*(T_p M)$ is transverse to the normal fibers of the tubular neighborhood $E$. From the assumption of $\{G_s : s \in I\}$ being small, we thus conclude $(G_s)_*(T_p M) = (F_s)_*(T_p M)$ is also transverse to the normal fibers of $E$, for all $s \in I$. In virtue of this transversality, we may consider $(F_s, f)$ as a formal section of the jet bundle $(J^1 E, M, \pi^{(1)})$ for every $s \in I$. 

We define a differential relation $\mathcal{R}^E_A \subset J^1 E$ consisting of $1$--jets of sections $s : \mathrm{Op}(\{p\}) \to E$ which are $A$--directed as a map to $E \cong \iota(E) \subset N$. Since $\mathcal{R}_A$ is open and ample, so is $\mathcal{R}^E_A$. As $(F_s, f)$ is an $A$--directed immersion, it must in fact correspond to a formal section of $\mathcal{R}^E_A \subset J^1 E$. By the parametric version of Theorem \ref{thm-hample}, we can find now find the required homotopy $(F_{s, t}, f_t)$ to a holonomic section of $\mathcal{R}^E_A$. 

In general, the homotopy $\{G_s : s \in I\}$ may not be small. In that case, we choose a sequence of times $0 = s_0 < s_1 < \cdots < s_N = 1$ such that the homotopies $\{G_s : s \in [s_{i-1}, s_i]\}$ are small for all $1 \leq i \leq N$. Let $A_i := \{(G_{s_i})_* P : P \in \mathrm{Gr}_m(N)\}$ be the subset of $\mathrm{Gr}_m(N)$ consisting of pushforward of the points in $A \cap \mathrm{Gr}_m(E)$ by $(G_{s_i})_* : \mathrm{Gr}_m(E) \to \mathrm{Gr}_m(N)$. The strategy next is to inductively apply the above procedure on each of the truncated homotopies $\{G_s : s \in [s_{i-1}, s_i]\}$. Suppose we inductively obtain an embedding $f_{s_i} : M \to N$, and the normal bundle $E_i$ to $f_{s_i}$. We then apply the procedure above for the small homotopy $\{G_s : s \in [s_i, s_{i+1}]\}$, with the differential relation $\mathcal{R}^{E_i}_{A_i}$, to obtain $f_{s_{i+1}} : M \to N$ (see Figure \ref{fig-adirwiggles}). The final embedding $f_1 := f_{s_N}$ has all the required properties. \end{proof}

\begin{remark}We emphasize that the crux of the proof of Theorem \ref{thm-adiremb} is to decompose the homotopy of formal $A$--directed immersions associated to a given formal $A$--directed embedding as a composition of smaller homotopies. Inductively, we reinterpret these smaller homotopies as a homotopy through formal $1$--jets of sections of the normal bundle, which can thus be approximated by a homotopy through holonomic sections using the convex integration theorem. This inductive procedure gives rise to progressively ``wiggled" sections, which requires that we shrink the tubular neighborhood before running the next stage.\end{remark}

\begin{figure}[h]
\centering
\includegraphics[scale=0.13, trim={0 15cm 0 16cm}, clip]{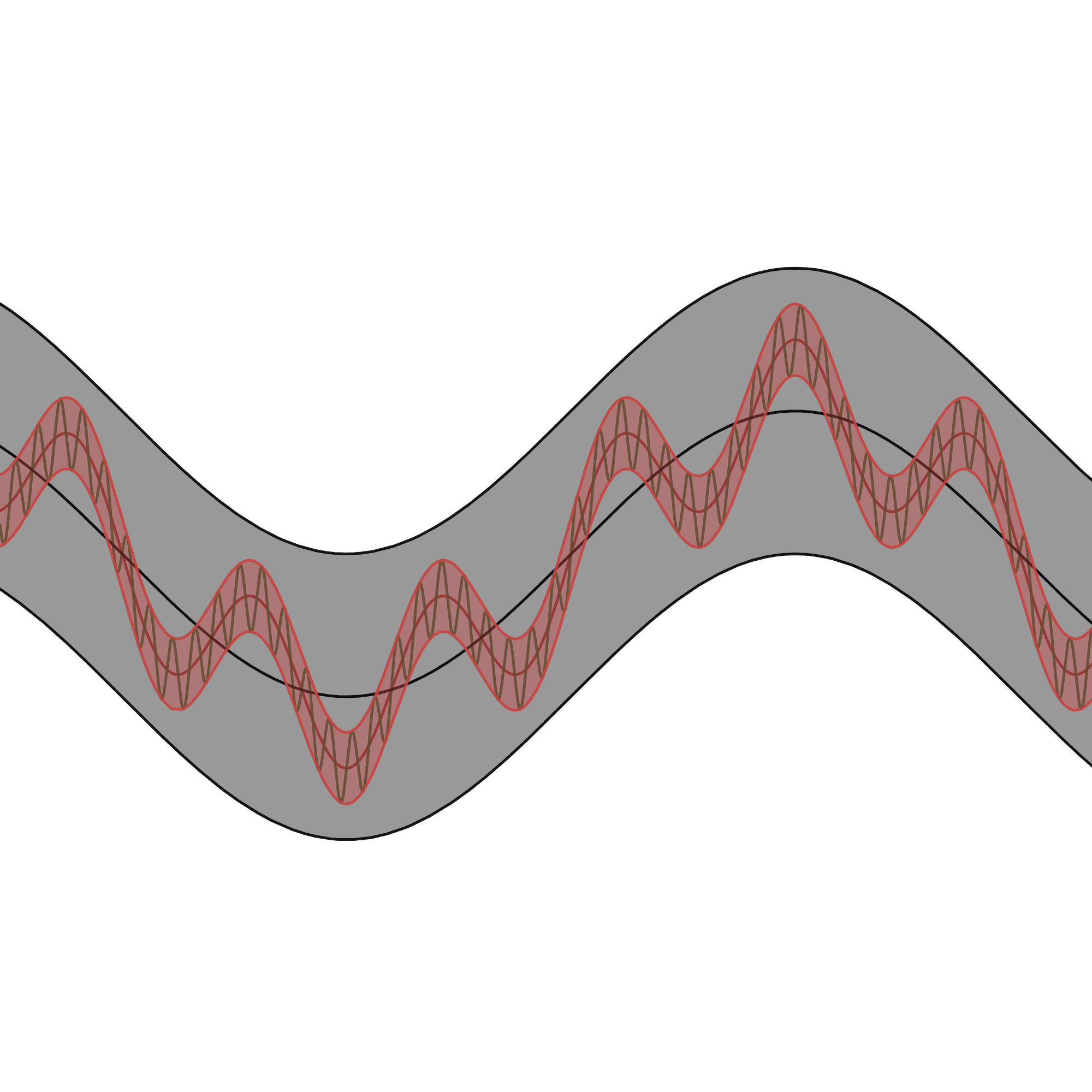}
\caption{$f_{s_1}, f_{s_2}, f_{s_3}$ (respectively, black, red, green) and the tubular neighborhoods $E_1, E_2$ (respectively, shaded in black and red).}
\label{fig-adirwiggles}
\end{figure}

\section{The $h$--principle for wrinkled embeddings}\label{sec-wrinkemb}

In this section, we discuss a generalization of Theorem \ref{thm-adiremb} by Eliashberg and Mishachev \cite{empaper}, in the case where we impose no directedness restrictions but weaken the approximating maps to be topological embeddings with \emph{mild} singularities. 

Let $\psi_\delta : \Bbb R \to \Bbb R^2$ be the map defined by,
$$\psi_\delta(t) := \left ( t^3 - 3\delta t, \int_0^t (s^2 - \delta)^2 ds \right).$$ 
The curve parametrized by $\psi_\delta$ (see, Figure \ref{fig-modelzig}) is a model case of an interpolating zig-zag discussed in Observation \ref{obs-intpzigzag}. We shall use these to construct higher-dimensional singularities. 

\begin{figure}[h]
\centering
\includegraphics[scale=0.2]{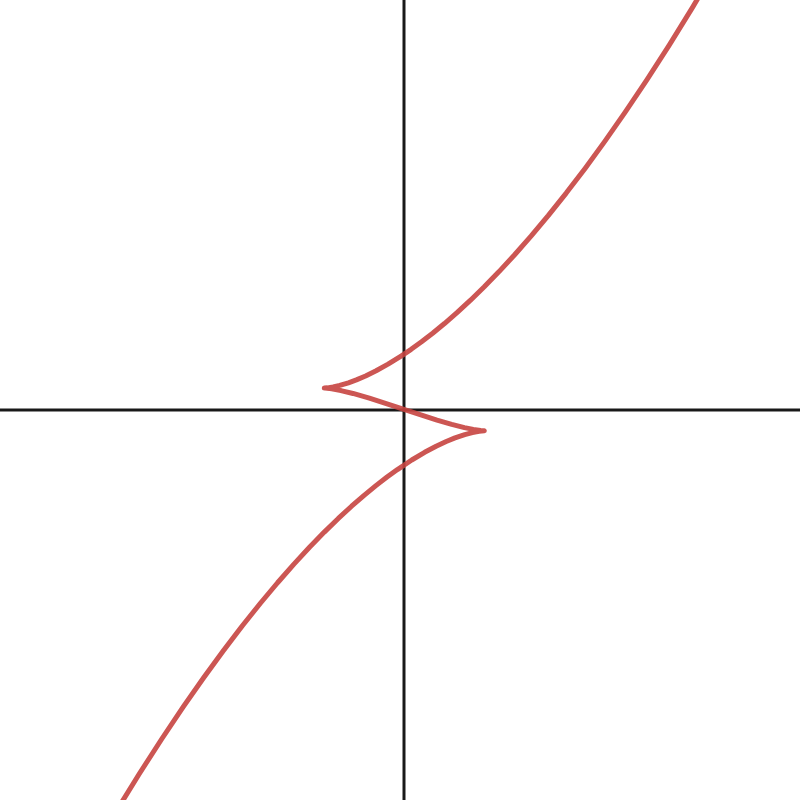}
\caption{Parametric plot of $\psi_\delta$, a model zig-zag.}
\label{fig-modelzig}
\end{figure}

\begin{definition}\label{def-wrinkledemb}Let $M^m, N^n$ be manifolds of indicated dimensions, and $m < n$. A smooth map $f : M \to N$ is a \emph{wrinkled embedding} if $f$ is a topological embedding that is a smooth embedding away from a collection of embedded spheres $\{S^{m-1}_j \subset V : j \in I\}$ such that,
\begin{enumerate}
\item Each sphere $S^{m-1}_j \subset M$ bounds an embedded $m$--ball $D^m_j \subset M$,
\item Near each sphere $S^{m-1}_j$, $f$ is equivalent to the map
\begin{gather*}
f : \mathrm{Op}_{\Bbb R^m}(S^{m-1}) \to \Bbb R^n\\
 f(x_1, \cdots, x_{m-1}, t) = (x_1, \cdots, x_{m-1}, \psi_{1-|x|^2}(t), 0, \cdots, 0)
\end{gather*}
where $x_1, \cdots, x_{m-1}, t$ are coordinates of $\Bbb R^m$, and $S^{m-1} = \{|x|^2 + t^2 = 1\}$.
\end{enumerate}
A family of maps $\{f_s : M \to N : s \in I^d\}$ is called a \emph{family of wrinkled embeddings} if there exists a codimension $1$ compact submanifold $\mathcal{E} \subset \mathrm{in}(I^d)$ such that for all $s \in I^d \setminus \mathcal{E}$, $f_s$ is a wrinkled embedding and if $\tau$ is a choice of coordinate on $I^d$ transverse to $\mathcal{E} = \{\tau = 0\}$, then the family $\{f_\tau : \tau \in (-\varepsilon, \varepsilon)\}$ is conjugate to an \emph{embryo singularity}, given by the family
\begin{gather*}f_\tau : \mathrm{Op}_{\Bbb R^m}(0) \to \Bbb R^n\\
f_\tau(x_1, \cdots, x_{m-1}, s) = (x_1, \cdots, x_{m-1}, \psi_{\tau-|x|^2}(s), 0, \cdots, 0)\end{gather*}
\end{definition}

\begin{remark}Let $f : M \to N$ be a wrinkled embedding, and $S^{m-1}_j \subset M$ be one of the singular spherical locus of $f$. The singularities of $f$ along $S^{m-1}_j$ comes in two flavours:
\begin{enumerate}
\item Along the equator $S^{m-2}_j := \{t = 0\}$ of $S^{m-1}_j$, $f$ has \emph{unfurled swallowtail singularities}.
\item Along the two hemispheres $S^{m-1}_j \setminus S^{m-1}_j$, $f$ has \emph{cuspidal singularities}.
\end{enumerate}
A picture of a two-dimensional model wrinkle $f : \Bbb R^2 \to \Bbb R^3$ is given below. 

\begin{figure}[h]
\centering
\includegraphics[scale=0.2]{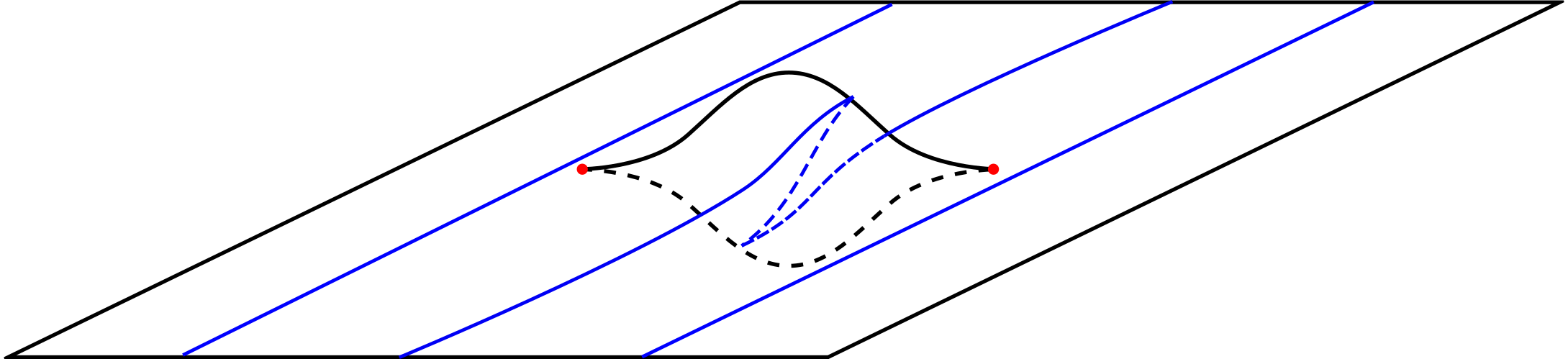}
\caption{Parametric plot of a two-dimensional model wrinkle. Cuspidal edges are drawn in black, and swallowtail points are drawn in red.}
\label{fig-wrinkle}
\end{figure}

Note that a wrinkled embedding $f$ is allowed to have further wrinkled singularities inside the disks $D^{m-1}_j$ bounded by the spherical loci $S^{m-1}_j$. That is to say, the collection of disks $\{D^{m-1}_j : j \in I\}$ may be nested. Allowing such ``wrinkles within wrinkes" (see, Figure \ref{fig-wriwrinkle}) will be crucial for the $h$--principle (Theorem \ref{thm-hwrink}) to hold.

\begin{figure}[h]
\centering
\includegraphics[scale=0.2]{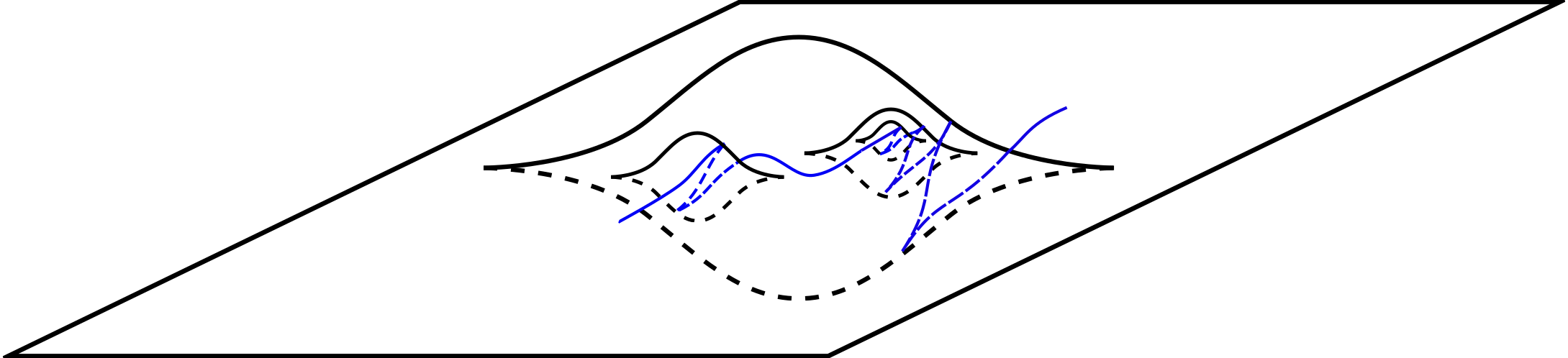}
\caption{Wrinkling within wrinkles.}
\label{fig-wriwrinkle}
\end{figure}

Embryo singularities are higher-dimensional analogues of wrinkles in the sense that while wrinkle singularities are observed in families of cuspidal embeddings where a pair of cusps are created or destroyed, embryo singularities appear when wrinkles themselves are created or destroyed.\end{remark}

If $f : M \to N$ is a wrinkled embedding, then $f_0 := f|(M \setminus \cup_j S^{m-1}_j)$ is a smooth embedding. Thus, we have a corresponding Gauss map,
$$G(df_0) : M \setminus \cup_j S^{m-1}_j \to \mathrm{Gr}_m(N).$$
Let $x \in M$ be a point at which the rank of $df_x$ drops, and let $y = f(x)$. Then $x$ is either a cuspidal singularity, or a swallowtail singularity. In the local model of such singularities, it can be checked that if $\{x_n\} \subset M$ is a sequence of points converging to $x$, then the tangent spaces to the image $f_*(T_{x_n} M)$ converge to a unique rank $m$ subspace of $T_y N$ which is independent of the chosen sequence. Therefore, we may uniquely extend $G(df_0)$ to a map,
$$G(df) : M \to \mathrm{Gr}_m(N),$$
even though $f$ is not a smooth embedding on all of $M$. 

The following theorem of Eliashberg and Mishachev \cite[Theorem 2.5.1]{empaper} establishes an $h$--principle for wrinkled embeddings.

\begin{theorem}[$C^0$--dense $h$--principle for wrinkled embeddings]\label{thm-hwrink}Let $\{G_t : M \to \mathrm{Gr}_m(N) : t \in I\}$ be a \emph{tangential homotopy} of a fixed embedding $f : M \to N$, that is,
\begin{enumerate}[label=(\arabic*), font=\normalfont]
\item For every $t \in I$, $G_t$ \emph{covers} $f : M \to N$, i.e., for all $p \in M$, $G_t(T_p M) \subset T_{f(p)} N$.
\item $G_0 = G(df)$. 
\end{enumerate}
Then, there is a homotopy of wrinkled embeddings $\{f_t : M \to N : t \in I\}$, such that $f_0 = f$, and $G(df_t)$ is $C^0$--close to $G_t$ for all $t \in I$. Moreover, the theorem is true parametrically and relatively.
\end{theorem}

\section{Formal Legendrian embeddings and their fronts}\label{sec-flegemb}

In this and subsequent sections, we begin the process of addressing Question \ref{que-htpylegemb}. Taking inspiration from ideas in Section \ref{sec-dirimmemb} and Corollary \ref{cor-hadiremb} in particular, we give the definition of a formal \emph{Legendrian} embedding, introduced by Murphy \cite[Definition 1.1]{mur}. To fix notation, let $(Y^{2n+1}, \xi)$ be a contact manifold and $\Lambda^n$ be a connected manifold.

\begin{definition}[Formal and holonomic Legendrian embeddings]\label{def-flegemb}A \emph{formal Legendrian embedding} is a tuple $(\{F_s : s \in I\}, f)$, more succintly denoted as $(F_s, f)$, where 
$$\{F_s : T\Lambda \to TY : s \in [0, 1]\}$$ 
is a $1$--parameter family of fiberwise monomorphisms covering a smooth embedding $f : \Lambda \to Y$ such that $F_0 = df$, and  $F_1(T\Lambda) \subseteq \xi$. If $F_s = df$ for all $s \in I$, we call the tuple a \emph{holonomic Legendrian embedding}. In this case, $f$ itself is a Legendrian embedding.\end{definition}

Recall from Example \ref{eg-1jet} and \ref{rmk-1jetcont} that for any manifold $Q^n$, the $1$--jet bundle $J^1 Q \cong \Bbb R \times T^* Q$ admits a natural contact structure $\xi_0 = \ker(dz - \sum_i p^i dq_i)$. We had also defined a map 
$$\pi_{\rm{front}} : J^1 Q \to Q \times \Bbb R.$$
In this particular case, it is given by $\pi_{\rm{front}}(q_1, \cdots, q_n, p_1, \cdots, p_n, z) = (q_1, \cdots, q_n, z)$. Thus, the front projection map considered in Definition \ref{def-legprojknot} is a special case of $\pi_{\rm{front}}$ for $Q = \Bbb R$. We shall call $\pi_{\rm{front}} : J^1 Q \to Q$ the \emph{front projection} from now on.

If $\Lambda \subset (Y, \xi)$ is a (holonomic) Legendrian submanifold, then by the Weinstein tubular neighborhood theorem (Theorem \ref{thm-weinstein}), we know there is a tubular neighborhood of $\Lambda$ contactomorphic to $J^1 \Lambda$. Thus, the front projection $\pi_{\rm{front}} : J^1 \Lambda \to \Lambda \times \Bbb R$ defines a projection map on such a tubular neighborhood. Legendrian submanifolds of $(Y, \xi)$ which are nearby $\Lambda$ may thus be projected via $\pi_{\rm{front}}$ to $\Lambda \times \Bbb R$ and studied using techniques from Section \ref{sec-legknots}.

The main theorem of this section is a construction of a front projection for formal Legendrian embeddings (see, \cite[Proposition 2.1]{mur}).

\begin{theorem}[Graphicality of formal Legendrian embeddings]\label{thm-graphfLeg}Let $(F_s, f) : (T\Lambda, \Lambda) \to (TY, Y)$ be a formal Legendrian embedding. There exists a smooth isotopy of the base embedding $f$ to an embedding $\widetilde{f} : \Lambda \to Y$, an open subset $U \subset Y$ containing $\widetilde{f}(\Lambda)$ and a contactomorphism onto image $\varphi : U \to J^1 \Lambda$ such that:
$$\pi_{\rm{front}} \circ \varphi \circ \widetilde{f} = \mathrm{id}_{\Lambda}.$$
We summarize this condition by saying $\varphi(\widetilde{f}(\Lambda))$ is \emph{graphical} in $J^1\Lambda$.\end{theorem}

To prove Theorem \ref{thm-graphfLeg}, we introduce the following notion.

\begin{definition}[Legendrian submersions]\label{def-legsub}Let $(Y^{2n+1}, \xi)$ be a contact manifold, and $B^{n+1}$ be a smooth manifold. A smooth submersion $\pi : Y \to B$ is called a \emph{Legendrian submersion} if $\ker d\pi \subseteq \xi$. Equivalently, the fibers of $\pi$ are Legendrian submanifolds of $(Y, \xi)$.\end{definition}

\begin{example}\label{eg-frprlegsub}The front projection $\pi_{\rm{front}} : J^1 \Lambda \to \Lambda \times \Bbb R$ is an example of a Legendrian submersion, since the fibers of $\pi_{\rm{front}}$ are given by 
$$\{z = \mathrm{const}., q_i = \mathrm{const}.\},$$
along which the contact form $dz - \sum_i p_i dq_i$ vanishes. \end{example}

\begin{definition}\label{def-legsubrel}Let $(Y^{2n+1}, \xi)$ be a contact manifold, and $B^{n+1}$ be a smooth manifold. We define the \emph{differential relation of Legendrian submanifolds} $\mathcal{R}_{\rm{sLeg}} \subset J^1(Y, B)$, consisting of $1$--jets of maps $f : Y \to B$ which are Legendrian submersions.\end{definition}

It is clear from definition that $\mathcal{R}_{\rm{sLeg}}$ is a locally integrable differential relation (Definition \ref{def-locintrel}). We will next show that it is also microflexible (Definition \ref{def-micflexrel}). We begin with some lemmas.

\begin{lemma}\label{lem-microimhtpysub} Let $M, N$ be smooth manifolds and $\{f_s : M \to N : s \in I\}$ be a homotopy of submersions onto image. Suppose there exists a compact subset $K \subset M$ such that $f_s = f_0$ on $M \setminus K$. Then, there exists an $\varepsilon > 0$ such that for all $s \in [0, \varepsilon]$, $f_s(M) = f_0(M)$. \end{lemma}

\begin{proof}For convenience of notation, let us denote $M' := f_0(M)$ and $K' := f_0(K)$. Let $H : M \times I \to N$ define the homotopy map, given by $H(x, s) = f_s(x)$. As $f_0$ is a submersion, and therefore an open map, the image $M' = f_0(M) \subset N$ is open. Thus, $H^{-1}(M') \subset M \times I$ is an open subset, and $K \times \{0\} \subset H^{-1}(M')$. By compactness of $K$, we may find $\varepsilon_1 > 0$ such that $K \times [0, \varepsilon_1] \subset H^{-1}(f_0(M))$. On the other hand, $(M \setminus K) \times I \subset H^{-1}(U')$ as $H(x, s) = f_s(x) = f_0(x) \in M'$ for all $x \in M \setminus K$. Therefore, $M \times [0, \varepsilon_1] \subset H^{-1}(M')$. Thus, $f_s(M) \subset M'$ for all $s \in [0, \varepsilon_1]$.

For the reverse inclusion, observe that for any $y = f_0(x) \in M'$, $H^{-1}(y) \subset M \times I$ is a submanifold of dimension $n+1$. A connected component of $H^{-1}(y)$ containing $x$ must project under $M \times I \to I$ to a connected subset of $I$ containing $0$. If the image under projection is only $\{0\}$, then $H^{-1}(x) = f_0^{-1}(x)$ which is a contradiction as $f_0^{-1}(x)$ is a manifold of dimension $n$. Therefore, the projection must contain $[0, \delta(x)]$ for some $\delta(x) > 0$ depending on $x \in U'$. Note that this says $U' \subset f_s(U)$ for all $s \in [0, \delta(x)]$. We may choose $\delta : U' \to (0, 1]$ to be a continuous function, and hence it must admit a positive minimum on the compact subset $K' \subset U'$. Let us denote the minimum value by $\varepsilon_2 > 0$. Then $U' \subset f_s(U')$ for all $s \in [0, \varepsilon_2]$. Taking $\varepsilon := \mathrm{min}\{\varepsilon_1, \varepsilon_2\}$ proves the claim.\end{proof}

\begin{lemma}\label{lem-loctrivlegs}Any Legendrian submersion is contactomorphic to the Euclidean front projection $\pi_{\rm{front}} : J^1\Bbb R^n \to \Bbb R^n \times \Bbb R$ (cf. Example \ref{eg-frprlegsub}) in a fiber-preserving fashion. \end{lemma}

\begin{proof}Let $f : (Y^{2n+1}, \xi) \to B^{n+1}$ be a Legendrian submersion. We choose an almost-complex structure $J$ on the contact distribution $\xi \subset TY$ compatible with the fiberwise symplectic form $\mathrm{CS}(\xi)$. Let us choose local coordinates $(x^1, \cdots, x^n)$ for the fibers of $f$. We define $y^i := J \circ x^i$ for $1 \leq i \leq n$. As $\ker df \subset \xi$ is a fiberwise Lagrangian subbundle, by Lemma \ref{lem-lagcomp}, we obtain that $\xi$ is fiberwise spanned by the vector fields $\partial/\partial x_i, \partial/\partial y_i$, $1 \leq i \leq n$. We define $z : Y \to \Bbb R$ to be the coordinate given by flowing along a Reeb vector field (Definition \ref{def-reeb}) for $(Y, \xi)$. Then, $(x^1, \cdots, x^n, y^1, \cdots, y^n, z)$ is a local coordinate system for $Y$, and as $f : Y \to B$ is a submersion, the vector fields $f_*(\partial/\partial y_i)$ for $1 \leq i \leq n$ and $f_*(\partial/\partial z)$ fiberwise span $TB$. Thus, $(x^1, \cdots, x^n, y^1, \cdots, y^n, z)$ is an adapted coordinate chart for $f : Y \to B$. Moreover, the contact form on $Y$ is given in these coordinates as 
$$\alpha = dz - \sum_{i = 1}^n y^i dx^i.$$
This proves the lemma.\end{proof}

\begin{proposition}\label{prop-slegmflex}$\mathcal{R}_{\rm{sLeg}}$ is a microflexible differential relation.\end{proposition}

\begin{proof}Let $U \subset Y$ be an open subset and $\{f_s : U \to B : s \in I\}$ be a homotopy of Legendrian submersions such that $f_s = f_0$ outside a compact set $K \subset U$. Let us denote $U' := f_0(U)$ and $K' := f_0(K)$. By Lemma \ref{lem-microimhtpysub}, we may choose $\varepsilon > 0$ such that $f_s(U) = U'$ for all $\varepsilon > 0$. Thus, $\{f_s : U \to U' : s \in [0, \varepsilon]\}$ is a homotopy through surjective submersions. Let $F : U \times [0, \varepsilon] \to U' \times [0, \varepsilon]$ denote the movie map, given by $F(u, s) = (f_s(u), s)$. Then $F$ must also be a surjective submersion. Let $X$ denote a vector field on $U \times [0, \varepsilon]$ given by a lift of the vector field $\partial/\partial s$ on $U' \times [0, \varepsilon]$, i.e., $F_*(X) = \partial/\partial s$. We extend $X$ to a vector field on $U \times I$, and continue to denote it as $X$ by a slight abuse of notation. Note that the vector field $X$ is zero outside $K \times [0, \varepsilon]$, hence it is complete. If $\phi_s$ denotes the flow of $X$, then by construction $f_s = f_0 \circ \phi_s$. 

By Lemma \ref{lem-loctrivlegs}, we may choose a Darboux chart $U \subset (Y, \xi = \ker \alpha)$ with coordinates $(x, y, z)$ where $x = (x_1, \cdots, x_n)$ and $y = (y_1, \cdots, y_n)$, so that 
$$\alpha = dz - \sum_i y_i dx_i,$$
and $f_0(x, y, z) = (x, z)$. Note that $\phi_t^*(\ker df_0) = \ker df_s \subseteq \xi$, and $\ker df_0$ is the tangent distribution to $\{y = \mathrm{const.}\}$. Thus, the form $\phi_s^*\alpha$ must vanish on $\{y = \mathrm{const.}\}$. Hence,
$$\phi_s^*\alpha = a_0(x, y, z) dz - \sum_i a_i(x, y, z) dx_i$$
for some smooth functions $a_i : U \to \Bbb R$, such that $a_i(0, 0, 0) = 1$. If $\varepsilon > 0$ is sufficiently small, then $a_0$ is $C^\infty$--close to the constant function $1$ on $U$. Dividing by $a_0$, we get
$$\phi_t^*\alpha = a_0(x, y, z) \beta, \;\; \beta := dz - \sum_i \frac{a_i(x,y,z)}{a_0(x,y,z)} dx_i.$$
The contactness condition $\beta \wedge (d\beta)^{\wedge n} \neq 0$ implies that the determinant of the matrix $(\partial/\partial y_i(a_j/a_0))_{i, j}$ is nonvanishing. Thus, by the inverse function theorem, we may find a change of coordinates on $U$ preserving the fibers $\{f_0 = \mathrm{const}.\}$ such that in the new coordinares, $a_i(x, y, z)/a_0(x, y, z) = y_i, 1 \leq i \leq n$. Thus, $\phi_s^*\alpha = a_0 \alpha$. Consequently, $\phi_s$ is a contact isotopy of $U$ such that $f_s = f_0 \circ \phi_s$ for all $s \in [0, \varepsilon]$. We may multiply by an appropriate cut-off function to extend $\phi_s$ to a global isotopy of $Y$.

Let $H \subset Y$ be an embedded handle with core $C \subset H$, $f : \mathrm{Op}(H) \to B$ be a Legendrian submersion and $f_s : \mathrm{Op}(\partial H \cup C) \to B$ be an isotopy through Legendrian submersions fixed near $\mathrm{Op}(\partial H)$ such that $f_0 = f$. We may now cover $H$ by Darboux charts and use the argument above to realize the isotopy $f_s$ by a contact isotopy $\phi_s$ of $Y$. Then the required microextension is now defined by letting $\widetilde{f}_s :=  f_0 \circ \phi_s$, $s \in [0, \varepsilon]$.\end{proof}

Though formal Legendrian embeddings are in general far from being a (holonomic) Legendrian embeddings, we shall eventually show that they are nonetheless approximable by embeddings which are \emph{close} to being a Legendrian embedding. We introduce the following definition which formalizes this notion.

\begin{definition}[$\varepsilon$--Legendrian embedding]\label{defn-epsLeg}Given $\varepsilon > 0$ and an ambient Riemannian metric on $Y$, an embedding $f : \Lambda^n \to (Y^{2n+1}, \xi)$ is said to be \emph{$\varepsilon$--Legendrian} if for any $p \in \Lambda$ and $q = f(p) \in Y$, there exists a Legendrian subspace $L_q \subset \xi_q$ such that $f_*(T_p\Lambda) \subset T_q Y$ makes an angle of size at most $\varepsilon$ with $L$.\end{definition}

The following lemma can be intuited as a version of the Weinstein tubular neighborhood theorem for $\varepsilon$--Legendrian embeddings.

\begin{lemma}\label{lem-epsLagnc}Let $f : \Lambda \to (Y, \xi)$ be a $\varepsilon$--Legendrian embedding. Then we can find a bundle monomorphism $P : T^*\Lambda \oplus \Bbb R \to f^* TY$, such that
\begin{enumerate}[label=(\arabic*), font=\normalfont]
\item $P(T^*\Lambda \oplus \Bbb R)$ is fiberwise transverse to $T\Lambda \subseteq f^* TV$.
\item $P(T^*\Lambda \oplus \{0\})$ is a Legendrian plane in $f^* \xi$.
\item $P(\{0\} \oplus \Bbb R)$ is fiberwise transverse to $f^* \xi$.
\end{enumerate}
\end{lemma}

\begin{proof}Choose a contact form $\alpha$ such that $\xi = \ker \alpha$, an ambient Riemannian metric on $Y$ and an almost complex structure $J$ on $\xi$ compatible with $\mathrm{CS}(\xi)$. For every $x \in \Lambda$, choose a Legendrian subspace $L_x \subset \xi_{f(x)}$ which makes an angle of size at most $\varepsilon$ with $f_*(T_x\Lambda)$. Since the space of $\varepsilon$--Legendrian subspaces deformation retracts to the space of Legendrian subspaces, this choice can be made to vary smoothly with $x \in \Lambda$. Let $\pi_x : f_*(T\Lambda_x) \to L_x$ be the orthogonal projection with respect to the ambient metric, which is an isomorphism in virtue of the $\varepsilon$--Legendrian condition. We define $P|_{T^*\Lambda \oplus \{0\}} := J \circ \pi \circ df$, and $P|_{\{0\} \oplus \Bbb R}$ as scalar multiplication by a vector field transverse to $\xi$. This defines a bundle monomorphism $P : T^*\Lambda \oplus \Bbb R \to f^*TY$ satisfying all the required properties.\end{proof}

Let $\mathcal{R}_{\varepsilon\rm{Leg}} \subset J^1(\Lambda, Y)$ be the differential relation of $1$--jets of $\varepsilon$--Legendrian embeddings $f : \Lambda \to (Y, \xi)$. The key ingredient required to approximate a formal Legendrian embedding by $\varepsilon$--Legendrian embeddings is the following proposition.

\begin{proposition}\label{prop-epsLeg}$\mathcal{R}_{\varepsilon\rm{Leg}}$ is an ample differential relation (Definition \ref{def-amplerel}).\end{proposition}

\begin{proof}Let $y \in Y$ be a point, and $S \subset \xi_y \subset T_y Y$ be an isotropic subspace of corank $1$. Then,
$$\Omega_S = \{v \in T_y Y : \mathrm{span}(S, z) \subset T_y Y \text{ is Legendrian}\}$$
is equal to the symplectic complement $S^{\perp\omega}$ of $S$ in $\xi$ with respect to the conformal symplectic structure $\mathrm{CS}(\xi)$. Let $S^{\perp}$ denote the \emph{orthocomplement} of $S$ in $TY$ with respect to some ambient Riemannian metric. Then,
\begin{enumerate}
\item $\dim S^{\perp} = (2n+1) - (n-1) = n+2$, $\dim S^{\perp\omega} = n$ and $S^{\perp} + S^{\perp\omega} \supseteq S^\perp + S = T_y Y$. Thus, $S^{\perp}$ and $S^{\perp\omega}$ are transverse in $T_y Y$ and $S^\perp \cap S^{\perp\omega}$ is rank $2$.
\item Let $\pi : S^\perp \to S^{\perp} \cap  S^{\perp \omega}$ be the orthogonal projection. If the angle between $v \in S^\perp$ and $\pi(v)$ is at most $\varepsilon$, then $\mathrm{span}(S, v)$ and $\mathrm{span}(S, \pi(v))$ make an angle of at most $\varepsilon$ as well, therefore $\mathrm{span}(S, v)$ is $\varepsilon$--Legendrian. Consequently, $v \in \Omega_S$.
\end{enumerate}
Therefore, $\Omega_S \cap S^\perp$ contains a solid $\varepsilon$--cone around the rank $2$ subspace $S^{\perp\omega} \cap S^\perp$. Since a solid cone about a rank $2$ subspace of a vector space is an ample subset, $\Omega_S$ is also ample. This finishes the proof in case $S$ is isotropic. If not, we have two cases:
\begin{enumerate}
\item $S$ makes an angle of at most $\varepsilon$ with an isotropic subspace, in which we can repeat the above with the isotropic subspace nearby $S$.
\item $S$ does not make an angle of at most $\varepsilon$ with any isotropic subspace, in which case $\Omega_S$ is empty hence vacuously ample.\qedhere
\end{enumerate}
\end{proof}

We now come to the proof of Theorem \ref{thm-graphfLeg}.

\begin{proof}[Proof of Theorem \ref{thm-graphfLeg}]

Let us start with a formal Legendrian embedding,
$$(F_s, f) : (T\Lambda, \Lambda) \to (TY, Y).$$ 
We proceed along the following steps: 
\vspace{5pt}
\begin{enumerate}[wide, labelwidth=!, labelindent=0pt, itemsep=5pt]
\item[\emph{Approximate by $\varepsilon$--Legendrian embedding}:] By Proposition \ref{prop-epsLeg} and Theorem \ref{thm-hample}, we $C^0$--approximate $f$ by a $\varepsilon$--Legendrian embedding $\widetilde{f} : \Lambda \to Y$.

\item[\emph{Microextend to Legendrian submersion}:] Let $W = \Lambda \times (-1, 1)$ and let  $U \subset Y$ be a smooth tubular neigborhood of $\widetilde{f}(\Lambda)$. We construct a \emph{formal Legendrian submersion of $U$ in $W$}, i.e., a formal section of the differential relation $\mathcal{R}_{\rm{sLeg}} \subset J^1(U, W)$ over $U$. Let $g : U \to W$ denote the tubular projection to $\Lambda \times \{0\} \subset W$. By Lemma \ref{lem-epsLagnc}, we obtain a bundle monormorphism,
\begin{gather*}
P : T^*\Lambda \oplus \Bbb R \to \widetilde{f}^*TU, \text{ such that:}\\
\widetilde{f}^*TU \cong T\Lambda \oplus P(T^*\Lambda \oplus \Bbb R) = T\Lambda \oplus P(T^*\Lambda \oplus 0) \oplus P(0 \oplus \Bbb R).
\end{gather*}
We define $G : TU \to TW$ by defining $G$ on each of the three factors above:

\begin{enumerate}[label=({\roman*}), leftmargin=2cm] 
\item Let $G$ send $T\Lambda$ fiberwise isomorphically to $T\Lambda \times 0 \subset TW$.
\item Let $G$ be fiberwise zero on $P(T^*\Lambda \oplus 0)$.
\item Let $G$ send $P(\{0\} \oplus \Bbb R)$ fiberwise isomorphically to $0 \times T(-1, 1) \subset TW$.
\end{enumerate}

\noindent
But this merely defines $G$ on $\widetilde{f}^*TU \cong TU|_{\widetilde{f}(\Lambda)}$. We then extend $G$ to all of $TU$ by choice of an Ehresmann connection compatible with the contact structure $\xi|_U$. By construction, the fiber-preserving morphism $G : TU \to TW$ covers $g : U \to W$. Therefore, the pair of maps
$$(G, g) : (TU, U) \to (TW, W),$$ 
defines a formal section of the jet bundle $J^1(U, W)$ over $U$, by the discussion in Example \ref{eg-mapjet}. Moreover, $G : TU \to TW$ sends each fiber isomorphic to $T_x \Lambda \oplus T_x^* \Lambda \oplus \Bbb R$ to $T_x\Lambda \oplus \Bbb R$. Thus, $(G, g)$ in fact defines a section of $\mathcal{R}_{\rm{sLeg}}$. 

\item[\emph{Holonomic approximation of Legendrian submersion}:] Let $R$ be the Reeb vector field (Definition \ref{def-reeb}) on $(Y, \xi)$ and let $X \subset (U, \xi|_U)$ be the union of all Reeb trajectories passing through $\widetilde{f}(\Lambda)$. Then $f(\Lambda) \subseteq X$ is codimension $1$, and $\widetilde{f}$ defines a formal section of $\mathcal{R}_{\rm{sLeg}}$ over $X$. By Proposition \ref{prop-slegmflex}, Theorem \ref{thm-holapp} and using the fact that the relation $\mathcal{R}_{\rm{sLeg}}$ is invariant under the capacious subgroup $\mathrm{Cont}_0(Y) \subset \mathrm{Diff}(Y)$ of contactomorphism isotopic to identity, we find a $C^0$--small isotopy from $f : \Lambda \to W$ to $\widetilde{f} : \Lambda \to W$ such that $\widetilde{f}(\Lambda) \subset R$, as well as a (holonomic) Legendrian submersion $\pi : U \to \Lambda \times (0, 1)$ from a neighborhood of $\widetilde{f}(\Lambda) \subset W$. 

\item[\emph{Constructing graphical coordinates}:]
In the proof of the holonomic approximation theorem (Theorem \ref{thm-holapp}), the diffeotopy $h_{s, t}|_K$ is graphical with respect to $K$, where $K \subset B$ is the codimension $1$ subcomplex (see, \cite[Theorem 13.4.1]{embook} for details). Therefore, as $f(\Lambda) \subset Y$ is transverse to the Reeb trajectories, the same is true for $\widetilde{f}(\Lambda) \subset Y$, inspite of the maps $f$ and $\widetilde{f}$ being only $C^0$--close. With this remark in mind, let $(x_1, \cdots, x_n)$ be local coordinates on $\Lambda$, $(y_1, \cdots, y_n)$ be local coordinates on the Legendrian fibers of $\pi$ and $z$ be the Reeb coordinate. Then, $(x, y, z)$ jointly forms local coordinates for the tubular neighborhood $U$. If $\alpha$ is the contact form on $Y$, we have $\alpha(\partial/\partial z) = 1$ and $\alpha(\partial/\partial y_i) = 0$. Therefore, $\alpha = dz - \sum \widetilde{y}_i dx_i$ for some smooth functions $h_i : U \to \Bbb R$. Since $\alpha$ is a contact form, the determinant $\mathrm{det}((\partial h_i/\partial y_j)_{i,j})$ is nonvanishing. Hence, the collection $(h_1, \cdots, h_n)$ forms a local coordinate system for the Legendrian fibers of $\pi$. Thus, we obtain the required contactomorphism $\varphi : U \to J^1 \Lambda$.\qedhere
\end{enumerate}
\end{proof}

\section{The $h$--principle for wrinkled Legendrian embeddings}\label{sec-wrinkleg}

In this section, we attempt to produce an $h$--principle for Legendrian embeddings by combining ideas from Section \ref{sec-hLegimm} and Section \ref{sec-wrinkemb}. The basic principle goes all the way back to Section \ref{sec-legknots}, where we demonstrated a process in Theorem \ref{thm-hprinclegknots} by which one can approximate any formal slope-field on a front diagram by interpolating zig-zags to obtain a holonomic Legendrian embedding nearby a formal one. In this section, we address a higher dimension version of this process, where zig-zags will be replaced by wrinkle singularities from Definition \ref{def-wrinkledemb}. However, this process will only give rise to topological embeddings which are Legendrian away from some mild singularities, as we shall see below.

We begin by recalling Example \ref{eg-1jet}. Let $\Lambda^n$ be a smooth manifold, $J^1 \Lambda = \Bbb R \times T^* \Lambda$ be the $1$--jet space, and $\pi : J^1 \Lambda \to \Bbb R \times \Lambda$ denote the front projection. We may identify $J^1 \Lambda$ with the subspace of $\mathrm{Gr}_n(\Bbb R \times \Lambda)$ consisting of the \emph{non-vertical planes}, i.e., the rank $n$ planes which are transverse to the rank $1$ subbundle $T \Bbb R \times 0 \subset T(\Bbb R \times \Lambda)$.

Let $f : M^n \to \Bbb R \times \Lambda$ be a smooth map which is an immersion on some dense open set $U \subseteq M$. Suppose that the tangential Gauss map $G(df) : U \to \mathrm{Gr}_n(\Lambda \times \Bbb R)$ is non-vertical at every point of $U$. Then the smooth map $(f, G(df)) : U \to J^1 \Lambda$ is a Legendrian immersion covering $f|U$, which may be thought as a higher dimensional analogue of the phenomenon in Proposition \ref{prop-legfr}. Indeed, $(f, G(df))$ can be rewritten as simply $j^1 f_0$ where $f_0 : M \to \Lambda$ is given by $f_0 := \mathrm{proj} \circ f$ where $\mathrm{proj} : \Bbb R \times \Lambda \to \Lambda$ is projection to the second factor. The fact that it is a Legendrian immersion now follows from the discussions in  Example \ref{eg-1jet} and Example \ref{eg-phaselag}.

 If, moreover, $f$ is an immersion and $G(df)$ extends to a smooth map $G : M \to \mathrm{Gr}_n(\Lambda \times \Bbb R)$ covering $f$ which is non-vertical at every point, then $(f, G) : M \to J^1 \Lambda$ is a Legendrian immersion covering $f$ \emph{provided} that $f$ is an immersion on all of $M$. However, such an immersive Legendrian lift does not exist if $f$ is assumed to be immersive only on $U \subseteq M$.

\begin{example}\label{eg-wrinklelift}Let us consider the standard wrinkle $f : \Bbb R^2 \to \Bbb R^3$, 
$$f(u, v) = \left (u, v^3 - 3(1 - u^2) v, \frac15 v^5 - \frac23 (1 - u^2) v^3 + (1 - u^2)^2 v \right )$$
Let us use coordinates $(x_1, x_2, z)$ on $\Bbb R^3$ and $(x_1, x_2, y_1, y_2, z)$ on $J^1 \Bbb R^3 = \Bbb R^5$, so that the projection $\pi : \Bbb R^5 \to \Bbb R^3$ is the Legendrian front projection with respect to the standard contact structure $\alpha_{\mathrm{std}} = dz - y_1 dx_1 - y_2 dx_2$ on $J^1\Bbb R^2 = \Bbb R^5$. We compute the Jacobian of $f$:
$$Df(u, v) = \begin{pmatrix}1 & 6uv & 4 u v^3/3 - 4(1 - u^2)uv \\ 0 & 3v^2 - 3(1 - u^2)& v^4 - 2(1 - u^2)v^2 + (1 - u^2)^2 \end{pmatrix}$$
To obtain a Legendrian lift of $f$ to $g : \Bbb R^2 \to \Bbb R^5$, we must solve the equation,
$$\begin{pmatrix}1 & 6uv \\ 0 & 3v^2 - 3(1 - u^2)\end{pmatrix} \begin{pmatrix}y_1 \\ y_2\end{pmatrix} = \begin{pmatrix} \frac{4}{3} u v^3 - 4(1 - u^2)uv \\ v^4 - 2(1 - u^2)v^2 + (1 - u^2)^2 \end{pmatrix}$$
Formally inverting the matrix, we find the solution,
$$(y_1, y_2) = \frac{1}{3}(2 u v(3u^2 - v^2 - 3), u^2 + v^2 - 1)$$
Therefore, we find the candidate Legendrian lift $g : \Bbb R^2 \to \Bbb R^5$ where $g(u, v)$ is
$$\left (u, v^3 - 3(1 - u^2) v, \frac15 v^5 - \frac23 (1 - u^2) v^3 + (1 - u^2)^2 v, \frac23 u v(3u^2 - v^2 - 3), \frac13 (u^2 + v^2 - 1) \right )$$
Observe that the Jacobian $Dg(u, v)$ is
$$\begin{pmatrix}1 & 6uv & 4 u v^3/3 - 4(1 - u^2)uv & v (6 u^2 - 2v^2/3 - 2) & 2u/3 \\ 0 & 3v^2 - 3(1 - u^2)& v^4 - 2(1 - u^2)v^2 + (1 - u^2)^2 & 2u (u^2 - v^2 - 1) & 2v/3 \end{pmatrix}$$
Thus, $Dg$ has rank $1$ on the equator $\{(\pm 1, 0)\} \subset S^1 \subset \Bbb R^2$ of the singular locus of the wrinkle. \end{example}

The takeaways from this discussion are:
\begin{enumerate}
\item Even though the standard wrinkle is not immersive along a sphere $S^{n-1}$, it has a well-defined Gauss map and therefore admits a smooth lift which is Legendrian away from the equator $S^{n-2} \subset S^{n-1}$.
\item However, lift is not immersive at the swallowtail locus given by the equator $S^{n-2} \subseteq S^{n-1}$, where the rank drops by $1$.
\end{enumerate}
However, the use of such singular Legendrian lifts of wrinkled embeddings will be ubiquitous in what follows, so we record this notion as a definition.

\begin{definition}\cite[Definition 3.3]{mur}\label{def-wrinkleg} Let $\Lambda^n$ be a smooth manifold and $(Y^{2n+1}, \xi)$ be a contact manifold. A \emph{wrinkled Legendrian embedding} is a smooth map $f : \Lambda \to Y$ together with a collection of (not necessarily disjoint) Darboux charts $\{W_j \cong \Bbb R^{2n+1}\}$ in $Y$, such that
\begin{enumerate}
\item $f$ is a topological embedding,
\item $f$ is an immersion outside a subset of codimension $2$ that is a union of embedded spheres $S^{n-2}_j$, each of which is contained in some $W_j \subset Y$ such that $f^{-1}(W_j)$ is diffeomorphic to $\Bbb R^n$,
\item Let $\pi_{\rm{front}} : J^1\Bbb R^n = \Bbb R^{2n+1} \to \Bbb R^{n+1}$ denote the front projection. Then, the image of $f|f^{-1}(W_j)$ under the front projection,
$$f_j := \pi_{\rm{front}} \circ f|_{f^{-1}(W_j)} : \Bbb R^n \to \Bbb R^{n+1},$$
is a wrinkled embedding with swallowtail locus containing $S^{n-2}_j$.
\end{enumerate}

\noindent 
A \emph{family of wrinkled Legendrians} $\{f_s : \Lambda \to Y : s \in I^d\}$ is a continuous homotopy of smooth maps where for each $s \in I^d$, $f_s$ is a wrinkled Legendrian embedding and $F : \Lambda \times I \to Y$ is smooth except on (singular) Legendrian lifts of embryo singularities. 

We also demand that the associated collection of Darboux charts in the collection $\{W_j\}$ vary continuously throughout the homotopy containing their respective wrinkle singularities $S_j^{n-2}$ throughout their individual lifetime, being allowed only to continuously appear or disappear at the embryo singularities.
\end{definition}

Given a wrinkled Legendrian embedding $f : \Lambda \to (Y, \xi)$, we shall construct a formal Legendrian embedding corresponding to it in a nearly canonical way, depending only on a contractible space of choices. Let us first prove a simple lemma. We fix notation by letting $f_0 : \Bbb R^n \to \Bbb R^{n+1}$ be a model wrinkle Defintion \ref{def-wrinkledemb} (see also, Figure \ref{fig-wrinkle}), such that
$$f_0(x, t) = \left (x, t^3 - 3(1-|x|^2)t, \int_0^t (s^2 - \delta)^2 ds \right), \; (x, t) \in \mathrm{Op}_{\Bbb R^n}(S^{n-1}),$$
where $(x, t)$ denotes the spherical coordinates on $\mathrm{Op}_{\Bbb R^n}(S^{n-1}) \cong S^{n-1} \times \Bbb R$.

\begin{lemma}\label{lem-modelreg}Let us denote $f : \Bbb R^n \to \Bbb R^{2n+1}$ to be the Legendrian wrinkle obtained from lifting $f_0$ along the front projection $\pi_{\rm{front}} : \Bbb R^{2n+1} = J^1\Bbb R^n \to \Bbb R^{n+1}$. There exists a $C^\infty$--small isotopy from $f$ to a smooth embedding $\widetilde{f} : \Bbb R^n \to \Bbb R^{2n+1}$, and a formal Legendrian embedding 
$$(F_s, \widetilde{f}) : (T\Bbb R^n, \Bbb R^n) \to (T\Bbb R^{2n+1}, \Bbb R^{2n+1})$$
such that $F_0 = d\widetilde{f}$. \end{lemma}

\begin{proof}Let $f_x, f_y : \Bbb R^n \to \Bbb R^n$ denote the $x$--coordinate and $y$--coordinate of $f$, respectively, and $f_z : \Bbb R^n \to \Bbb R$ be the $z$--coordinate. We modify the front projection $f_0 = (f_x, f_z) : \Bbb R^n \to \Bbb R^{n+1}$ by rounding out the cuspidal edges and swallowtail singularities using an appropriate smoothing function. Let $\widetilde{f}_0 = (\widetilde{f}_x, \widetilde{f}_z) : \Bbb R^n \to \Bbb R^{n+1}$ be the resulting map. Let us define,
$$\widetilde{f} := (\widetilde{f}_x, f_y, \widetilde{f}_z) : \Bbb R^n \to \Bbb R^{2n+1}.$$
We rotate the graph of $\widetilde{f}$ while fixing the coordinate $f_y$ by a $C^\infty$--small amount, to ensure $\widetilde{f}$ is transverse to the vector field $\partial/\partial z$. Let us continue to call the resulting map as $\widetilde{f}$ by a slight abuse of notation.

Thus, the image of the tangent spaces under $d\widetilde{f} : T\Bbb R^n \to T\Bbb R^{2n+1}$ are everywhere transverse to the planes spanned by $\partial/\partial y_1, \cdots, \partial/\partial y_n, \partial/\partial z$. As every rank $n$ subspace of $T_0\Bbb R^{n+1} \cong \Bbb R^{n+1}$ transverse to these coordinates is graph of a linear map $\Bbb R^n \to \Bbb R^{n+1}$, we deduce that the Gauss map $G(\widetilde{f})$ has image contained in $M_{n, n+1} \subset \mathrm{Gr}_n(\Bbb R^{n+1})$, where $M_{n, n+1}$ consist of subspaces parametrized by $n \times (n+1)$ matrices. Among these, the Lagrangian subspaces of $\mathrm{CS}(\xi)_0 = \{\partial/\partial z = 0\}$ are parametrized by the space $S_n$ of symmetric $n \times n$ matrices. As $M_{n, n+1}, S_n$ are both contractible, the former deformation retracts to the latter. Thus, we obtain the required canonical homotopy $\{F_s\}$ of bundle monomorphisms.
\end{proof}

By globalizing the model case in Lemma \ref{lem-modelreg}, we make the following definition.

\begin{definition}[Regularization]\label{def-reg} Let $f : \Lambda \to (Y, \xi)$ be a wrinkled Legendrian embedding. On every chart $W_j \subset Y$ containing a wrinkle, we modify $f$ using Lemma \ref{lem-modelreg} to obtain a smooth embedding $\widetilde{f} : \Lambda \to (Y, \xi)$ which is $C^\infty$--close to $f$, and a formal Legendrian embedding $(F_s, \widetilde{f}) : (T\Lambda, \Lambda) \to (TY, Y)$. We shall call the formal Legendrian embedding $(F_s, \widetilde{f})$ as the \emph{regularization} of the wrinkled Legendrian embedding $f$. \end{definition}

\begin{remark}The regularization map from the space of $d$--parametric family of wrinkled embeddings to the space of $d$--parametric family of formal Legendrian embeddings furnished by Definition \ref{def-reg} is canonically defined upto a contractible space of choices, for any $d$.\end{remark}

We state and prove below an $h$--principle for wrinkled Legendrian embeddings (see, \cite[Proposition 3.4]{mur}). 

\begin{theorem}[$h$--principle for wrinkled Legendrian embeddings]\label{thm-hwrinkleg} Let $(Y^{2n+1}, \xi)$ be a contact manifold, and $\Lambda^n$ be a smooth manifold. Let $(F_s, f) : (T\Lambda, \Lambda) \to (TY, Y)$ be a formal Legendrian embedding. Then for any $\varepsilon > 0$, there exists a homotopy 
$$\{(F_{s, t}, f_t) : (T\Lambda, \Lambda) \to (TY, Y) : t \in I\}$$
through formal Legendrian embeddings such that $F_{s, 0} = F_s$, $f_0 = f$, and $f_1 : \Lambda \to Y$ is a wrinkled Legendrian embedding. Moreover, the result is also true parametrically and relatively.\end{theorem}

\begin{proof}By Theorem \ref{thm-graphfLeg}, after an isotopy we can find a neighborhood $U$ of $f(\Lambda)$ contactomorphic to $J^1 \Lambda$ such that $f$ is graphical with respect to the front projection $\pi_{\rm{front}} : J^1\Lambda \to \Lambda \times \Bbb R$. By projecting to the front, we obtain a smooth embedding $\pi_{\rm{front}} \circ f : \Lambda \to \Lambda \times \Bbb R$ as well as a tangential homotopy $\{G_t : \Lambda \to \mathrm{Gr}_n(\Lambda \times \Bbb R) : t \in I\}$ of the Gauss map of $\pi_{\rm{front}} \circ f$. By Theorem \ref{thm-hwrink}, we can find a homotopy of wrinkled embeddings $\{g_t : \Lambda \to \Lambda \times \Bbb R : t \in I\}$ such that $G(dg_t)$ is $C^0$--close to $G_t$. We conclude the proof by lifting $g_t$ to wrinkled Legendrian embeddings $\widetilde{g}_t : \Lambda \to J^1 \Lambda$ and regularizing (Definition \ref{def-reg}). \end{proof}

\section{Loose Legendrian embeddings and removal of wrinkles}\label{sec-loose}

In general, Theorem \ref{thm-hwrinkleg} cannot be improved. That is, the wrinkle singularities appearing in the singular Legendrian embeddings (obtained by an isotopy from the given formal Legendrian embedding) cannot be removed by a further Legendrian isotopy. In this section, we introduce a large class of Legendrian embeddings called \emph{loose Legendrians} where this removal is possible, as observed and demonstrated by Murphy \cite{mur}. 

\subsection{Definition and properties of loose charts} We shall define a loose Legendrian as a Legendrian submanifold which contain a very specific adapted chart, called a loose chart. These charts are high-dimensional generalizations of zigzags, which were discussed in Observation \ref{obs-intpzigzag} and the beginning of Section \ref{sec-wrinkemb}. Let us start with some preliminaries regarding these one-dimensional zigzags.

\begin{definition}[Legendrian zig-zag]\label{def-legzigzag}Consider the model zig-zag $\psi_\delta$ (cf. Figure \ref{fig-modelzig}), appropriately cut-off so that $\psi_\delta$ is compactly supported on an interval around the origin containing the semicubical cusp singularities. This defines a front diagram $Z \subset \Bbb R^2_{xz}$ satisfying the properties in Proposition \ref{prop-legfr}. Thus, it admits a lift to a Legendrian arc. We denote $\mathcal{Z} \subset (\Bbb R^3, \xi_{\rm{std}})$ to be this Legendrian lift and call it the \emph{Legendrian zig-zag}.\end{definition}

\begin{figure}[h]
\centering
\includegraphics[scale=0.55]{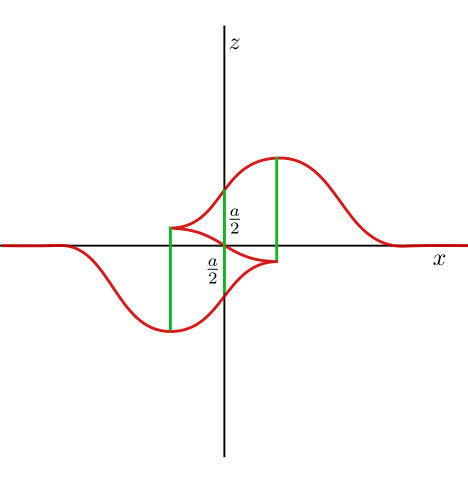}
\caption{A Legendrian zig-zag, with Reeb chords (green) of size $a$.}
\label{fig-legzig}
\end{figure}

\noindent
Note that for any Legendrian zig-zag, there must exist a pair of times $t, t' \in \Bbb R$ such that
\begin{enumerate}
\item The $x$--coordinates of $\psi_\delta(t)$ and $\psi_\delta(t')$ agree, and 
\item The slopes of the parametric curve given by $\psi_\delta$ at $\psi_\delta(t)$ and $\psi_\delta(t')$ agree. 
\end{enumerate}
Therefore, there must exist a flowline of the Reeb vector field $\partial/\partial z$ for $(\Bbb R^3, \xi_{\rm{std}})$ which intersects $\mathcal{Z}$ at a pair of points. In light of this, we give the following definition.

\begin{definition}A $z$--segment joining two points of $\mathcal{Z} \subset \Bbb R^3$ is called a \emph{Reeb chord} of $\mathcal{Z}$. The (Euclidean) length of the smallest Reeb chord is called the \emph{action} of $\mathcal{Z}$. We shall define $\mathcal{Z}_a$ to be the Legendrian zig-zag of action $a$ given by lifting the front diagram in Figure \ref{fig-legzig}. \end{definition}

\begin{remark}We have drawn $\mathcal{Z}_a \subset \Bbb R^2_{xz}$ so that there is a Reeb chord of length $a$ joining the heighest and the lowest points along every $\{x = \mathrm{const}.\}$ line intersecting $\mathcal{Z}_a$. Moreover, the Reeb chord passing through the origin has the origin as its midpoint. These arrangements are made precise for convenience of discussion later; the precise model of the zig-zag is mostly unimportant for validity of the results to come, as long as one fixes a choice beforehand.\end{remark}

Let $n \geq 2$. Let us write $\Bbb R^{2n+1} = \Bbb R^3 \times \Bbb R^{2n-2}$ and equip it with coordinates $(x, y, z, p, q)$ where $p = (p_1, \cdots, p_{n-1})$ and $q = (q_1, \cdots, q_{n-1})$. Let $\alpha_{\mathrm{std}} := \alpha_0 - \lambda$ be the standard contact form on $\Bbb R^{2n+1}$, where $\alpha_0 = dz - ydx$ is the standard contact form on $\Bbb R^3$ and $\lambda = \sum p_i dq_i$ is the tautological $1$--form (Example \ref{eg-phase}) on $\Bbb R^{2n-2} = T^* \Bbb R^{n-1}$. 

\begin{definition}[Loose charts]\label{def-loose}\leavevmode
\begin{enumerate}
\item $\mathcal{Z}_a \subset (\Bbb R^3, \alpha_0)$ be a Legendrian zig-zag of action $a$,
\item $C \subset \Bbb R^3$ be interior of a compact cube containing a Legendrian zig-zag $\mathcal{Z}_a$,
\item Let $B_\rho = \{(p, q) : |p| < \rho, |q| < \rho\} \subset \Bbb R^{2n-2}$, and
\item Let $J_\rho := \{p = 0\} \subset B_\rho$. 
\end{enumerate}
The pair $(C \times B_\rho, \mathcal{Z}_a \times J_\rho)$ is called a \emph{loose chart} if the \emph{size parameter} $\rho^2/a$ satisfies the quantitative constraint $\rho^2/ a > 1/2$.
\end{definition}

\begin{figure}[h]
\centering
\includegraphics[scale=0.2]{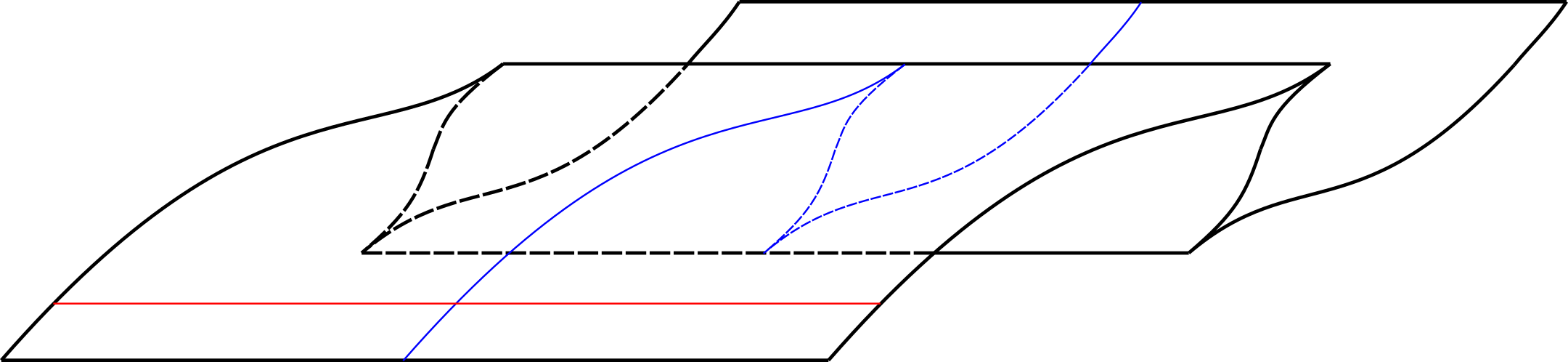}
\caption{A loose chart, with width (red) $2\rho$ and a vertical slice (blue) a zig-zag of action $a$. The condition $\rho^2/a > 1/2$ suggests the chart is sufficiently wide.}
\label{fig-loose}
\end{figure}

\begin{remark}[Size parameter is scale-invariant]\label{rmk-sizescale}Observe that for any constant $c > 0$, 
$$\varphi_c : \Bbb R^{2n+1} \to \Bbb R^{2n+1}, \;\; \varphi_c(x, y, z, p, q) = (cx, cy, c^2z, cp, cq),$$
is a contactomorphism of $(\Bbb R^{2n+1}, \xi_{\rm{std}})$ which scales $\alpha_{\rm{std}}$ by $c$. We call $\varphi_c$ the \emph{contact scaling} of $\Bbb R^{2n+1}$ by $c > 0$. If $(C \times B_\rho, \mathcal{Z}_a \times J_\rho)$ is as above (not necessarily a loose chart), then 
$$\varphi_c(C \times B_\rho, \mathcal{Z}_a \times J_\rho) = (c \cdot C \times B_{c\rho}, \mathcal{Z}_{c^2 a} \times J_{c\rho})$$
has size parameter $(c\rho)^2/(c^2 a) = \rho^2/a$ unchanged. In particular, loose charts are preserved under contact scaling.

For a Legendrian submanifold $\Lambda \subset \Bbb R^{2n+1}$, we shall use the convenient notation $c \cdot \Lambda := \varphi_c(\Lambda) \subset \Bbb R^{2n+1}$ to denote the contact scaling of $\Lambda$. \end{remark}

\begin{definition}[Loose Legendrian submanifolds]\label{def-looseLegsub}
A Legendrian submanifold $\Lambda \subset (V, \xi)$ in a contact manifold is called a \emph{loose Legendrian} if there is a Darboux chart $U \subset V$, a loose chart $(C \times B_\rho, \mathcal{Z}_a \times J_\rho)$ and a contactomorphism 
$$\varphi : U \to C \times B_\rho,$$
where $\varphi(U \cap \Lambda) = \mathcal{Z}_a \times J_\rho$. In this case, we shall say $(U, U \cap \Lambda)$ is a \emph{loose chart in $\Lambda$.}\end{definition}

We give a preliminary example of a loose Legendrian, obtained from spinning fronts of Legendrian knots containing Legendrian zig-zags. 

\begin{example}[Spun zig-zags]\label{eg-zigspin} Let $\Lambda_0 \subset (\Bbb R^3, \xi_{\mathrm{std}})$ be a Legendrian knot containing a Legendrian zig-zag as an arc. By spinning its front projection in $\Bbb R^2_{xz}$ along an appropriate axis, we produce a front diagram in $\Bbb R^3$. This has only cuspidal singularities, therefore admits a Legendrian lift $\Lambda \subset (\Bbb R^5, \xi_{\mathrm{std}})$. By construction, we have an open set $U \subset \Bbb R^5$ and a contactomorphism,
$$\varphi : (U, U \cap \Lambda) \stackrel{\cong}{\to} (C \times T^*S^1, \mathcal{Z}_a \times S^1),$$
where the contact structure on $C \times T^*S^1$ is given by the $1$--form $\alpha = dz - ydx - p_\theta d\theta$, $(x,y,z)$ being Euclidean coordinates on $C$ and $(\theta, p_\theta)$ being the generalized position-momentum coordinates for $T^*S^1$. We model $S^1 = \Bbb R/\Bbb Z$ as a circle of unit arclength. 

To find a loose chart, our idea is to take an arc $J_\rho \subset S^1$ of fixed size on the circle, and consider $\mathcal{Z}_a \times J_\rho$ inside $C \times T^*J_\rho$. The difficulty is that the action $a$ of the zig-zag may be large, which may cause the failure of this chart to be loose. We shall rememdy this by applying a contact isotopy which uniformly compresses the circle-parametric family of zigzags. To this end, apply the contact isotopy,
\begin{gather*}\psi : C \times T^*S^1 \to C \times T^*S^1,\\
\psi(x, y, z, \theta, p_\theta) = \left (x, y, \frac{z}{400a},\theta,\frac{p_\theta}{400a}\right ).
\end{gather*}
Notice $\psi$ takes $\mathcal{Z}_a \times S^1$ to $\mathcal{Z}_{1/400} \times S^1$. Let $\rho \in (1/20, 1)$ and $J_\rho \subset S^1$ be an arc of length $2\rho$. Then, $C \times B_{\rho}$ contactly embeds in $C \times T^*J_\rho \subset C \times T^*S^1$. Let $V = (\psi \circ \rho)^{-1}(C \times B_\rho)$. Therefore, we have a contactomorphism,
$$\psi \circ \rho : (V, V \cap \Lambda) \stackrel{\cong}{\to} (C \times B_\rho, \mathcal{Z}_{1/400} \times J_\rho).$$
This is the required loose chart, as the size parameter $\rho^2/(1/400) = (40\rho)^2 > 1 > 1/2$.
\end{example}

\begin{remark}\label{rmk-zigspin}The construction above generalizes to Legendrian knots in $(\Bbb R^{2n+1}, \xi_{\rm{std}})$ produced by iterated spinning constructions applied to a Legendrian front diagram in $\Bbb R^2$ containing a zig-zag. Instead of a circle-parametric family of zig-zags, this will produce a torus-parametric family of zig-zags. Hence, such a knot will contain a contactly embedded pair $(C \times T^* T^{n-1}, \mathcal{Z}_a \times T^{n-1})$. The argument for finding a loose chart in Example \ref{eg-zigspin} now goes through mutatis mutandis.\end{remark}

The following proposition is a generalization of Example \ref{eg-zigspin}, and gives a criterion for looseness of certain Legendrians submanifolds in $1$--jet bundles, which can be applied without needing to establish any quantiative estimates, as for instance in Definition \ref{def-loose}. In a nutshell, the result is that if one can isolate a Legendrian zig-zag in the front projection by a smoothly embedded transverse disk, then the Legendrian is automatically loose.

\begin{proposition}Let $Q^n$ be a smooth $n$--dimensional manifold with $n \geq 2$, and $\Lambda \subset (J^1 Q, \xi)$ be a Legendrian submanifold. Let $\pi : J^1 Q \to Q \times \Bbb R$ denote the front projection. Suppose there exists a smoothly embedded disk $D \subset Q \times \Bbb R$ such that
\begin{enumerate}[label=(\arabic*), font=\normalfont]
\item $D$ transversely intersects the front projection $\pi(\Lambda) \subset Q \times \Bbb R$, i.e. $D$ transversely intersects the smooth locus of $\pi(\Lambda)$ as well as the cuspidal locus of $\pi(\Lambda)$ and is disjoint from all other singularities of higher codimension,
\item There is a diffeomorphism of $D$ to the standard unit disk $D^2$ taking $D \cap \pi(\Lambda)$ to the front diagram $Z \subset D^2$ of a Legendrian zig-zag.
\end{enumerate}
Then, $\Lambda$ is a loose Legendrian.
\end{proposition}

\begin{proof}
First, we shall reduce to the case where the disk $D \subset Q \times \Bbb R$ is \emph{vertical}, i.e.
$$T_{(q, y)}(\{q\} \times \Bbb R) \subset T_{(q, y)}D \text{ for all } (q, y) \in Q \times \Bbb R.$$
To accomplish this, observe $\pi(\Lambda) \subset Q \times \Bbb R$ is nowhere vertical as it is a front projection of a Legendrian submanifold of $J^1 Q$. Let $D' \subset Q \times \Bbb R$ be the union of vertical segments of width $\epsilon > 0$ passing through every point of $D \cap \pi(\Lambda)$. As $D \cap \pi(\Lambda)$ is topologically an interval, $D'$ is a topological disk. We smoothen the corners to obtain a smooth disk. By choosing $\epsilon > 0$ sufficiently small, we can ensure $D' \subset Q \times \Bbb R$ is $C^\infty$--close to a tubular neighborhood $\mathrm{Op}_D(D \cap \pi(\Lambda))$ of $D \cap \pi(\Lambda)$ inside $D$. Therefore, $D'$ is isotopic to $\mathrm{Op}_D(D \cap \pi(\Lambda))$. By transversality of $D$ with $\pi(\Lambda)$, we obtain $D' \cap \pi(\Lambda) = D \cap \pi(\Lambda)$. Thus, $D'$ is a vertical disk with the required properties. We therefore assume, without loss of generality, that $D$ is already vertical.

Since $D$ is transverse to $\pi(\Lambda)$, we may choose a neighborhood of the disk $D \subset Q \times \Bbb R$ of the form $U = D^2 \times D^{n-1}(\varepsilon)$ for some $\varepsilon > 0$, where $D = D^2 \times \{0\}$, $D^{n-1}(\varepsilon) \subset Q$, and
$$U \cap \pi(\Lambda) = (D \cap \pi(\Lambda)) \times D^{n-1}(\varepsilon).$$
By hypothesis, there is a diffeomorphism of pairs $(D, D \cap \pi(\Lambda)) \cong (D, Z)$. We compose this with the diffeomorphism $(D, Z) \cong (D, \delta Z)$ scaling the disk by $\delta$. By the smooth isotopy extension theorem, we may apply an ambient isotopy of $Q \times \Bbb R$ modifying the front $\pi(\Lambda)$ such that,
$$U \cap \pi(\Lambda) = \delta Z \times D^{n-1}(\varepsilon/2).$$
We cut off the ambient isotopy in the radial direction of the second component inside the region $D^2 \times (D^{n-1}(\varepsilon) \setminus D^{n-1}(\varepsilon/2)) \subset U$, making it compactly supported on $U$. We also modify $\Lambda$ accordingly by lifting the isotopy of the front to a contact isotopy of $J^1 Q$.

Since $D \subset Q \times \Bbb R$ is vertical, there is a contactomorphism,
$$\varphi : \pi^{-1}(U) \stackrel{\cong}{\to}  (D^2 \times \Bbb R) \times T^*D^{n-1}(\varepsilon/2).$$ 
Here, the codomain of $\varphi$ is equipped with the contact form $\alpha_{\rm{std}} = \alpha_0 - \lambda$, with $\alpha_0 = dz - ydx$ and $\lambda = \sum p_i dq_i$, where $(x, z)$ denotes the coordinates on the first factor $D^2$, $y$ denotes the coordinates on the second factor $\Bbb R$ and $(q, p)$ denotes the generalized position-momentum coordinates on $T^*D^{n-1}(\varepsilon)$. Recall the notation from Definition \ref{def-loose}. Observe, 
$$\varphi(\Lambda \cap \pi^{-1}(U)) = \delta\mathcal{Z} \times \{0\} \times \{|q| < \varepsilon/2, p = 0\} \subset D^2 \times \Bbb R \times T^* D^{n-1}(\varepsilon/2),$$
where $\mathcal{Z}$ is the Legendrian zig-zag given by the lift of $Z \subset D^2 \subset \Bbb R^2_{xz}$ to $(\Bbb R^3, \xi_\mathrm{std})$. Suppose the action of $\mathcal{Z}$ is $a$. Fix $\rho > 0$, and consider the contactomorphism,
\begin{gather*}\psi : (D^2 \times \Bbb R) \times T^*D^{n-1} \to (D^2 \times \Bbb R) \times T^*D^{n-1}, \\
\psi(x, z, y, q, p) = \left (x, \frac{z}{\delta}, \frac{y}{\delta}, \frac{2\rho q}{\varepsilon}, \frac{\varepsilon p}{2\rho \delta} \right ).\end{gather*}
Notice $\psi$ takes $\delta \mathcal{Z} \times \{0\} \times \{|q| < \varepsilon/2, p = 0\}$ to $\mathcal{Z} \times \{0\} \times \{|q| \leq \rho, p = 0\}$. Let $C \subset D^2 \times \Bbb R$ be a cube containing $\mathcal{Z}$. Let $V := (\psi \circ \varphi)^{-1}(C \times B_\rho)$. Then, we have a contactomorphism,
$$\psi \circ \varphi: (V, \Lambda \cap V) \stackrel{\cong}{\to} (C \times B_\rho, \mathcal{Z}_a \times J_\rho).$$
If $\rho$ is chosen sufficiently large so that $\rho^2/a > 1/2$, this will be a loose chart for $\Lambda$.
\end{proof}

The following proposition records a crucial fact regarding loose charts which uses the quantitative estimate $\rho^2/a > 1/2$ on the size parameter in an essential way.

\begin{proposition}[One loose chart contains arbitrarily many]\label{prop-manyloose}For any $\sigma > 1/2$, any loose chart contains arbitrarily many disjoint, isocontactly embedded copies of loose charts of size parameter $\sigma$.\end{proposition}

\begin{proof}The essential idea is reminiscent of Example \ref{eg-zigspin}: Given a loose chart, we compress the disk of zig-zags on a smaller subdisk. The crucial issue now is that unlike the case of $\mathcal{Z}_a \times S^1$ in Example \ref{eg-zigspin} we do not have periodic boundary conditions on a given loose chart. Thus, the compression isotopy necessarily needs to be cut off in the radial coordinate of the disk. We produce such an isotopy by directly constructing a smooth isotopy on the front projection. The condition $\rho^2/a > 1/2$ then ensures that the cut-off function can be chosen to be relatively less steep, so as to control the sizes of the box $C \times B_\rho$ that the new loose chart will be contained in. We proceed to explicate these ideas in detail.

Let $(C \times B_\rho, \mathcal{Z}_a \times J_\rho)$ be a loose chart, with size parameter $\rho^2/a > 1/2$. We may apply a scaling as in Remark \ref{rmk-sizescale} to assume without loss of generality that $a = 2$. Thus, looseness implies $\rho > 1$. Let us denote $\Lambda_0 := \mathcal{Z}_2 \times J_\rho \subset C \times B_\rho \subset \Bbb R^{2n+1}$. For $\delta > 0$ small, we shall define a smooth function $m_{\delta} : \Bbb R \to (0, \infty)$ such that:

\begin{enumerate}[leftmargin=5cm]
\item $m_\delta(x) = x$ for $x \geq 2\delta$, 
\item $m_\delta(x) = \delta$ for $x \leq \delta/2$,
\item $\delta \leq m_\delta(x) \leq x$ for $\delta/2 \leq x \leq 2\delta$, and
\item $m_{\delta}'(x) \leq 1$ for all $x \in \Bbb R$.
\end{enumerate}

\noindent
Explicitly, we construct $m_{\delta}$ by interpolating linearly on the interval $[\delta/2, 2\delta]$ and slightly smoothing out at the endpoints. To verify that Condition $(4)$ is satisfied, observe that the slope of the linear interpolation is $(2\delta - \delta)/(2\delta - \delta/2) = 2/3 < 1$. Thus, we can produce a smoothing so that $m_\delta'(x) \leq 1$ continues to hold. Next, let 
$$h_\delta(q) := m_\delta(|q| + 1 - \rho).$$
For a fixed $q$, let $h_\delta(q) \mathcal{Z}_2 \subset \Bbb R^3$ be the Legendrian curve obtained from contact scaling (Remark \ref{rmk-sizescale}) the Legendrian zig-zag $\mathcal{Z}_2$ of action $a = 2$, by $h_\delta(q)$. We define,
$$\Omega := \{(q, x, z) \in \Bbb R^{n+1} : (x, z) \in \pi_{\rm{front}}(h_\delta(q)\mathcal{Z}_2)\} \cap \{|q| < \rho\}$$
Observe that $\Omega \subset \Bbb R^{n+1}$ is a codimension $1$ topological submanifold with boundary. We may consider $\Omega$ as a multi-graph $\{z = z(q, x)\}$ over $\Bbb R^n_{q, x}$, singular along the cuspidal edges. Defining $y = \partial z/\partial x$ and $p_i = \partial z/\partial q_i$ for $1 \leq i \leq n-1$ gives a genuinely Legendrian lift $\Lambda \subset \Bbb R^{2n+1}$ of $\Omega$, as all the pertaining singularities are cuspidal (see, Proposition \ref{prop-legfr} as well as the calculation in Example \ref{eg-wrinklelift} away from the swallowtail points).

Intuitively, the Legendrian submanifold $\Lambda \subset \Bbb R^{2n+1}$ is obtained from ``squeezing $\Lambda_0$ in the middle", see Figure \ref{fig-loosesq} for an illustration. We make the following observations:
\begin{enumerate}
\item For any $(x, y, z, p, q) \in \Lambda$, we have: 
\begin{enumerate}[label=(\roman*), leftmargin=4cm]
\item $(x, y, z) \in h_\delta(q) \mathcal{Z}_2 \subset C$,
\item $|q| < \rho$ by definition of $\Omega$, 
\item $|p| = |\partial z/\partial q| \leq (1 - \delta)/\rho < \rho$.
\end{enumerate}
Indeed, $\rm{(iii)}$ holds as the maximum $z$--height of $\mathcal{Z}_a$ from the $x$--axis is $a/2$, and the $z$--height of $\Omega$ increases from $\delta$ to $1$ from $|q| = 0$ to $|q| = \rho$. Hence, $\Lambda \subset C \times B_\rho$. 
\item $\Omega$ and $\pi_{\rm{front}}(\Lambda_0)$ are related by a compactly supported smooth isotopy in $\Bbb R^{n+1}$. Thus, $\Lambda$ and $\Lambda_0$ are related by a compactly supported contact isotopy in $\Bbb R^{2n+1}$, by lifting the aforementioned smooth isotopy. 
\end{enumerate}

\begin{figure}[h]
\centering
\includegraphics[scale=0.2]{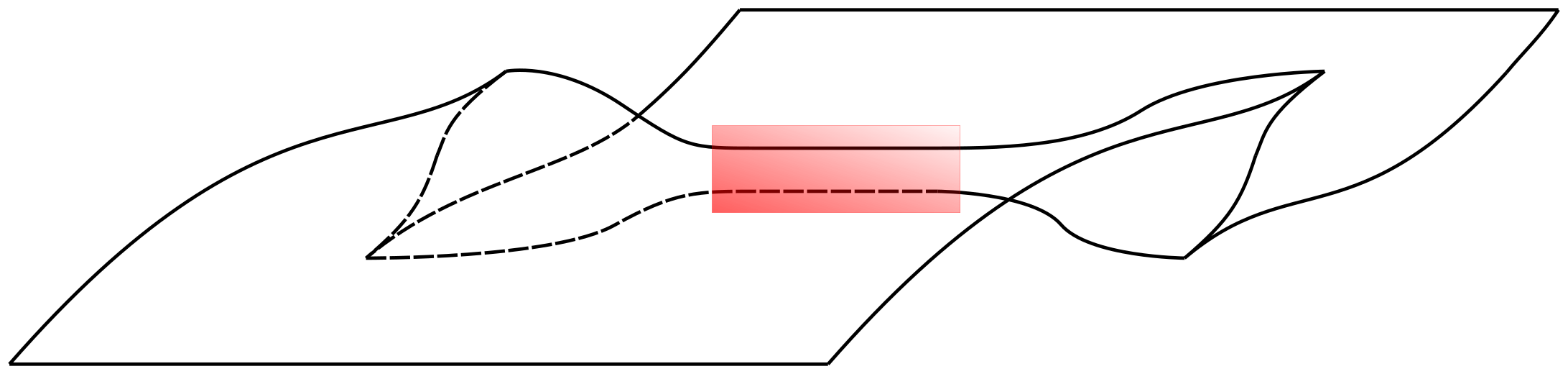}
\caption{Squeezing a loose chart in the middle to produce a new loose chart, depicted as the opaque red region.}
\label{fig-loosesq}
\end{figure}

Combining these observations, we obtain that if the cube $C \subset \Bbb R^3$ is sufficiently large, there is a compactly supported contactomorphism of $C \times B_\rho$ taking $\Lambda_0 = \mathcal{Z}_2 \times B_\rho$ to $\Lambda$. Let $\rho' := \rho - 1 + \delta/2$. Observe, 
$$(\delta C \times B_{\rho'}) \cap \Lambda = \delta \mathcal{Z}_2 \times J_{\rho'} = \mathcal{Z}_{2\delta} \times J_{\rho'}.$$
Thus, the chart $(\delta C \times B_{\rho'}, \mathcal{Z}_{2\delta} \times J_{\rho'})$ contactly embeds in $(C \times B_\rho, \Lambda)$. But note that $(\delta C \times B_{\rho'}, \mathcal{Z}_{2\delta} \times J_{\rho'})$ is a chart with size parameter $\rho'^2/(2\delta)$, which can be made arbitrarily large by choosing $\delta > 0$ sufficiently small. In particular, we can ensure $\rho'^2/\delta = \sigma > 1/2$.

Thus, $(\delta C \times B_{\rho'}, \mathcal{Z}_{2\delta} \times J_{\rho'})$ is a loose chart of size parameter $\sigma$ isocontactly embedded in the given loose chart $(C \times B_{\rho}, \mathcal{Z}_a \times J_\rho)$.  The same strategy can be employed to produce two, and hence arbitrarily many, disjoint loose charts isocontactly embedded in $(C \times B_{\rho}, \mathcal{Z}_a \times J_\rho)$ by ``squeezing down" at disjoint sub-domains in $\{|q| < \rho\}$, see Figure \ref{fig-manyloose}.
\end{proof}

\begin{figure}[h]
\centering
\includegraphics[scale=0.2]{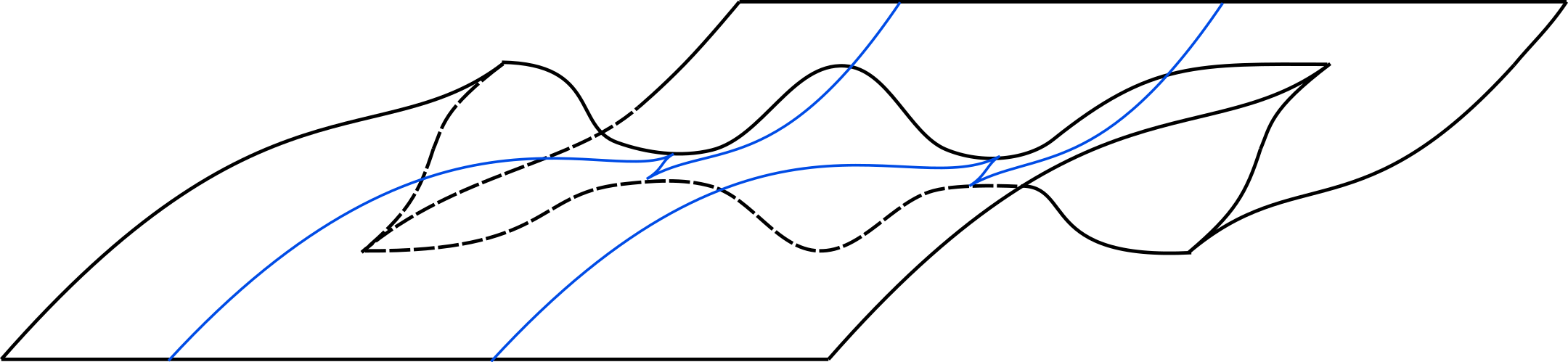}
\caption{Squeezing a loose chart to produce multiple new loose charts.}
\label{fig-manyloose}
\end{figure}

\subsection{Looseness and formal Legendrian embeddings} In this section we define what it means for a formal Legendrian embedding to have a loose chart, and set up a space of formal Legendrian embeddings with a fixed loose chart. This will be required for stating the $h$--principle for loose Legendrian embeddings. 

\begin{definition}[Space of formal Legendrian embeddings with a fixed loose chart]\label{def-fLegloose}Let $(Y^{2n+1}, \xi)$ be a contact manifold of dimension $2n+1 \geq 5$, and $\Lambda^n$ be a smooth $n$--dimensional manifold. Let $U \subset Y$ be a Darboux chart, equipped with a contactomorphism
$$\phi : (U, \xi|_U) \stackrel{\cong}{\to} (C \times B_{\rho}, \xi_{\rm{std}}).$$
Let $D^n \subset \Lambda$ be a chart, and $\varphi : D^n \to (U, \xi|_U)$ be a Legendrian embedding so that $\phi \circ \varphi$ is a parametrization of a loose chart $\mathcal{Z}_a \times J_\rho \subset C \times B_{\rho}$, with size parameter $\rho^2/a > 1/2$. 

We define $\mathrm{Emb}^f_{\rm{Leg}, \ell}(\Lambda, Y; \{U, \phi, D^n, \varphi\}) \subset \mathrm{Emb}^f_{\rm{Leg}}(\Lambda, Y)$ to be the subspace of formal Legendrian embeddings $(F_s, f) : (T\Lambda, \Lambda) \to (TY, Y)$ such that 
\begin{enumerate}
\item $f^{-1}(U) = D^n$,
\item $(F_s = df, f)$ is a holonomic Legendrian embedding on $D^n$, and
\item $f|D^n = \varphi$. 
\end{enumerate}
In other words, $\mathrm{Emb}^f_{\rm{Leg}, \ell}(\Lambda, Y; \{U, \phi, D^n, \varphi\})$ is the space of formal Legendrian embeddings with a fixed loose chart given by the data $\{U, \phi, D^n, \varphi\}$. We shall supress $\phi, D^n, \varphi$ from the notation, by denoting this space as $\mathrm{Emb}^f_{\rm{Leg}, \ell}(\Lambda, Y; U)$. We shall also denote,
$$\mathrm{Emb}_{\rm{Leg}, \ell}(\Lambda, Y; U) := \mathrm{Emb}^f_{\rm{Leg}, \ell}(\Lambda, Y; U) \cap \mathrm{Emb}_{\rm{Leg}}(\Lambda, Y)$$
to be the space of holonomic Legendrian embeddings with a fixed loose chart given by the data $\{U, \phi, D^n, \varphi\}$.
\end{definition}

\begin{remark}The subscript $\ell$ in the notation above stands for \emph{loose}.\end{remark}

We end the section with a description of a procedure called \emph{stabilization}, which produces a loose Legendrian submanifold from an arbitrary (possibly non-loose) Legendrian submanifold, by a homotopy of formal Legendrian embeddings. The procedure here differs only very slightly from the one described in \cite[Proposition 7.23]{cebook}, in that we begin with a high-dimensional version of the Legendrian Reidemeister move $\rm{I}$ (Example \ref{eg-reidmoves}) which is symmetric under the action of a torus on a local chart on the front.

\begin{definition}\label{def-reid1torus}Let $n \geq 2$, and $(\Bbb R^{2n+1}, \xi_{\rm{std}})$ be the standard contact structure on the Euclidean space, with coordinates $x = (x_1, \cdots, x_n), y = (y_1, \cdots, y_n)$ and $z$. Let $\Bbb R^n \subset \Bbb R^{2n+1}$ be the Legendrian subspace $\{y = z = 0\}$, with front projection $\Bbb R^n \subset \Bbb R^{n+1}_{xz}$ given by $\{z = 0\}$. The \emph{toroidal Legendrian Reidemeister move I} applied to the front produces an isotopic Legendrian submanifold $\Lambda \subset (\Bbb R^{2n+1}, \xi_{\rm{std}})$ with a $T^{n-1}$--symmetric front $\pi(\Lambda) \subset \Bbb R^{n+1}$ such that,
\begin{enumerate}
\item $\pi(\Lambda)$ agrees with $\{z = 0\}$ outside a compact ball around the origin,
\item $\pi(\Lambda)$ contains an embedded $K \times T^{n-1}$, where $K \subset \Bbb R^2$ is the top-left front diagram in Figure \ref{fig-reid}.
\end{enumerate}
\noindent
Explicitly, we may construct such a front as follows: let $K \times T^{n-1} \subset \Bbb R^{n+1}$ be embedded as an iterated hypersurface of revolution from $K \subset \Bbb R^2$. We translate it sufficiently in the $z$--direction to be disjoint from $\Bbb R^n = \{z = 0\} \subset \Bbb R^{n+1}$. Let $T^{n-1} \subset \Bbb R^n$ be a torus embedded in a compact ball around the origin. We glue $\Bbb R^n \setminus ([-\varepsilon, \varepsilon] \times T^{n-1})$ to $K \times T^{n-1}$ by an immersed tube $T^{n-1} \times I$ connecting $\{\pm \varepsilon\} \times T^{n-1}$ and $\partial K \times T^{n-1}$, in a way that does not produce any vertical tangencies. See Figure \ref{fig-r1surf}. 
\begin{figure}[h]
\centering
\includegraphics[scale=0.4]{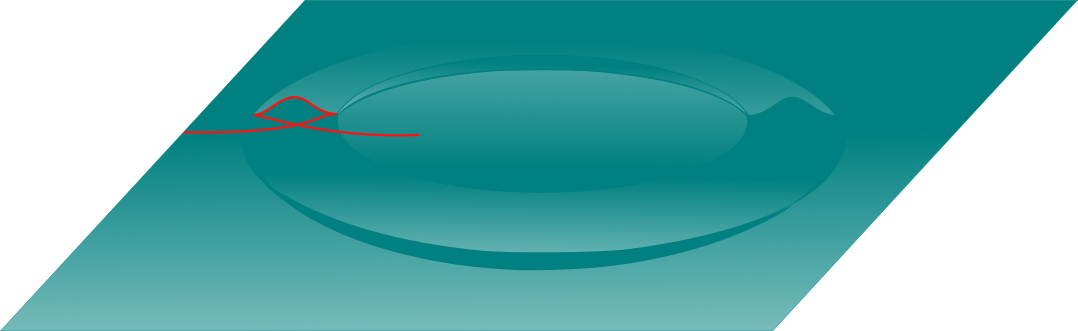}
\caption{$2$--dimensional toroidal Legendrian Reidemeister move $\rm{I}$.}
\label{fig-r1surf}
\end{figure}
\end{definition}

\begin{example}[Stabilization]\label{eg-stabloose}
Let $(Y^{2n+1}, \xi)$ be a contact manifold of dimension $2n+1 \geq 5$ and $\Lambda^n$ be a smooth $n$-dimensional manifold. Let $f : \Lambda \to (Y, \xi)$ be a Legendrian embedding. We identify $\Lambda$ with its image under $f$, and fix a point $p \in \Lambda$. Let us also fix a Darboux chart $U \subset Y$ around $p$. Thus, we have a contactomorphism of triples,
$$(U, \Lambda \cap U, p) \cong (\Bbb R^{2n+1}, \Bbb R^n, 0).$$
We apply the toroidal Legendrian Reidemeister move $\rm{I}$ around a ball containing $p$, with respect to Euclidean coordinates given by the contactomorphism above. Thus, we find an embedded $K \times T^{n-1}$ in the modified front around $p$. Consider the following sequence of moves applied to the front diagram $K$:
\begin{figure}[h]
\centering
\includegraphics[scale=0.4]{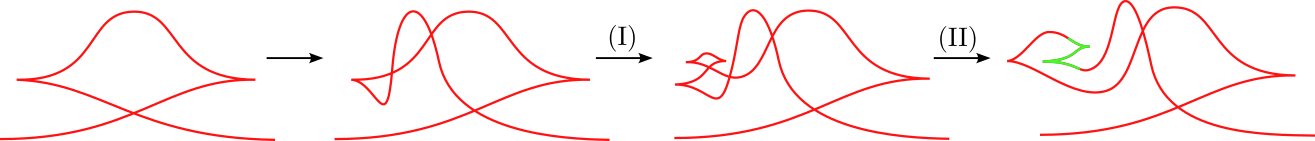}
\caption{The first arrow is a topological isotopy, and the final diagram contains a zig-zag (green)}
\label{fig-stabloose}
\end{figure}

\noindent
Here, the first move is a topological isotopy which is \emph{not} Legendrian. The second and third moves are the Legendrian Reidemeister moves $\rm{I}$ and $\rm{II}$ respectively. Let us call the final front diagram $K_1$. Then the isotopy above gives rise to a topological isotopy between the Legendrian fronts $K \times T^{n-1}$ and $K_1 \times T^{n-1}$ in $\Bbb R^{n+1}$. As the isotopy is cut-off near $\partial K$, this also gives rise to a topological isotopy of the Legendrian front $\pi(\Lambda \cap U) \subset \Bbb R^{n+1}$. Observe that $K_1$ contains an isolated zig-zag, depicted by green in Figure \ref{fig-stabloose}. Thus, $(U, \Lambda \cap U)$ contains a contactly embedded pair $(C \times T^* T^{n-1}, \mathcal{Z} \times T^{n-1})$. By Example \ref{eg-zigspin} and Remark \ref{rmk-zigspin}, this produces a loose Legendrian submanifold $\Lambda_1 \subset (Y, \xi)$ topologically isotopic to $\Lambda$.

It remains to see that this topological isotopy can be improved to a formal Legendrian isotopy. The only issue here is the first move in Figure \ref{fig-stabloose}, as the rest of the moves are Legendrian Reidemeister moves, hence in particular Legendrian isotopies. Firstly, observe that this move is already more than a topological isotopy: It lifts to an isotopy through Legendrian \emph{immersions}. Secondly, this move is entirely local near a neighborhood of the cuspidal locus: We only require the two branches of the front lying over each other, as we pass the bottom branch through the top. We state the existence of the formal Legendrian isotopy in this simple local model as Lemma \ref{lem-forlegev}.
\end{example}

\begin{lemma}\label{lem-forlegev}Let $(\Bbb R^{2n+1}, \xi_{\rm{std}})$ be the Euclidean space with the standard contact structure and coordinates $x = (x_1, \cdots, x_n)$, $y = (y_1, \cdots, y_n)$ and $z$. Define,
\begin{align*}\Sigma_- & := \{y = z = 0\},\\
\Sigma_+ & := \{y = 0, z = 1\},\\
\Sigma & := \Sigma_+ \cup \Sigma_-.\end{align*}
Let $N \subset \Sigma_-$ be a compact domain of zero Euler characteristic $\chi(N) = 0$, $U \subset \Sigma_-$ be an $\varepsilon$-neighborhood of $N$, and $\phi : \Sigma_- \to \Bbb R$ be a smooth function, with $\mathrm{supp}(\phi) \subset U$ and $\phi|_N > 2$. Define,
\begin{align*}
\Sigma_{-, \phi} & := \{(x, \nabla \phi(x), \phi(x)) \in \Bbb R^{2n+1} : x \in \Sigma_-\},\\
\Sigma_\phi & := \Sigma_+ \cup \Sigma_{-, \phi}. 
\end{align*}
The Legendrian submanifolds $\Sigma, \Sigma_\phi$ are formally Legendrian isotopic in $(\Bbb R^{2n+1}, \xi_{\rm{std}})$.
\end{lemma}

\begin{proof}As $\chi(N) = 0$, there exists a nowhere-vanishing vector field $X$ on $N$ which agrees with $\nabla \phi$ near $\partial N$. Consider the family of smooth maps $\{f_t : \Sigma \to \Bbb R^{2n+1} : 0 \leq t \leq 1\}$, defined by $f_t|_{\Sigma_{+}} = \mathrm{id}$, and for all $x \in \Sigma_{-}$, 
\[
f_t(x)=
\begin{cases} 
(x, 3t \cdot X(x), 0) & t \in [0, 1/3] \\
(x, X(x), 3t \cdot \phi(x)) & t \in [1/3, 2/3] \\
(x, (3t - 2) \cdot \nabla\phi(x) + (3 - 3t) \cdot X(x), \phi(x)) & t \in [2/3, 1] 
\end{cases}
\]
Note that as $X \neq 0$ on $N$, the family $\{f_t : \Sigma \to \Bbb R^{2n+1} : 0 \leq t \leq 1\}$ is a homotopy through smooth embeddings. Furthermore, consider the family of smooth maps $\{g_t : \Sigma \to \Bbb R^{2n+1} : 0 \leq t \leq 1\}$, defined by $g_t|_{\Sigma_{+}} = \mathrm{id}$, and for all $x \in \Sigma_{-}$, 
$$g_t(x) := (x, t \nabla \phi(x), t\phi(x)), \ t \in [0, 1].$$
Then $\{g_t : \Sigma \to \Bbb R^{2n+1} : 0 \leq t \leq 1\}$ is a homotopy through Legendrian immersions. Define,
\begin{gather*}F_{s, t} : T \Sigma \to T \Bbb R^{2n+1}, \\
F_{s, t} := s \cdot df_t + (1 - s) \cdot dg_t, \ s \in [0, 1]
\end{gather*}
Then, $\{(F_{s, t}, f_t) : t \in [0, 1]\}$ defines the required formal Legendrian isotopy, as we have $f_0(\Sigma) = \Sigma$, $f_1(\Sigma) = \Sigma_\phi$, and also $F_{s, t} = df_t$ for $t \in \{0, 1\}$ since $f_t = g_t$ for $t \in \{0, 1\}$.
\end{proof}

Applying Lemma \ref{lem-forlegev} to the situation in Example \ref{eg-stabloose} with $N = I \times T^{n-1}$, we obtain the required formal Legendrian isotopy. Notice that $n \geq 2$ is crucial, as otherwise $\chi(N)$ would not be zero.

\subsection{$h$--principle for loose Legendrians} We are now ready to state the $h$--principle for loose Legendrian embeddings, due to Murphy \cite[Theorem 1.3]{mur}. 

\begin{theorem}[$h$--principle for loose Legendrian embeddings]\label{thm-hloose}Let $(Y^{2n+1}, \xi)$ be a contact manifold of dimension $2n+1 \geq 5$, and $\Lambda^n$ be a smooth $n$--dimensional manifold. Let $\{U, \phi, D^n, \varphi\}$ be the data of a fixed Legendrian immersion of $D^n \subset \Lambda$ in $U \subset Y$ parametrizing a loose chart, as in Definition \ref{def-fLegloose}.

Let $\{(F_{s, t}, f_t) : t \in I^d\}$ be a $d$--parameter family of formal Legendrian embeddings in $\mathrm{Emb}^f_{\rm{Leg}, \ell}(\Lambda, Y; U)$. Suppose for all $t \in \partial I^d$, $(F_{s, t} = df_t, f_t)$ is a holonomic Legendrian embedding. Then the family $\{(F_{s, t}, f_t) : t \in I^d\}$, considered as a $d$--parameter family in $\mathrm{Emb}^f_{\rm{Leg}}(\Lambda, Y)$, is family--homotopic to a $d$--parameter family of holonomic Legendrian embeddings relative to $\partial I^d$.

In other words, the inclusion map 
$$(\mathrm{Emb}^f_{\rm{Leg}, \ell}(\Lambda, Y; U), \mathrm{Emb}_{\rm{Leg}, \ell}(\Lambda, Y; U)) \hookrightarrow (\mathrm{Emb}^f_{\rm{Leg}}(\Lambda, Y), \mathrm{Emb}_{\rm{Leg}}(\Lambda, Y))$$
induces the zero map on all relative homotopy groups.
\end{theorem}

The key strategy of the proof of Theorem \ref{thm-hloose} involves using the parametric and relative version of Theorem \ref{thm-hwrinkleg} to homotope the family $\{(F_{s, t}, f_t) : t \in I^d\}$ to a family of wrinkled Legendrian embeddings relative to $\partial I^d$, as well as relative to the fixed loose chart $f_t^{-1}(U) = D^n$ for all $t \in I^d$. Then the fixed loose chart that was left undisturbed in the aforementioned homotopy is used to ``cancel" the wrinkle singularities, resulting in a family of (nonsingular) Legendrian embeddings. We discuss the necessary prerequisites to carry out the cancellation procedure before discussing the proof of Theorem \ref{thm-hloose}.

\begin{definition}[Markings]\label{def-marking}Let $f : \Lambda^n \to (Y^{2n+1}, \xi)$ be a wrinkled Legendrian embedding. A \emph{marking} for $f$ is a compact submanifold with boundary $\Phi \subset \Lambda$ of codimension $1$, such that 
\begin{enumerate}
\item $\partial \Phi = \sqcup_j S_j^{n-2}$ is a union of some of the wrinkle loci $S_j^{n-2} \subset \Lambda$ of $f$,
\item For a Darboux chart $W_j \subset Y$ containing a wrinkle such that $f^{-1}(W_j) \cong \Bbb R^n$, recall from Definition \ref{def-wrinkleg} (cf. Definition \ref{def-wrinkledemb}) the coordinates $(x_1, \cdots, x_n, t)$ on $f^{-1}(W_j)$, as well as the front projection $f_j = \pi_{\rm{front}} \circ f|_{f^{-1}(W_j)} : \Bbb R^n \to \Bbb R^{n+1}$. Then $f_j$ is a wrinkled embedding, with swallowtail locus $S_j^{n-2} = \{t = 0, |x|^2 = 1\}$. We demand,
$$f^{-1}(W_j) \cap \Phi = \{t = 0, |x| \geq 1\}.$$
\item The interior of $\Phi$ is disjoint from the singular set of $f$.
\end{enumerate}
See Figure \ref{fig-marking} (left) for an illustration of a marking. Let $\{f_s : \Lambda \to (Y, \xi) : s \in I^d\}$ be a family of wrinkled Legendrian embeddings. Let us denote the collection of times where the family attains an embryo singularity as
$$\mathcal{E} = \{s \in I^d : f_s \text{ has an embryo singularity}\} \subset I^d.$$
Notice that $\mathcal{E}$ is a codimension $1$ submanifold of $I^d$. A \emph{family of markings} for $\{f_s\}$ is a family of compact submanifolds $\{\Phi_s \subset \Lambda\}$ of codimension $1$, such that 
\begin{enumerate}
\item For all $s \in I^d \setminus \mathcal{E}$, $\Phi_s$ is a marking of $f_s$,
\item $\{\Phi_s : s \in I^d \setminus \mathcal{E}\}$ is a smooth isotopy of embeddings, and
\item For any coordinate $\tau$ on $I^d$ transverse to $\mathcal{E}$ and coordinates $(x_1, \cdots, x_n, t)$ on the Darboux charts $W_j$ containing an embryo as in Definition \ref{def-wrinkleg} (cf. Definition \ref{def-wrinkledemb}), 
$$f^{-1}(W_j) \cap \Phi_\tau = \{t = 0, |x|^2 \geq \tau\}.$$
\end{enumerate}
\end{definition}

\begin{proposition}[Resolution of wrinkles along markings]\label{prop-resowrink}Let $\{f_s : \Lambda \to (Y, \xi) : s \in I^d\}$ be a family of wrinkled Legendrian embeddings, and $\{\Phi_s \subset \Lambda\}$ be a family of markings for $\{f_s : s \in I^d\}$. Then $\{f_s : s \in I^d\}$ is $C^0$--close to a family of wrinkled Legendrian embeddings $\{\widetilde{f}_s : \Lambda \to (Y, \xi) : s \in I^d\}$ such that $\widetilde{f}_s$ is smooth on $\mathrm{Op}(\Phi_s)$, and agrees with $f_s$ outside of $\mathrm{Op}(\Phi_t)$. See Figure \ref{fig-marking} (right) for an illustration of resolution of wrinkles along a marking.\end{proposition}

\begin{proof}Let $f : \Lambda \to Y$ be a single wrinkled Legendrian embedding, and let $\Phi \subset \Lambda$ be a marking for $f$. Suppose $\partial \Phi = \sqcup_j S_j^{n-2}$, where the image of each singular locus is contained in a Darboux chart $f(S_j^{n-2}) \subset W_j$ as in Definition \ref{def-wrinkleg}. Let $C \cong \partial \Phi \times [0, \varepsilon) \subset \Lambda$ be a small outward-pointing embedded collar of $\Phi \subset \Lambda$. We may write $C = \bigsqcup_j C_j$ where,
$$C_j := C \cap f^{-1}(W_j) = \{t = 0, 1 \leq |x| \leq 1 + \varepsilon\} \subset f^{-1}(W_j),$$
We shall denote $\Phi^{+} := \Phi \cup_{\partial} C$ and $\Phi^\circ := \Phi \setminus \partial \Phi$. Let us also choose an ambient Riemannian metric, and define $\rho : \Phi^+ \to \Bbb R$ to be a smoothing of $\mathrm{dist}^2(\cdot, \partial \Phi)$, so that $\rho$ is negative precisely on the outward collar $C$.

Let us choose a coordinate transverse to $\Phi^{+} \subset \Lambda$. Thus, we obtain a normal neighborhood $\Sigma := \Phi^{+} \times (-\varepsilon, \varepsilon) \subset \Lambda$. The restriction $f|\Phi^\circ \times (-\varepsilon, \varepsilon)$ is a smooth Legendrian embedding. Thus, by Weinstein's tubular neighborhood theorem (Theorem \ref{thm-weinstein}), there exists a tubular neighborhood of $f(\Phi^\circ \times (-\varepsilon, \varepsilon))$ contactomorphic to $J^1(\Phi^\circ \times (-\varepsilon, \varepsilon))$. We extend the Weinstein tubular neighborhood slightly into the collar to obtain a neighborhood $V \subset Y$ of $f(\Phi^{+} \times (-\varepsilon, \varepsilon)) = f(\Sigma)$ in $Y$ such that,
\begin{enumerate}
\item $V$ is contactomorphic to $J^1(\Sigma)$,
\item The front projection $\pi_V : V \cong J^1(\Sigma) \to \Sigma \times \Bbb R$ 
satisfies, for all $j$, 
$$\pi_V|f(C_j \times (-\varepsilon, \varepsilon)) = f_j.$$
\end{enumerate}
Thus, $f|_{\Sigma}$ is the Legendrian lift along $\pi_V$ of the graph of 
$$\phi : \Sigma = \Phi^+ \times (-\varepsilon, \varepsilon) \to \Bbb R, \;\; \phi(x, t) = \psi_{-\rho(x)}(t),$$
where $\psi_\delta(t)$ is the cut-off model zig-zag in Definition \ref{def-legzigzag}, appropriately scaled so that the cusps lie over $(-\varepsilon, \varepsilon)$ (cf. Figure \ref{fig-modelzig}). We modify $f : \Lambda \to Y$ to a Legendrian embedding $\widetilde{f} : \Lambda \to Y$ by setting $\widetilde{f} = f$ outside $\Sigma$, and defining $\widetilde{f}|_{\Sigma}$ to be the Legendrian lift along $\pi_V$ of the graph of,
$$\widetilde{\phi} : \Sigma = \Phi^+ \times (-\varepsilon, \varepsilon) \to \Bbb R, \;\; \widetilde{\phi}(x, t) = \psi_{m_\delta(-\rho(x))}(t),$$
where $m_\delta : \Bbb R \to (0, \infty)$ is the function defined in the proof of Proposition \ref{prop-manyloose}, and $\delta \ll \varepsilon$.
Note that the front of $\widetilde{f}|_{\Sigma}$, which is the graph of $\widetilde{\phi}$ in $\Sigma \times \Bbb R$, has only cuspidal singularities. Therefore, $\widetilde{f}$ is a smooth Legendrian embedding. Since a family of markings are consistently defined at embryo singularities, the modification above can be defined in families as well, proving the proposition.
\end{proof}

\begin{figure}[h]
\centering
\includegraphics[scale=0.19]{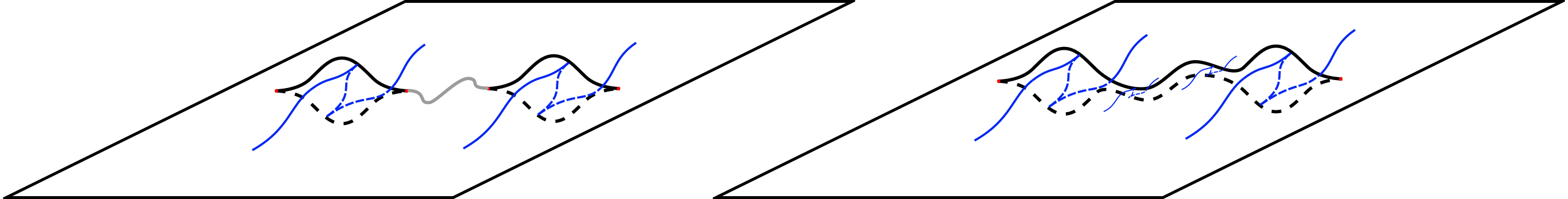}
\caption{A pair of $2$-dimensional wrinkles with a $1$-dimensional marking in grey (left) and its resolution (right).}
\label{fig-marking}
\end{figure}

The proof of Proposition \ref{prop-resowrink} may be better understood in the following model case.

\begin{example}[Resolving an inside-out wrinkle to a loose chart]\label{eg-iowresol}
Recall the cut-off model zig-zag $\psi_\delta$ in Definition \ref{def-legzigzag} (cf. Figure \ref{fig-modelzig}). We define an \emph{inside-out wrinkle} to be the map
$$w : \Bbb R^n \to \Bbb R^{n+1},\;\; w(x, t) = (x, \psi_{|x|^2 - 1}(t)).$$
The wrinkle locus of $w$ is the hyperboloid $\{|x|^2 - t^2 = 1\} \subset \Bbb R^n$ and swallowtail locus is the sphere $\{t = 0, |x| = 1\}$. Let $\widetilde{w} : \Bbb R^n \to \Bbb R^{2n+1}$ be the wrinkled Legendrian embedding with front $w$. A convenient marking (Definition \ref{def-marking}) for $\widetilde{w}$ is given by the \emph{disk marking},
$$\Phi = \{t = 0, |x|^2 \leq 1\} \subset \Bbb R^n.$$
Resolving $\widetilde{w}$ along $\Phi$ as in the proof of Proposition \ref{prop-resowrink} gives a Legendrian with front:
$$w_{\rm{res}} : \Bbb R^n \to \Bbb R^{n+1}, \;\; w_{\rm{res}}(x, t) = (x, \psi_{m_\delta(|x|^2 - 1)}(t)).$$
For $\delta > 0$ sufficiently small, $w_{\rm{res}}$ contains a loose chart given by $(\mathrm{Op}(w_{\rm{res}}(\Phi)), w_{\rm{res}}(\mathrm{Op}(\Phi)))$. 
\end{example}

\begin{proof}[Proof of Theorem \ref{thm-hloose}]
By the parametric and relative version of Theorem \ref{thm-hwrinkleg}, we may homotope the given family of formal Legendrian embeddings in $\mathrm{Emb}_{\rm{Leg},\ell}^f(\Lambda, Y; U)$ to a family 
$$\{f_t : \Lambda \to Y : t \in I^d\},$$
of wrinkled Legendrian embeddings, keeping the family-homotopy fixed on the subset $(\Lambda \times \partial I^d) \cup (U \times I^d) \subset \Lambda \times I^d$. Let us denote $\mathcal{E} = \{t \in I^d : f_t \text{ has an embryo singularity}\}$.
Then $\mathcal{E} \subset \mathrm{int}(I^d)$ is a codimension $1$ smooth submanifold. Let $k$ be the number of connected components of $\mathcal{E}$. If $k = 0$, $\{f_t : t \in I^d\}$ is already a family of smooth Legendrian embeddings and we have nothing to prove. Thus, assume $k \geq 1$ and let $\mathcal{E}_1, \cdots, \mathcal{E}_k$ be the connected components of $\mathcal{E}$. 

By hypothesis, $f_t^{-1}(U) = D^n \subset \Lambda$ is a fixed chart, $f_t = f_0$ on $D^n$ for all $t \in I^d$, and $(U, f_t(D^n))$ is a fixed loose chart in $f_t(\Lambda)$ for all $t \in I^d$. By Proposition \ref{prop-manyloose}, there exists disjoint open subsets $U_i \subseteq U$, $1 \leq i \leq k$ such that $(U_i, f_t(\Lambda) \cap U_i)$ is a fixed loose chart in $f_t(\Lambda)$ for all $t \in I^d$.

Consider a new family $\{g_t : \Lambda \to Y : t \in I^d\}$ of wrinkled Legendrian embeddings, given by setting $g_t = f_t$ on $\Lambda \setminus \bigsqcup_{i=1}^k U_i$ and defining $g_t|f_t^{-1}(U_i)$ to be a fixed inside-out wrinkle (Example \ref{eg-iowresol}), with the same boundary as the loose chart $f_t(\Lambda) \cap U_i \subset U_i$. We choose a family of markings $\{\Phi^i_t\}$ for $\{g_t : t \in I^d\}$ such that:
\begin{enumerate}
\item For all $t \in \partial I^d$, $\Phi^i_t \subset g_t^{-1}(U_i)$ is the disk marking (see, Example \ref{eg-iowresol}) of $g_t|U_i$.
\item For all $t \in I^d$, $\Phi^i_t$ is diffeomorphic to either $D^{n-1}$ or $S^{n-2} \times [0, 1]$. 
\item The embryos of $\mathcal{E}_i \subset I^d$ are all contained in $\Phi^i_t$, for each $1 \leq i \leq k$.
\item For all $t \in I^d$, $\partial \Phi^i_t$ consists only of the Legendrian wrinkles created by the embryos of $\mathcal{E}_i$, other than the inside-out wrinkle contained in $g_t^{-1}(U_i)$.
\end{enumerate}
We claim that such a choice is possible. Indeed,  the singular sets of the family $\{g_t : t \in I^d\}$ is a $d$--parameter family of submanifolds of $\Lambda$ given for $t \in I^d \setminus \mathcal{E}$ by codimension $2$ embedded contractible spheres, which are allowed to shrink to a point for $t \in \mathcal{E}$. Thus, any such family can be realized as boundary of embedded copies of $D^{n-1}$ or $S^{n-2} \times [0, 1]$.

We apply Proposition \ref{prop-resowrink} to resolve $\{g_t : t \in I^d\}$ along the markings $\{\Phi^i_t : t \in I^d\}$, one at a time for $i = 1, \cdots, k$. This leaves us with a family of smooth Legendrian embeddings,
$$\{\widetilde{g}_t : \Lambda \to Y : t \in I^d\}.$$
There is evidently a family-homotopy between $\{f_t : t \in \partial I^d\}$ and $\{\widetilde{g}_t : t \in \partial I^d\}$, supported on $\bigsqcup_{i = 1}^k U_i$. We add an annular collar $\partial I^d \times I$ carrying this homotopy to the domain $I^d$ of $\{\widetilde{g}_t : t \in I^d\}$. This gives a family of smooth Legendrian embeddings extending $\{f_t : t \in \partial I^d\}$, which we continue to denote as $\{\widetilde{g}_t : \Lambda \to Y : t \in I^d\}$ by a slight abuse of notation.

As a final step, we need only check that $\{\widetilde{g}_t : t \in I^d\}$ is homotopic to $\{f_t : t \in I^d\}$ as a family of \emph{formal Legendrian} embeddings, relative to $\partial I^d$. Let us fix a time $t \in I^d$. Let $U_i \subset Y$ be such that $(U_i, f_t(\Lambda) \cap U_i)$ is a loose chart for $f_t(\Lambda)$, and $W_j \subset \Lambda$ be a Darboux chart containing a wrinkle, generated from the embryo locus $\mathcal{E}_i$. Suppose $\Phi^i_t \cong S^{n-2} \times [0, 1]$ is a marking for $g_t$ such that 
\begin{enumerate}
\item $S^{n-2} \times \{0\} \subset f_t^{-1}(U_i)$ is the singular locus of the inside-out wrinkle $g_t|f^{-1}(U_i)$, 
\item $S^{n-2} \times \{1\} \subset W_j$ is the singular locus of the wrinkle $g_t|W_j = f_t|W_j$.
\end{enumerate}
Let $V \subset Y$ be the open neighborhood of $\Phi$ contactomorphic to $J^1(\Phi^+ \times (-\varepsilon, \varepsilon))$, considered in the proof of Proposition \ref{prop-resowrink}. Consider the open set $O = U_i \cup V \cup W_j \subset Y$. Shrinking $O$ slightly if necessary, we may assume $f_t^{-1}(O) = g_t^{-1}(O) = \widetilde{g}_t^{-1}(O)$. We shall demonstrate that $f_t$ and $\widetilde{g}_t$ are formally Legendrian isotopic in $O$. See Figure \ref{fig-loosewrinkle} for an illustration of this formal isotopy.

Indeed, note that resolving $g_t$ along the disk marking of the inside-out wrinkle in $U_i$ (Example \ref{eg-iowresol}) returns $f_t$, consisting only of a loose chart in $U_i$ and a disjoint isolated wrinkle in $W_j$. As a model wrinkle $\Bbb R^n \to \Bbb R^{2n+1} = J^1 \Bbb R^n$ is formally Legendrian isotopic to the zero section, $f_t$ is formally Legendrian isotopic to the loose chart $U_i$ within $O$. On the other hand, resolving $g_t$ along $\Phi_t^i$ returns $\widetilde{g}_t$. Once again, this consists of a single loose chart in $O$, with the same boundary conditions as before. Thus, $f_t$ and $\widetilde{g}_t$ are both formally Legendrian isotopic in $O$ to a fixed loose chart. 

This proves that $f_t$ and $\widetilde{g}_t$ are formally Legendrian isotopic. If $t \in I^d$ is a time for which $\Phi^i_t \cong D^{n-1}$ or $\Phi^i_t$ contains an embryo singularity, a completely analogous argument shows $f_t$ and $\widetilde{g}_t$ are still formally Legendrian isotopic. As this formal Legendrian isotopy is well-defined upto a contractible choice, we may use it to define a formal Legendrian isotopy between the families $\{f_t : t \in I^d\}$ and $\{\widetilde{g}_t : t \in I^d\}$, as desired. \end{proof}

\begin{figure}
\centering
\includegraphics[scale=0.18]{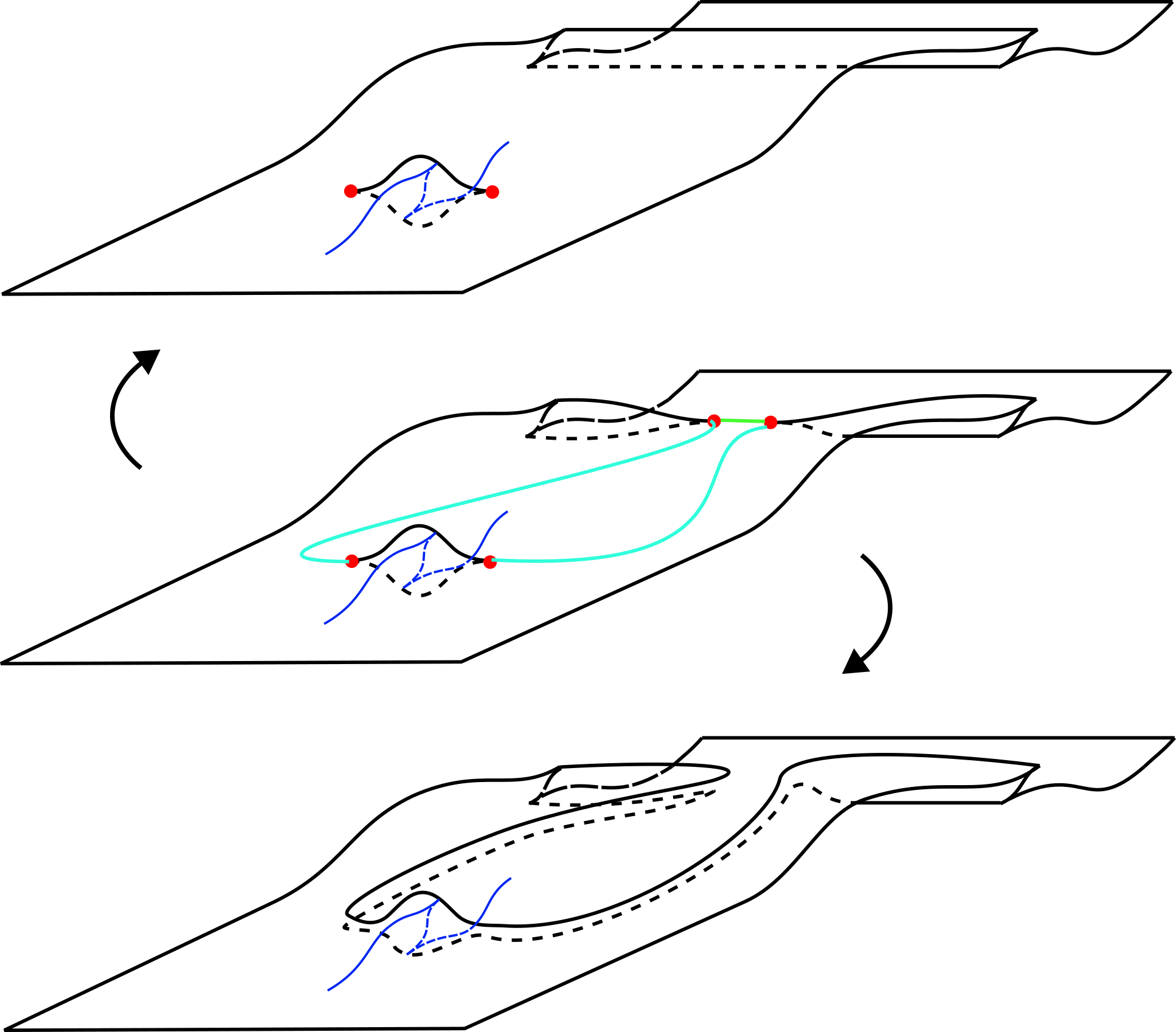}
\caption{A loose chart and a disjoint wrinkle (top), an inside-out wrinkle and a disjoint wrinkle (middle), a loose chart and no wrinkles (bottom). The middle front has a marking joining the inside-out wrinkle with the disjoint wrinkle (cyan) and a disk marking for the inside-out wrinkle (green). Resolving the middle front along the green marking produces the top front, whereas resolving along the cyan marking produces the bottom front.}
\label{fig-loosewrinkle}
\end{figure}

The following is an important consequence of Theorem \ref{thm-hloose} (see, \cite[Theorem 1.2]{mur}).

\begin{corollary}\label{cor-hloosepi0}Let $(Y, \xi)$ be a contact manifold with $\dim Y \geq 5$ and $f_0, f_1 : \Lambda \to Y$ be a pair of loose Legendrian embeddings. If $f_0, f_1$ are isotopic through formal Legendrian embeddings, then $f_0, f_1$ are Legendrian isotopic.\end{corollary}

\begin{proof}Let us suppose that the formal Legendrian isotopy between $f_0$ and $f_1$ is,
$$\{(F_{t, s}, f_t) : (T\Lambda, \Lambda) \to (TY, Y) : 0 \leq t \leq 1, 0 \leq s \leq 1\},$$
where $F_{0, s} = df_0$ and $F_{1, s} = df_1$ for all $0 \leq s \leq 1$. Let us choose loose charts $(U_0, U_0 \cap f_0(\Lambda))$ and $(U_1, U_1 \cap f_1(\Lambda))$ for $f_0$ and $f_1$, respectively. We may assume without loss of generality that $f_0^{-1}(U_0) = f_1^{-1}(U_1) = D^n \subset \Lambda$ is a fixed open disk, by reparametrizing the domain of $f_1$ by a diffeotopy if necessary. As $U_0, U_1 \subset Y$ are Darboux charts, by the contact isotopy extension theorem, there exists an ambient isotopy of contactomorphisms $\{\varphi_t : Y \to Y : 0 \leq t \leq 1\}$ such that $\varphi_0 = \mathrm{id}$, and $\varphi_1(U_1) = U_0$. Let us denote $g_t :=\varphi_t \circ f_t : \Lambda \to Y$. Then, $g_0^{-1}(U_0) = g_1^{-1}(U_1) = D^n$. Note that $g_0 = f_0$ and $g_1 = \varphi_1 \circ g_0$ is Legendrian isotopic to $f_1$. 

In other words, $g_0, g_1 \in \mathrm{Imm}_{\rm{Leg},\ell}(\Lambda, Y; U_0)$ are Legendrian embeddings with a fixed loose chart in $U_0$. We may \emph{smoothly} isotope the family $\{g_t : \Lambda \to Y : 0 \leq t \leq 1\}$ to a family of immersions $\{h_t : \Lambda \to Y : 0 \leq t \leq 1\}$ such that $h_0 = g_0, h_1 = g_1$ and $h_t|D^n = g_0|D^n$ for all $0 \leq t \leq 1$. Consider the projection map,
$$\mathrm{Imm}_{\rm{Leg}, \ell}^f(\Lambda, Y; U_0) \to \mathrm{Imm}(\Lambda, Y; U_0), \;\; (F_s, f) \mapsto f,$$
where $\mathrm{Imm}(\Lambda, Y; U_0)$ denotes the space of immersions of $\Lambda$ in $Y$ which restrict to $g_0|D^n$ on $D^n \subset \Lambda$. This is a Serre fibration, hence $\{h_t : 0 \leq t \leq 1\}$ lifts to a formal Legendrian isotopy with fixed loose chart $U_0$ between $g_0$ and $g_1$. By Theorem \ref{thm-hloose}, $g_0$ and $g_1$ are Legendrian isotopic. Consequently, $f_0$ and $f_1$ are also Legendrian isotopic, as required.
\end{proof} 

\section{Existence and detection of non-loose Legendrians}\label{sec-microsheaf}

The purpose of this section is to give examples of non-loose Legendrian submanifolds. A number of Legendrian isotopy invariants in high-dimensional contact manifolds exist in literature which are able to detect non-isotopic Legendrian embeddings in the same formal Legendrian isotopy class. Some instances of such invariants are existence of a generating function \cite{egbook}, Legendrian contact homology \cite{ees} and microlocal sheaf theory \cite{stz}. All of these invariants are known to vanish for loose Legendrian submanifolds. 

In this section, however, we focus on the invariants coming from microlocal sheaf theory. The observation that a loose Legendrian submanifold in a cosphere bundle does not occur as microsupport of a constructible sheaf is made by Murphy \cite[p.~17]{mur}: ``It is also easy to see that a loose Legendrian in any co-sphere bundle is never the singular support of a constructable sheaf". We shall explicate this comment in detail. Our proof will use some of the foundational results in this field developed by Kashiwara-Schapira \cite{ksbook} and Guillermou \cite{guipaper}, as well as a key result of Shende-Treumann-Zaslow \cite{stz} proving that the Legendrian zig-zag does not occur as microsupport of a constructible sheaf.

\subsection{Derived category and hypercohomology}

We recall some basic notions from sheaf theory and homological algebra. Let us fix a smooth manifold $M$ and a base field $\mathbf{k}$. Let $Sh(M)$ denote the category of sheaves of $\mathbf{k}$-vector spaces over $M$. 

\begin{definition}[Derived category of sheaves]\label{def-derivedcat}The \emph{bounded derived category} of sheaves of $\mathbf{k}$-vector spaces over $M$, denoted as $D^b(M)$, is a category defined as follows:
\begin{enumerate}[leftmargin=*]
\item An object of $D^b(M)$ is a cochain complex $\mathcal{F}^\bullet$ of sheaves over $M$ valued in $\mathbf{k}$-vector spaces, such that $H^n(\mathcal{F}^\bullet) = 0$ for all but finitely many $n \in \Bbb Z$, 
\item Let $\mathcal{F}^\bullet_i \in \mathrm{Ob}(D^b(M))$, $1 \leq i \leq 4$. A \emph{roof} is a diagram of the form:
\[
\begin{tikzcd}[row sep=large]
	& {\mathcal{F}^\bullet_3} \\
	{\mathcal{F}^\bullet_1} && {\mathcal{F}^\bullet_2}
	\arrow["s" ', from=1-2, to=2-1]
	\arrow["p", from=1-2, to=2-3]
\end{tikzcd}
\]
where $s$ and $p$ are chain homotopy-classes of chain maps, and $s$ is a \emph{quasi-isomorphism}, i.e. an isomorphism on homology in all degrees. Two roofs $\mathcal{F}_1^\bullet {\leftarrow} \mathcal{F}_i^\bullet {\rightarrow} \mathcal{F}_2^\bullet$, $ i = 3, 4,$ are called \emph{equivalent} if there exists $\mathcal{F}^\bullet_5 \in \mathrm{Ob}(D^b(M))$ and a third roof $\mathcal{F}_1^\bullet {\leftarrow} \mathcal{F}_5^\bullet {\rightarrow} \mathcal{F}_2^\bullet$ such that the following diagram commutes:
\[
\begin{tikzcd}[row sep = large]
	& {\mathcal{F}^\bullet_2} \\
	{\mathcal{F}^\bullet_1} & {\mathcal{F}^\bullet_5} & {\mathcal{F}^\bullet_2} \\
	& {\mathcal{F}^\bullet_4}
	\arrow[from=2-2, to=2-1]
	\arrow[from=2-2, to=2-3]
	\arrow[from=2-2, to=1-2]
	\arrow[from=2-2, to=3-2]
	\arrow[from=1-2, to=2-1]
	\arrow[from=1-2, to=2-3]
	\arrow[from=3-2, to=2-1]
	\arrow[from=3-2, to=2-3]
\end{tikzcd}
\]
For a pair of objects $\mathcal{F}^\bullet, \mathcal{G}^\bullet \in \mathrm{Ob}(D^b(M))$, we define morphisms in $\mathrm{Hom}(\mathcal{F}^\bullet, \mathcal{G}^\bullet)$ to be equivalence classes of roofs $\mathcal{F}^\bullet \leftarrow \mathcal{F}'^\bullet \to \mathcal{G}^\bullet$.
\end{enumerate}
\end{definition}

It is not immediately clear how to compose two morphisms in $D^b(M)$ as, naively speaking, composition of two roofs should give rise to a zig-zag of roofs. However, such a diagram can always be completed to a roof; we refer the reader to \cite[Chapter 1.6]{ksbook} for details. In categorical language, $D^b(M)$ is obtained from the naive homotopy category of bounded complexes of sheaves of $\mathbf{k}$-vector spaces over $M$, by localizing at (or formally inverting) the quasi-isomorphisms. In particular, quasi-isomorphisms are genuine isomorphisms in $D^b(M)$. 

\begin{example}\label{eg-degzerosheaf}A single sheaf $\mathcal{F} \in \mathrm{Ob}(Sh(M))$ can be considered as an element of $D^b(M)$ by treating it as a cochain complex concentrated in degree $0$:
\begin{equation}\label{eq-degzerosheaf}\cdots \to 0[-1] \to \mathcal{F}[0] \to 0[1] \to 0[2] \to \cdots\end{equation}
Suppose $\mathcal{I}^\bullet$ is an injective resolution of $\mathcal{F}$. That is, $\mathcal{I}^\bullet$ is a cochain complex of injective sheaves, $\mathcal{I}^n = 0$ for $n < 0$, and there is a morphism $\mathcal{F} \to \mathcal{I}^0$, such that
\begin{equation}\label{eq-injressheaf}0 \to \mathcal{F} \to \mathcal{I}^0 \to \mathcal{I}^1 \to \mathcal{I}^2 \to \cdots\end{equation}
is an exact sequence. Then, the morphism $\mathcal{F} \to \mathcal{I}^0$ gives rise to a chain map from the cochain complex in Equation \ref{eq-degzerosheaf} to the cochain complex $\mathcal{I}^\bullet$. Since Equation \ref{eq-injressheaf} is an exact sequence and hence acyclic, this chain map must be a quasi-isomorphism. Therefore, as objects of the derived category $D^b(M)$, the sheaf $\mathcal{F}$ is isomorphic to any injective resolution $\mathcal{I}^\bullet$ of it.
\end{example}

For any open set $U \subset M$ and a sheaf $\mathcal{F} \in \mathrm{Ob}(Sh(M))$, let us denote $\Gamma(U; \mathcal{F})$ to be the $\mathbf{k}$-vector space of sections of $\mathcal{F}$ over $U$. Then, $\Gamma(U; -) : Sh(M) \to \mathrm{Vect}_{\mathbf{k}}$ is a left-exact functor, whose derived functors are the \emph{sheaf cohomology} vector spaces $H^*(U; \mathcal{F})$. We extend the notion of sheaf cohomology from a single sheaf to a complex of sheaves in the following definition.

\begin{definition}[Hypercohomology]\label{def-hypercoh}Let $\mathcal{F}^\bullet \in \mathrm{Ob}(D^b(M))$ be a complex of sheaves. Let $\mathcal{I}^\bullet$ be an injective resolution of $\mathcal{F}^\bullet$, i.e., a complex of injective sheaves along with a quasi-isomorphism $\mathcal{F}^\bullet \to \mathcal{I}^\bullet$. Then, we define the \emph{sheaf hypercohomology} of $\mathcal{F}^\bullet$ over $U$ by,
$$H^k(U; \mathcal{F}^\bullet) := H^k(\Gamma(U; \mathcal{I}^\bullet)).$$ \end{definition}

\begin{remark}Even though Definition \ref{def-hypercoh} is completely identical to the definition of sheaf cohomology of a single sheaf, one confusion that might arise in practice is as follows: given a complex of sheaves $\mathcal{F}^\bullet$, every degree consists of a single sheaf $\mathcal{F}^i$, which has its own sheaf cohomology vector spaces $H^j(U; \mathcal{F}^i)$ over $U$. Hence, it might seem that the ``sheaf cohomology of $\mathcal{F}^\bullet$" ought to be a doubly-graded object indexed by $(i, j)$, whilst the hypercohomology of $\mathcal{F}^\bullet$ has a single grade. We pause here to address this potential confusion.

We shall construct an injective resolution for the complex of sheaves $\mathcal{F}^\bullet$ by constructing compatible injective resolutions $\mathcal{F}^i \to \mathcal{J}^{i, \bullet}$. Thus, one obtains a double complex $\{\mathcal{J}^{i, j}\}$ of injective sheaves. One then obtains a genuine cochain complex $\mathcal{I}^\bullet$ of injective sheaves by the \emph{totalization construction}, setting:
$$\mathcal{I}^n := \bigoplus_{i + j = n} \mathcal{J}^{i, j},$$
where the differentials are a (signed) sum of the horizontal and vertical differentials of $\mathcal{J}^{i, j}$. The chain maps $\mathcal{F}^i \to \mathcal{J}^{i, j}$ give rise to a chain map $\mathcal{F}^\bullet \to \mathcal{I}^\bullet$, which is the desired injective resolution. 

Thus, the hypercohomology $H^*(U; \mathcal{F}^\bullet)$ is a more refined object than the doubly graded sheaf cohomology $H^j(U; \mathcal{F}^i)$. Indeed, by the spectral sequence for cohomology of a double complex, we obtain from above that there is a spectral sequence with $E^2$ page consisting of $H^j(U; \mathcal{F}^i)$, converging to $H^{i+j}(U; \mathcal{F}^\bullet)$ in the $E^\infty$ page.\end{remark}

\subsection{Microsupport and its properties} In this section we introduce the key notion of microsupport of a complex of sheaves over a smooth manifold. Let us begin with the following definition, introducing the coarser notion of support:

\begin{definition}[Support]Let $M$ be a smooth manifold, and $\mathcal{F}^\bullet \in \mathrm{Ob}(D^b(M))$ be a complex of sheaves. The \emph{support} of $\mathcal{F}^\bullet$ is a subset $\mathrm{supp}(\mathcal{F}^\bullet) \subset M$ defined to be the closure of the set of points $x \in M$ such that the stalk $\mathcal{F}^\bullet_x$ (which is a complex of $\mathbf{k}$-vector spaces) is not quasi-isomorphic to the zero complex.
\end{definition}

Next, we introduce the microsupport of a complex of sheaves. It is a subset of the cotangent bundle of the underlying manifold, consisting of the codirections along which the hypercohomology of the complex of sheaves does not ``propagate" or ``parallel transport". This notion was first introduced in the setup of sheaves over manifolds by Kashiwara and Schapira \cite{ksbook}. We follow the exposition in Guillermou \cite{guipaper}.

\begin{definition}[Microsupport]\label{def-microsupp}Let $M$ be a smooth manifold, and $\mathcal{F}^\bullet \in \mathrm{Ob}(D^b(M))$ be a complex of sheaves. The \emph{microsupport} or \emph{singular support} of $\mathcal{F}^\bullet$ is a subset $\mathrm{SS}(\mathcal{F}^\bullet) \subset T^*M$ of the cotangent bundle defined to be the closure of the set of points $(x_0, \xi_0) \in T^*M$ such that there exists a smooth function $\varphi : M \to \Bbb R$ with $\varphi(x_0) = 0$ and $d\varphi(x_0) = \xi_0$ so that the restriction map,
\begin{equation}\label{eq-microstalk}\varinjlim_{U \ni x_0} H^k(U; \mathcal{F}^\bullet) \to \varinjlim_{U \ni x_0} H^k(U \cap \{\varphi < 0\}; \mathcal{F}^\bullet),\end{equation}
is \emph{not} an isomorphism for \emph{some} $k \in \Bbb Z$.
\end{definition}

\begin{example}Let $\mathcal{F} \in Sh(M)$ be a single sheaf, treated as a complex of sheaves concentrated in degree $0$ as in Example \ref{eg-degzerosheaf}. Suppose $(x_0, \xi_0) \notin \mathrm{SS}(\mathcal{F})$ is a covector not contained in the microsupport. Then, for every $(x, \xi) \in T^*M$ contained in a neighborhood around $(x_0, \xi_0)$, for any smooth function $\varphi : M \to \Bbb R$ with $\varphi(x) = 0$ and $d\varphi(x) = \xi$, the restriction map in sheaf cohomology,
$$\varinjlim_{U \ni x} H^k(U; \mathcal{F}) \to \varinjlim_{U \ni x} H^k(U \cap \{\varphi < 0\}; \mathcal{F}),$$
is an isomorphism for all $k$. Taking $k = 0$, we obtain an isomorphism
$$\varinjlim_{U \ni x} \Gamma(U; \mathcal{F}) \stackrel{\cong}{\to} \varinjlim_{U \ni x} \Gamma(U \cap \{\varphi < 0\}; \mathcal{F}).$$ 
Let us choose a ball $B$ in a chart around $x \in M$. The hypersurface $\{\varphi = 0\}$ cuts $B$ into two regions, $B_{\pm} = \{\pm \varphi > 0\} \cap B$. The isomorphism above implies that every section of $\mathcal{F}^\bullet$ defined on $B_{-}$ uniquely ``propagates" to the other side $B_+$ of the hypersurface, after possibly shrinking the ball $B$. For $k > 0$, one can offer similar interpretations in terms of ``propagation" of higher Cech cocycles.\end{example}

\begin{definition}\label{def-conic}There is a natural $\Bbb R_{>0}$-action on $T^*M$ given by fiberwise scaling $c \cdot (x, \xi) = (x, c \xi)$. We shall call a subset $Z \subset T^*M$ \emph{conic} if it is invariant under this action.\end{definition}

\begin{proposition}For a complex of sheaves $\mathcal{F}^\bullet \in \mathrm{Ob}(D^b(M))$, $\mathrm{SS}(\mathcal{F}^\bullet) \subset T^*M$ is a closed conic subset. Moreover, $\mathrm{SS}(\mathcal{F}^\bullet) \cap 0_M = \mathrm{supp}(\mathcal{F}^\bullet)$.
\end{proposition}

\begin{proof}
It is straightforward from the definition that $\mathrm{SS}(\mathcal{F}^\bullet) \subset T^*M$ is a closed conic subset. We begin to prove the second half of the proposition by writing a simplified expression for the domain of the restriction map in Equation \ref{eq-microstalk}:
$$\varinjlim_{U \ni x_0} H^k(U; \mathcal{F}^\bullet) \cong \varinjlim_{U \ni x_0} H^k(\Gamma(U; \mathcal{I}^\bullet)) \cong H^k(\mathcal{I}^\bullet_{x_0}) \cong H^k (\mathcal{F}^\bullet_{x_0})$$
Here, we use an injective resolution $\mathcal{I}^\bullet$ of $\mathcal{F}^\bullet$ to compute the hypercohomology in the first isomorphism. The second isomorphism follows from the fact that cohomology commutes with direct limits. Finally, as $\mathcal{F}^\bullet \to \mathcal{I}^\bullet$ is a quasi-isomorphism, it is so at the level of stalks as well. The final isomorphism also follows.

Suppose $(x_0, 0) \in 0_M \subset T^*M$. Let $\varphi$ be the zero function. As $U \cap \{\varphi < 0\} = \emptyset$, the target of the restriction map in Equation \ref{eq-microstalk} is zero. From the computation of the domain of the map above, we obtain that the restriction map is not an isomorphism if and only if $H^k(\mathcal{F}^\bullet_{x_0}) \ncong 0$ for some $k$. This holds if and only if $\mathcal{F}^\bullet_{x_0}$ is not quasi-isomorphic to the zero complex. Thus, $\mathrm{SS}(\mathcal{F}^\bullet) \cap 0_M$ is precisely $\mathrm{supp}(\mathcal{F}^\bullet)$.
\end{proof}

An important property of the microsupport is that it is a co-isotropic (cf. discussion before Lemma \ref{lem-lagcomp}) subset of $(T^*M, \omega)$, where $\omega$ is the canonical symplectic form in Example \ref{eg-phase}. Since the microsupport is in general not a submanifold of $T^*M$ but rather a singular subset, we need to define what we mean by this. Let us start with the following definition. These are both special cases of \emph{normal cones}, defined in \cite[Definition 4.1.1]{ksbook}.

\begin{definition}[Whitney tangent and secant cones] Let $M$ be a manifold, $X, X_1, X_2 \subset M$ be subsets of $M$ and $x \in M$ be a point.
\begin{enumerate}
\item The \emph{tangent cone} $C_x(X) \subset TM$ is defined as follows: we say $(x, v) \in C_x(X)$ if and only if there exists a sequence of points $\{x_n\} \subset X$ and a sequence of real numbers $\{c_n\}$ such that $x_n \to x$ and $c_n (x_n - x) \to v$ in some local coordinate system around $x$. 
\item The \emph{secant cone} $C_x(X_1, X_2) \subset TM$ is defined as follows: we say $(x, v) \in C_x(X_1, X_2)$ if and only if there exists sequences of points $\{x_n\}, \{y_n\} \subset X$ and a sequence of real numbers $\{c_n\}$ such that $x_n \to x, y_n \to x$ and $c_n(x_n - y_n) \to v$ in some local coordinate system around $x$.
\end{enumerate}
\end{definition}

\begin{definition}\cite[Definition 6.5.1]{ksbook}\label{def-coisoset} Let $Z \subset T^*M$ be a subset, and $\omega$ be the canonical symplectic form on $T^*M$ as defined in Example \ref{eg-phase}. We say $Z$ is \emph{co-isotropic at $p \in Z$} if for all vectors $(p, v) \in T_p(T^*M)$ such that $C_p(Z, Z) \subset \ker i_v \omega$, we have $(p, v) \in C_p(Z)$. We shall say $Z$ is \emph{co-isotropic} if it is co-isotropic at all points.\end{definition}

If $Z$ is a submanifold, then $C_p(Z) = C_p(Z, Z) = T_p Z$. Thus, in this case the above definition reduces to the usual notion of co-isotropic submanifolds. The following fundamental result of Kashiwara and Schapira states that the microsupport of any bounded complex of sheaves is a conic co-isotropic subset of $T^*M$.

\begin{theorem}\cite[Theorem 6.5.4]{ksbook}\label{thm-kscoiso} Let $\mathcal{F}^\bullet \in \mathrm{Ob}(D^b(M))$ be a complex of sheaves. Then $\mathrm{SS}(\mathcal{F}^\bullet) \subset T^*M$ is a co-isotropic subset.
\end{theorem}

\subsection{Constructible sheaves and conic Lagrangians} 

In this section we define an important class of complexes of sheaves on a manifold $M$ called constructible sheaves. They are characterized by the property that the cohomology of such complexes restricts to local systems on each stratum of a sufficiently regular stratification of $M$. We begin with the definition of a stratification, and a summary of the important flavours of regularity for stratifications that we will encounter.

\begin{definition}[Stratification] A \emph{stratification} $\Sigma = \{\Sigma_i : i \in I\}$ of a smooth manifold $M$ is a locally finite partition $M = \bigsqcup_{i \in I} \Sigma_i$ by locally closed smooth submanifolds or \emph{strata}, such that closure of each stratum is a union of strata. If $M$ is real analytic, we say a stratification $\Sigma$ is \emph{subanalytic} if $\Sigma_i \subset M$ are subanalytic sets.
\end{definition}

\begin{definition}[Conormal set]\label{def-conset}For a submanifold $S \subset M$, the \emph{conormal covectors} of $S$ in $M$ are covectors $(x, \xi) \in T^*M$ such that $x \in S$, and $\xi|_{T_x S} = 0$. We denote the set of all conormal covectors of $S$ as $T_S^*M$, which is naturally a bundle over $S$, called the \emph{conormal bundle} of $S$ in $M$. For a stratification $\Sigma = \{\Sigma_i : i \in I\}$ of $M$, we define the \emph{conormal set} of $\Sigma$ to be the union $T^*_{\Sigma} M := \bigsqcup_{i \in I} T^*_{\Sigma_i} M$ of the conormal bundles of each stratum.
\end{definition}

\begin{definition}[Regularity of stratification]\label{def-stratreg} Let $\Sigma = \{\Sigma_i : i \in I\}$ be a stratification of a manifold $M$. $\Sigma$ is said to be a \emph{Whitney stratification} if it satisfies Condition $(a)$ and Condition $(b)$ below. Moreover, if $M$ is real analytic and $\Sigma$ is subanalytic, it is said to be a \emph{$\mu$-stratification} (see, \cite[Definition 8.3.19]{ksbook}) if it satisfies Condition $(\mu)$ below.
\begin{enumerate}
\item[$(a)$] For any pair of strata $\Sigma_i, \Sigma_j$ such that $\Sigma_i \subset \overline{\Sigma_j}$ and for any sequence $\{x_k\} \subset \Sigma_j$ converging to $x \in \Sigma_i$, if the sequence of tangent planes $\{T_{x_k} \Sigma_j\}$ converges to a plane $\tau \subset T_x M$, then $T_x \Sigma_i \subset \tau$. Equivalently, $T^*_{\Sigma} M \subset T^*M$ is a closed subset.
\item[$(b)$] For any pair of strata $\Sigma_i, \Sigma_j$ such that $\Sigma_i \subset \overline{\Sigma_j}$ and for any pair of sequences $\{x_k\} \subset \Sigma_j$, $\{y_k\} \subset \Sigma_i$ both converging to $x \in \Sigma_i$, if the sequence of tangent planes $\{T_{x_k} \Sigma_j\}$ converges to a plane $\tau \subset T_x M$ and the sequence of secant lines $\{\overline{x_k y_k}\}$ converges to a line $\ell \subset T_x M$, then $\ell \subset \tau$.
\item[$(\mu)$] For any pair of sequences $\{(x_k, \xi_k)\}, \{(x_k, \eta_k)\} \subset T^*_{\Sigma} M$ such that $\{x_k\}, \{y_k\} \subset M$ converges to $x \in M$, $\{\xi_k + \eta_k\}$ converges to $\xi \in T^*_x M$ and $|x_k - y_k||\xi_k|$ converges to $0$, $(x, \xi) \in T^*_{\Sigma} M$.
\end{enumerate}
\end{definition}

\begin{remark}The notion of convergence used in the statement of the conditions in Definition \ref{def-stratreg} are Euclidean, in a local coordinate system around $x \in M$. Nonetheless, the conditions do not depend on the choice of the coordinate chart, see \cite[Section 4]{matnotes}.
\end{remark}

\begin{remark}\label{rmk-subanastrat}
The relative dependence between the conditions in Definition \ref{def-stratreg} is:
$$(\mu) \implies (b) \implies (a).$$
In particular, a $\mu$-stratification is a Whitney stratification, see \cite[Exercise VIII.12]{ksbook}. Any locally finite covering of a real analytic manifold by subanalytic subsets admits a refinement to a $\mu$-stratification (see, \cite[Theorem 8.3.20]{ksbook}) and hence to a Whitney stratification.
\end{remark}

\begin{definition}[Constructible sheaf]Let $M$ be a manifold, $\Sigma = \{\Sigma_i : i \in I\}$ be a stratification of $M$ and $\mathcal{F}^\bullet \in \mathrm{Ob}(D^b(M))$ be a complex of sheaves. 
\begin{enumerate}
\item $\mathcal{F}^\bullet$ is \emph{weakly constructible} with respect to $\Sigma$ if for all $i \in I$, $k \in \Bbb Z$, $H^k(\mathcal{F}^\bullet)|_{\Sigma_i}$ is a locally constant sheaf on $\Sigma_i$. 
\item $\mathcal{F}^\bullet$ is \emph{constructible} with respect to $\Sigma$ if $\mathcal{F}^\bullet$ is weakly constructible with respect to $\Sigma$ and moreover, for all $x \in M$, $i \in \Bbb Z$, $H^k(\mathcal{F}^\bullet)_x$ is a finite dimensional $\mathbf{k}$-vector space.
\end{enumerate}
\end{definition}

We shall now mention a result which describes the microsupport of constructible sheaves. We start with the following definition:

\begin{definition}\cite[Definition 3.8.9]{ksbook}\label{def-lagset} Let $M$ be a real analytic manifold, and let $(T^*M, \omega)$ denote the cotangent bundle of $M$ equipped with the canonical symplectic structure. A subanalytic subset $Z \subset T^*M$ is \emph{isotropic} if $\omega$ restricts to zero on the nonsingular locus $Z_{\emph{reg}} \subset Z$. If $Z$ is both co-isotropic and isotropic, we say $Z$ is a Lagrangian subset.\end{definition}

\begin{example}\label{eg-lagset}For a submanifold $S \subset M$, the conormal bundle $T^*_S M$ (Definition \ref{def-conset}) is an isotropic submanifold of $(T^*M, \omega)$. Indeed, for $(x, \xi) \in T^*_S M$ and $v \in T_{(x, \xi)}T^*_SM$, 
$$\theta(v) = \xi(\pi_*(v)) = 0,$$
as $\xi$ is conormal to $S$. Thus, $\theta$ restricts to $0$ on $T^*_S M \subset T^*M$, and likewise for $\omega = -d\theta$. As a consequence, we deduce that for any subanalytic stratification $\Sigma$ of $M$, the conormal set $T^*_{\Sigma} M \subset (T^*M, \omega)$ is a conic isotropic subset. 
\end{example}

\begin{theorem}\cite[Theorem 8.4.2]{ksbook}\label{thm-ksiso} Let $M$ be a real analytic manifold, and $\mathcal{F}^\bullet \in \mathrm{Ob}(D^b(M))$ be a complex of sheaves. $\mathcal{F}^\bullet$ is weakly constructible with respect to a subanalytic stratification if and only if $\mathrm{SS}(\mathcal{F}^\bullet) \subset T^*M$ is a closed conic subanalytic Lagrangian subset.

Furthermore, if $\Sigma$ is a subanalytic $\mu$-stratification of $M$, then $\mathcal{F}^\bullet$ is weakly constructible with respect to $\Sigma$ if and only if $\mathrm{SS}(\mathcal{F}^\bullet) \subset T^*_{\Sigma} M$.\end{theorem}

We give the following relevant corollary of Theorem \ref{thm-kscoiso}. Note that we do not require $M$ to be analytic for this corollary.

\begin{corollary}\cite[Corollary 1.3.9]{guipaper}\label{cor-sslageq} Let $M$ be a smooth manifold, and $L \subset T^*M \setminus 0_M$ be a conic connected Lagrangian submanifold. If $\mathcal{F}^\bullet \in \mathrm{Ob}(D^b(M))$ is a complex of sheaves such that $\mathrm{SS}(\mathcal{F}^\bullet) \setminus 0_M \subseteq L$ is a nonempty subset, then $\mathrm{SS}(\mathcal{F}^\bullet) \setminus 0_M = L$.\end{corollary}

\begin{proof}Suppose $U := L \setminus \mathrm{SS}(\mathcal{F}^\bullet)$ is nonempty. Then the boundary $\partial U$ of $U$ in $L$ is nonempty by the connectedness hypothesis on $L$. Thus, we may choose a point $p_0 \in \partial U$. Let $V$ be a chart around $p_0$ in $L$. We choose another point $p_1 \in U \cap V$. Let $B \subset V$ denote the open ball of maximal radius centered at $p_1$ such that $B \cap \mathrm{SS}(\mathcal{F}^\bullet) = \emptyset$. Thus, $\partial B \cap \mathrm{SS}(\mathcal{F}^\bullet)$ is nonempty. Hence, we may select yet another point $p \in \partial B \cap \mathrm{SS}(\mathcal{F}^\bullet)$. 

By Theorem \ref{thm-kscoiso}, $\mathrm{SS}(\mathcal{F}^\bullet) \subset T^*M$ is a co-isotropic subset. Observe that,
$$C_p(\mathrm{SS}(\mathcal{F}^\bullet)) \subset C_p(L \setminus B),$$
which is a half-space in $T_p L$. Meanwhile,
$$C_p(\mathrm{SS}(\mathcal{F}^\bullet), \mathrm{SS}(\mathcal{F}^\bullet)) \subset C_p(L, L) = T_p L.$$
By hypothesis, $L$ is a Lagrangian submanifold of $T^*M$. Therefore, for any $(p, v) \in T_pM$ such that $T_p L \subset \ker i_v \omega$, we have $(p, v) \in T_p L$ as well. Since $\mathrm{SS}(\mathcal{F}^\bullet)$ is co-isotropic, we must have $T_p L \subset C_p(\mathrm{SS}(\mathcal{F}^\bullet))$, by Definition \ref{def-coisoset}. But this is absurd, as we demonstrated $C_p(\mathrm{SS}(\mathcal{F}^\bullet))$ is a half-space in $T_p L$. This leads to the desired contradiction. \end{proof}

\subsection{Sheaf-theoretic Legendrian isotopy invariants}

We shall now define a category of sheaves over a manifold $M$, associated to a Legendrian submanifold in the cosphere bundle $T^\infty M$ which is an isotopy invariant of the Legendrian submanifold. Throughout this section we shall implicitly assume $M$ is a real analytic manifold. 

\begin{definition}[Cosphere bundle]\label{def-cosphere}The \emph{cosphere bundle} of a manifold $M$ is defined as,
$$T^\infty M := (T^*M \setminus 0_M)/\Bbb R_{>0},$$
where $0_M \subset T^*M$ is the zero section, and the $\Bbb R_{>0}$-action is one by scaling $c \cdot (p, \xi) = (p, c \xi)$. We equip $T^\infty M$ with the canonical contact structure given by kernel of the tautological $1$-form (cf. Example \ref{eg-phase}). We also denote the cotangent bundle projection as,
\begin{gather*}\pi : T^\infty M \to M, \ \pi(x, [\xi]) = x.\end{gather*} 
\end{definition}

\begin{remark}\label{rmk-cosphere}The projection $\pi : T^\infty M \to M$ is a Legendrian submersion (Definition \ref{def-legsub}). By Lemma \ref{lem-loctrivlegs}, $\pi$ is locally contactomorphic to the Euclidean front projection by a fiber-preserving fashion. Thus, we shall often refer to $M$ as the \emph{front} for $T^\infty M$, and $\pi$ as the \emph{front projection}. 
\end{remark}

\noindent
Let $\Lambda \subset T^\infty M$ be a compact subanalytic Legendrian submanifold. Also, let
$$\Bbb R_{>0}\Lambda = \{(x, c\xi) \in T^*M : (x, [\xi]) \in \Lambda, c > 0\}$$
denote the \emph{positive cone} on $\Lambda$. Then, 
$$L(\Lambda) := 0_M \cup \Bbb R_{>0} \Lambda \subset (T^* M, \omega),$$
defines a conic subanalytic Lagrangian subset in the sense of Definition \ref{def-lagset}. Since $\Lambda$ is a Legendrian submanifold, for any $(x, [\xi]) \in \Lambda$ such that $y = \pi(x)$ is a smooth point of the front $\pi(\Lambda) \subset M$, $\xi$ restricts to the zero functional on $T_y \pi(\Lambda) \subset T_y M$. Therefore, $\xi$ is a conormal covector (Definition \ref{def-conset}) to $\pi(\Lambda)$. Thus, $L(\Lambda)$ is the closure of the set of all positive conormal covectors to $\pi(\Lambda) \subset M$ at the smooth points. We shall simply call $L(\Lambda)$ the set of all positive conormal covectors to $\pi(\Lambda) \subset M$.

As $\pi(\Lambda) \subset M$ is subanalytic, we may choose a subanalytic Whitney stratification $\Sigma$ of $M$ such that $\pi(\Lambda)$ is a union of strata of $\Sigma$ by Remark \ref{rmk-subanastrat}. We may also ensure that the smooth locus and the cuspidal locus of $\pi(\Lambda)$ each constitute a single stratum of $\Sigma$. As a Whitney stratification satisfies Condition $(a)$ in Definition \ref{def-stratreg}, we have $L(\Lambda) \subset T^*_{\Sigma} M$.

\begin{definition}[Microlocal sheaf category of a Legendrian]\label{def-microcat}We define $D^b_{\Lambda}(M)$ to be the full subcategory of $D^b(M)$ consisting of complexes of sheaves $\mathcal{F}^\bullet \in \mathrm{Ob}(D^b(M))$ such that $\mathcal{F}^\bullet$ is constructible with respect to some subanalytic stratification of $M$, $\mathrm{SS}(\mathcal{F}^\bullet) \subseteq L(\Lambda)$ and $\mathrm{supp}(\mathcal{F}^\bullet) \subset M$ is compact.
\end{definition}

The following crucial theorem of Guillermou, Kashiwara and Schapira \cite{gks} ensures that $D^b_{\Lambda}(M)$ is an isotopy-invariant of the Legendrian submanifold $\Lambda$.

\begin{theorem}\label{thm-gks} Let $M$ be a manifold, $\Lambda_0, \Lambda_1 \subset T^\infty M$ be a pair of compact subanalytic Legendrian submanifolds which differ by a Legendrian isotopy. Then there is an equivalence of categories,
$$K : D^b_{\Lambda_0}(M) \stackrel{\simeq}{\longrightarrow} D^b_{\Lambda_1}(M).$$
Moreover, $K$ preserves the microsupport, i.e., $\mathrm{SS}(K(\mathcal{F}^\bullet)) = K(\mathrm{SS}(\mathcal{F}^\bullet))$.\end{theorem}

We shall use Theorem \ref{thm-gks} to prove existence of Legendrian submanifolds of the cosphere bundle which are not loose. Let us start by setting the background for a crucial lemma. Suppose $\mathcal{Z} \subset (\Bbb R^3, \xi_{\mathrm{std}})$ is a Legendrian zig-zag (Definition \ref{def-legzigzag}). We may embed $(\Bbb R^3, \xi_{\mathrm{std}})$ isocontactly as a subset of $T^\infty \Bbb R^2$, defined by
$$\{(x, [\xi]) \in T^\infty \Bbb R^2 : \xi = \xi_x dx + \xi_y dy, \xi_y > 0\}.$$
Under this identification, the front projection $\pi : T^\infty \Bbb R^2 \to \Bbb R^2$ (Remark \ref{rmk-cosphere}) restricts to the standard front projection $\pi : \Bbb R^3 \to \Bbb R^2$ of the contact Euclidean space. Let $Z \subset \Bbb R^2$ denote the front projection $Z = \pi(\mathcal{Z})$ of the Legendrian zig-zag. Then $L(\mathcal{Z}) \subset T^*\Bbb R^2$ consists of the positive conormal covectors to $Z \subset \Bbb R^2$. The following result due to Shende, Treumann and Zaslow \cite[Proposition 5.8]{stz} essentially states that $L(\mathcal{Z})$ cannot appear as microsupport of a constructible sheaf on $\Bbb R^2$. We reformulate it in slightly more concrete terms via the conormal interpretation.

\begin{lemma}[Legendrian zig-zags are not microsupports]\label{lem-zigsupp}Let $Z \subset \Bbb R^2$ be the front projection of a Legendrian zig-zag. Let $\Sigma_{\rm{zig}} = \{c_1, c_2, e_1, e_2, e_3, f_1, f_2\}$ be the stratification of $\Bbb R^2$ defined by declaring the cusps $c_1, c_2$ of $Z$ as two $0$-strata, the arcs $e_1, e_2, e_3$ of $Z$ as three $1$-strata and the complementary domains $f_1, f_2$ of $\Bbb R^2 \setminus Z$ as two $2$-strata. Let $\mathcal{F}^\bullet \in \mathrm{Ob}(D^b(\Bbb R^2))$ be a complex of sheaves constructible with respect to $\Sigma$ such that $\mathrm{SS}(\mathcal{F}^\bullet)$ is contained in the set of positive conormal covectors to $Z$. See, Figure \ref{fig-zigsupp}. Then, $\mathrm{SS}(\mathcal{F}^\bullet) = 0_{\Bbb R^2} \subset T^*\Bbb R^2$.\end{lemma}

\begin{proof}For any stratum $S \in \Sigma_{\rm{zig}}$, let us define the \emph{star} of $S$ as
$$\mathrm{star}(S) := \bigcup_{\{L \in \Sigma_{\rm{zig}} : S \subset \overline{L}\}} L.$$
Observe, each stratum in $\Sigma_{\rm{zig}}$ as well as their stars are contractible. Thus, there is a quasi-isomorphism $\mathcal{F}^\bullet(\mathrm{star}(S)) \simeq \mathcal{F}^\bullet(S)$. On the other hand, $\mathcal{F}^\bullet(S) \simeq \mathcal{F}^\bullet_x$, for any $x \in S$, as $\mathcal{F}^\bullet$ is constructible. We shall fix notation by identifying these three chain complexes by fixed quasi-isomorphisms, for every stratum $S \in \Sigma_{\rm{zig}}$, once and for all. Let us define a directed graph (or \emph{quiver}) by assigning vertices corresponding to the strata of $\Sigma_{\mathrm{zig}}$, and placing a directed edge from the vertex corresponding to a stratum $S \in \Sigma_{\mathrm{zig}}$ to the vertex corresponding to a stratum $L \in \Sigma_{\mathrm{zig}}$ if $S \subset \overline{L}$. The resulting graph is drawn in Figure \ref{fig-zigrep}. 

\begin{figure}[h]
\centering
\includegraphics[scale=0.5]{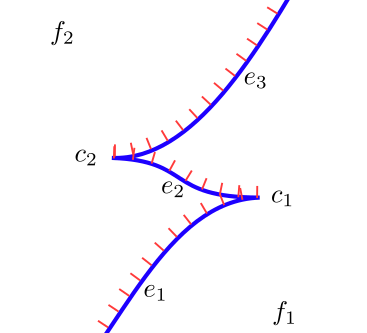}
\caption{Front projection of a Legendrian zig-zag $Z \subset \Bbb R^2$ (in blue), together with the positive conormal covectors to $Z$ (in red), and labels indicating the induced stratification by cusps, smooth arcs and the complementary regions of $Z$.}
\label{fig-zigsupp}
\end{figure}

Then, $\mathcal{F}^\bullet$ defines a \emph{representation} of this quiver in the derived category of complexes of $\mathbf{k}$-vector spaces. More concretely, to each vertex corresponding to a stratum $S$, we assign the complex $\mathcal{F}^\bullet(S)$, and for each edge corresponding to the containment $S \subset \overline{L}$, we assign the restriction map $\mathrm{res}(S, L) : \mathcal{F}^\bullet(S) \to \mathcal{F}^\bullet(L)$ given by,
$$\mathcal{F}^\bullet(S) \simeq \mathcal{F}^\bullet(\mathrm{star}(S)) \to \mathcal{F}^\bullet(\mathrm{star}(L)) \simeq \mathcal{F}^\bullet(L),$$
where the middle map is the restriction of sections along the natural inclusion $\mathrm{star}(S) \subset \mathrm{star}(L)$. This produces a diagram of complexes of $\mathbf{k}$-vector spaces.

By the condition that $\mathrm{SS}(\mathcal{F}^\bullet)$ is contained in the set of positive conormal covector to $Z \subset \Bbb R^2$, we obtain that all the restriction maps in this diagram are quasi-isomorphisms except perhaps for the five maps $\mathrm{res}({e_1, f_2}), \mathrm{res}({c_1, e_2}), \mathrm{res}({e_2, f_1}), \mathrm{res}({c_2, e_3})$ and $\mathrm{res}({e_3, f_2})$, which point towards the direction of an upward conormal covector to a stratum. These are indicated by red arrows in Figure \ref{fig-zigrep}. Let $V_i^\bullet := \mathcal{F}^\bullet(f_i)$, $i = 1, 2$. Since the restriction maps corresponding to the black arrows in Figure \ref{fig-zigrep} are quasi-isomorphisms, we obtain $\mathcal{F}^\bullet(e_1),\mathcal{F}^\bullet(c_1), \mathcal{F}^\bullet(e_3)$ are quasi-isomorphic to $V_1^\bullet$, and $\mathcal{F}^\bullet(e_2), \mathcal{F}^\bullet(c_2)$ are quasi-isomorphic to $V_2^\bullet$. We have three morphisms $p \in \mathrm{Hom}(V_1^\bullet, V_2^\bullet), q \in \mathrm{Hom}(V_2^\bullet, V_1^\bullet)$ and $r \in \mathrm{Hom}(V_1^\bullet, V_2^\bullet)$ in the derived category of complexes of $\mathbf{k}$-vector spaces, corresponding to the roofs:
\begin{gather*}
\mathcal{F}^\bullet(f_1) \stackrel{\mathrm{res}({e_1, f_1})}{\longleftarrow} \mathcal{F}^\bullet(e_1) \stackrel{\mathrm{res}({e_1, f_2})}{\longrightarrow} \mathcal{F}^\bullet(f_2),\\
\mathcal{F}^\bullet(f_1) \stackrel{\mathrm{res}({e_2, f_1})}{\longleftarrow} \mathcal{F}^\bullet(e_2) \stackrel{\mathrm{res}({e_2, f_1})}{\longrightarrow} \mathcal{F}^\bullet(f_1),\\
\mathcal{F}^\bullet(f_1) \stackrel{\mathrm{res}({e_3, f_1})}{\longleftarrow} \mathcal{F}^\bullet(e_3) \stackrel{\mathrm{res}({e_3, f_2})}{\longrightarrow} \mathcal{F}^\bullet(f_2).
\end{gather*}
Notice that $q \circ p \in \mathrm{Hom}(V_1^\bullet, V_1^\bullet)$ and $r \circ q \in \mathrm{Hom}(V_2^\bullet, V_2^\bullet)$ are equivalent to the morphism corresponding to the roofs:
\begin{gather*}
\mathcal{F}^\bullet(f_1) \stackrel{\mathrm{res}(c_1, f_1)}{\longleftarrow} \mathcal{F}^\bullet(c_1) \stackrel{\mathrm{res}(c_1, f_1)}{\longrightarrow} \mathcal{F}^\bullet(f_1), \\
\mathcal{F}^\bullet(f_2) \stackrel{\mathrm{res}(c_2, f_1)}{\longleftarrow} \mathcal{F}^\bullet(c_2) \stackrel{\mathrm{res}(c_1, f_2)}{\longrightarrow} \mathcal{F}^\bullet(f_2).
\end{gather*}
These are equivalent to identity morphisms $1_{V_1^\bullet}, 1_{V_2^\bullet}$, respectively. Thus, $q \circ p \simeq 1_{V_1^\bullet}$ and $r \circ q \simeq 1_{V_2^\bullet}$, which implies $V_1^\bullet \simeq V_2^\bullet$ and $p, q, r$ are all quasi-isomorphisms. Therefore, $\mathrm{res}(e_1, f_2), \mathrm{res}(e_2, f_1)$ and $\mathrm{res}(e_3, f_3)$ are quasi-isomorphisms. By commutativity of the diagram, we conclude $\mathrm{res}(c_1, e_2)$ and $\mathrm{res}(c_2, e_3)$ are quasi-isomorphisms as well. Thus, all the restriction maps of $\mathcal{F}^\bullet$ are quasi-isomorphisms. Consequently, $\mathrm{SS}(\mathcal{F}^\bullet) \subset T^*\Bbb R^2$ is the zero section.
\end{proof}

\begin{figure}[t]
\centering
\includegraphics[scale=0.65]{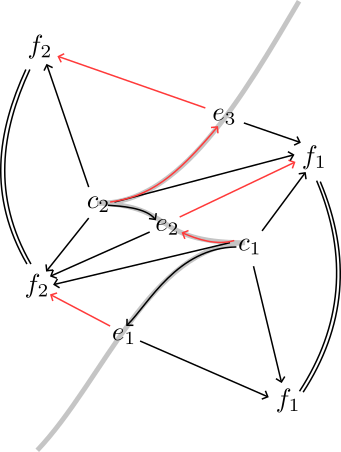}
\caption{Directed graph corresponding to the stratification of $\Bbb R^2$ by the cusps, smooth arcs and complementary regions of $Z$ (in grey). The edges corresponding to the positive conormal covectors are emphasized (in red).}
\label{fig-zigrep}
\end{figure}

\begin{proposition}[Loose Legendrians are not microsupports]\label{prop-suploose}Let $\Lambda \subset T^\infty M$ be  a subanalytic Legendrian submanifold which is loose, and $\mathcal{F}^\bullet \in \mathrm{Ob}(D^b_{\Lambda}(M))$. Then, $\mathrm{SS}(\mathcal{F}^\bullet) = 0_M \subset T^*M$. \end{proposition}

\begin{proof}Choose a point $x_0 \in \Lambda$ and, if necessary, isotope $\Lambda$ slightly to a subanalytic Legendrian so that $\pi|_{\Lambda}$ is a local immersion near $x_0$. Let $p_0 = \pi(x_0)$. Thus, $\pi(\Lambda) \subset M$ is subanalytic subset which is smooth around a neighborhood of $p_0$. Let $U \subset M$ be an adapted subanalytic chart for $\pi(\Lambda)$ around $p$. Thus, $(U, U \cap \pi(\Lambda)) \cong (\Bbb R^n, \Bbb R^{n-1})$. By Example \ref{eg-stabloose}, we may stabilize $\Lambda \subset T^\infty M$, by a formal Legendrian isotopy lifting a smooth isotopy of fronts compactly supported in $U$, to produce a new (subanalytic) loose Legendrian submanifold $\Lambda_0 \subset T^\infty M$. By Corollary \ref{cor-hloosepi0}, $\Lambda$ and $\Lambda_0$ are in fact Legendrian isotopic.

We choose a Whitney stratification $\Sigma$ of the front projection $\pi(\Lambda_0) \subset M$ as in the discussion following Remark \ref{rmk-cosphere}. Recall the front projection of a Legendrian zig-zag $Z \subset \Bbb R^2$ (Definition \ref{def-legzigzag}). By construction, there is a diffeomorphism onto image,
$$(D^2 \times D^{n-2}, Z \times D^{n-2}) \hookrightarrow (U, \pi(\Lambda_0) \cap U),$$
which lifts to the loose chart of $\Lambda_0$ contained in $U$. Let $D \subset U$ be the image of the interior of $D^2 \times \{0\}$ under this diffeomorphism. Then $D$ is a smoothly embedded open disk in $M$, which intersects $\pi(\Lambda_0)$ transversely (i.e. it intersects the smooth stratum and the cuspidal stratum of $\pi(\Lambda_0)$ transversely and is disjoint from all the strata of higher codimension), and there is a diffeomorphism,
$$(D, D \cap \pi(\Lambda_0)) \cong (D, Z).$$
Moreover, the stratification $\Sigma$ on $M$ induces the stratification $\Sigma_{\rm{zig}}$ on $D$, as defined in Lemma \ref{lem-zigsupp}. Recall that $0$- and $1$-strata of $\Sigma_{\rm{zig}}$ are cusps and smooth arcs of $Z \subset D$, and the $2$-strata are the two connected components of $D \setminus Z$.

Let $\mathcal{F}^\bullet \in \mathrm{Ob}(D^b_{\Lambda_0}(M))$ be a complex of sheaves. Let $i : D \to M$ be the inclusion map. Recall $\mathrm{SS}(\mathcal{F}^\bullet) \subset L(\Lambda_0)$ by definition, and $L(\Lambda_0) \subset T^*_{\Sigma} M$ by the discussion following Remark \ref{rmk-cosphere}. Hence, $\mathrm{SS}(\mathcal{F}^\bullet) \subset T^*_{\Sigma} M$, which implies by the second part of Theorem \ref{thm-ksiso} that $\mathcal{F}^\bullet$ is weakly constructible with respect to the chosen stratification $\Sigma$. Consequently, $i^*\mathcal{F}^\bullet$ is also weakly constructible with respect to the aforementioned stratification on $D$ induced by $\Sigma$. By Theorem \ref{thm-ksiso} again, we have $\mathrm{SS}(i^*\mathcal{F}^\bullet) \subset T^*_{\Sigma_{\rm{zig}}} D$. In fact, $\mathrm{SS}(i^*\mathcal{F}^\bullet)$ must be contained in the subset of $L(\Lambda_0) \cap T^*_{\Sigma_{\rm{zig}}} D$, since the restriction maps of $i^*\mathcal{F}^\bullet$ are the same as that of $\mathcal{F}^\bullet$ by constructibility. Since $D$ is transverse to $\pi(\Lambda_0)$, $L(\Lambda_0) \cap T^*_{\Sigma_{\rm{zig}}} D$ consists of precisely the positive conormal covectors to $Z \subset D$.

Therefore, by Lemma \ref{lem-zigsupp}, the $\mathrm{SS}(\mathcal{F}^\bullet) = 0_D \subset T^*D$. By once again identifying the restriction maps of $i^*\mathcal{F}^\bullet$ with that of $\mathcal{F}^\bullet$, we conclude $\mathrm{SS}(\mathcal{F}^\bullet)$ is a proper subset of $L(\Lambda_0)$. Indeed, the restriction maps of $\mathcal{F}^\bullet$ along the covectors $(x, \xi) \in T^*M$ are isomorphisms, where $x \in D \subset M$ and $\xi$ restricts to a positive conormal covector of $Z \subset D$. By applying Corollary \ref{cor-sslageq}, we conclude $\mathrm{SS}(\mathcal{F}^\bullet) = 0_M \subset T^*M$. As this is true for any complex of sheaves in $D^b_{\Lambda_0}(M)$, and $\Lambda$ is Legendrian isotopic to $\Lambda_0$, we may invoke Theorem \ref{thm-gks} to conclude the same about $\Lambda$ as the equivalence in Theorem \ref{thm-gks} preserves microsupports.\end{proof}

\subsection{Application: example of a non-loose Legendrian} In this section, we shall use the technology developed in the previous sections to demonstrate the existence of Legendrian embeddings which are not loose.

\begin{example}[Legendrian flying saucer]\label{eg-saucer} Let $(r, \theta)$ be spherical coordinates on $\Bbb R^n$. Let
$$F := \{(r, \theta, z) \in \Bbb R^n \times \Bbb R : -1 \leq r \leq 1, z^2 = (1-r^2)^3\}.$$
$F \subset \Bbb R^{n+1}$ is a topologically embedded sphere, which is smooth everywhere except the cuspidal locus along the equator $\{z = 0\} \cap F$. Thus, $F$ defines a valid front diagram. The Legendrian lift of $F$ to $\Lambda \subset (\Bbb R^{2n+1}, \xi_{\rm{std}})$ is called the \emph{Legendrian flying saucer}.
\end{example}

\begin{figure}[h]
\centering
\includegraphics[scale=0.6]{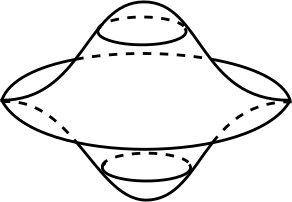}
\caption{The $2$-dimensional Legendrian flying saucer.}
\label{fig-legsaucer}
\end{figure}

\begin{proposition}\label{prop-saucernloose}The Legendrian flying saucer $\Lambda \subset (\Bbb R^{2n+1}, \xi_{\rm{std}})$ is not loose.\end{proposition}

\begin{proof}$(\Bbb R^{2n+1}, \xi_{\rm{std}})$ isocontactly embeds as a subset of $T^\infty \Bbb R^{n+1}$, defined by
$$\{(x, [\xi]) \in T^\infty \Bbb R^{n+1} : \xi = \xi_0 dx_0 + \cdots + \xi_n dx_n, \xi_n > 0\}.$$
In view of Proposition \ref{prop-suploose}, it suffices to construct a complex of sheaves $\mathcal{F}^\bullet \in D^b_{\Lambda}(\Bbb R^{n+1})$ with $\mathrm{SS}(\mathcal{F}^\bullet) = L(\Lambda)$. Consider the locally closed domain,
$$\Omega := \{(r, \theta, z) \in \Bbb R^n \times \Bbb R : -1 \leq r \leq 1, -(1-r^2)^{3/2} < z \leq (1-r^2)^{3/2}\}$$
Let $\underline{\mathbf{k}}$ denote the constant $\mathbf{k}$-valued sheaf on $\Bbb R^{n+1}$. Let $\underline{\mathbf{k}}_\Omega$ be the sheaf on $\Bbb R^{n+1}$ defined by,
$$\underline{\mathbf{k}}_{\Omega}(U) := \{s \in \mathbf{k}(U \cap \Omega) : \mathrm{supp}(s) \subset U \text{ is closed in } U\},$$
for any open subset $U \subset \Bbb R^{n+1}$, together with the obvious restriction maps. Let us consider the stratification of $\Bbb R^{n+1}$ given by $\Sigma = \{\Sigma_0, \Sigma_+, \Sigma_{-}, \Omega, \Bbb R^{n+1} \setminus \overline{\Omega}\}$, where $\Sigma_0 \subset F$ is the cuspidal equator and $\Sigma_{\pm} \subset F$ are the two open hemispheres. The stalks of $\underline{\mathbf{k}}_\Omega$ are as follows:
\[
(\underline{\mathbf{k}}_\Omega)_x =
\begin{cases} 
0, & x \in \Sigma_0\\
\mathbf{k}, & x \in \Sigma_+\\
0, & x \in \Sigma_-\\
\mathbf{k}, & x \in \Omega\\
0, & x \in \Bbb R^{n+1} \setminus \overline{\Omega}
\end{cases}
\]
Thus, $\underline{\mathbf{k}}_\Omega$ is constructible with respect to the stratification $\Sigma$. The restriction map in Definition \ref{def-microsupp} is an isomorphism for all $(x, \xi)$ such that $x \in \Omega \cup (\Bbb R^{n+1} \setminus \overline{\Omega})$. For $x \in \Sigma_+$ (resp. $x \in \Sigma_-$) and $\xi \in T^*_x\Bbb R^{n+1}$ a positive (i.e., pointing towards $\partial_z$) conormal to $F$, the restriction map in $0$-th cohomology in Definition \ref{def-microsupp} is equivalent to the map $\mathbf{k} \to 0$ (resp. $0 \to \mathbf{k}$), which is not an isomorphism. For any other $\xi$, the restriction map on a small enough open ball around $x$ is equivalent to the map in cohomology $H^k(\Bbb R^{n+1}; \mathbf{k}) \to H^k(Q; \mathbf{k})$ for some sector $Q \subset \Bbb R^{n+1}$ spanned by two hyperplanes, which is an isomorphism. Finally, for $x \in \Sigma_0$, and $\xi = \partial_z$, the restriction map in $0$-th cohomology is equivalent to the map $0 \to \mathbf{k}$, which is not an isomorphism. For any other $\xi \in T^* \Bbb R^{n+1}$, it is an isomorphism by analogous reasoning as before. 

Thus, $\mathrm{SS}(\underline{\mathbf{k}}_\Omega)$ is the set of positive conormal covector to $F \subset \Bbb R^{n+1}$, which is equal to $L(\Lambda) \subset T^*\Bbb R^{n+1}$. In conclusion, $\underline{\mathbf{k}}_\Omega \in D^b_\Lambda(\Bbb R^{n+1})$ and $\mathrm{SS}(\underline{\mathbf{k}}_\Omega) = L(\Lambda)$, as desired.
\end{proof}

\begin{remark}The proof above shows that for a Legendrian flying saucer $\Lambda \subset (\Bbb R^{2n+1}, \xi_{\rm{std}})$, the microlocal sheaf category $D^b_{\Lambda}(\Bbb R^{n+1})$ is isomorphic to the bounded derived category of complexes of $\mathbf{k}$-vector spaces, as the objects are of the form $\underline{\mathbf{k}}_{\Omega} \otimes V^{\bullet}$ for some cochain complex of vector spaces $V^\bullet$ with bounded cohomology. On the other hand, for a loose Legendrian $\Lambda' \subset (\Bbb R^{2n+1}, \xi_{\rm{std}})$, the microlocal sheaf category $D^b_{\Lambda'}(\Bbb R^{n+1}) \cong 0$ is trivial.\end{remark}

\subsection{Application: contact non-squeezing for loose charts} 

In symplectic geometry, the celebrated non-squeezing theorem of Gromov \cite[Section 0.3]{gropaper} shows that an open Euclidean ball cannot be symplectically embedded in an open Euclidean cylinder of smaller radius. More precisely, let $(\Bbb R^{2n}, \omega_{\rm{std}})$ be the Euclidean space with coordinates $(q_1, \cdots, q_n, p_1, \cdots, p_n)$ equipped with the standard symplectic form $\omega_{\rm{std}} = \sum_{i = 1}^n dq_i \wedge dp_i$. Let $D^{2n}(0; R) := \{|q|^2 + |p|^2 < R\}$ and $C_r := \{|q| < r\}$. The theorem states that if $R > r$, then there is no symplectic embedding of $D^{2n}(0; R)$ inside $C_r$. In this section, we prove a contact analogue of Gromov's non-squeezing theorem for loose charts. 

Let us recall some terminology from Definition \ref{def-loose}. Let $n \geq 2$, and $\Bbb R^{2n+1} = \Bbb R^3 \times \Bbb R^{2n-2}$ be equipped with coordinates $(x, y, z, p, q)$ and contact form $\mathrm{\alpha}_{\rm{std}} = \alpha_0 - \lambda$, where $\alpha_0 = dz - y dx$ and $\lambda = \sum_{i = 1}^n p_i dq_i$. Recall the following notation:
\begin{enumerate}
\item Let $\mathcal{Z}_a \subset (\Bbb R^3, \alpha_0)$ be a Legendrian zig-zag of action $a$,
\item Let $C \subset \Bbb R^3$ be interior of a compact cube containing $\mathcal{Z}_a$,
\item Let $B_\rho = \{(p, q) : |p| < \rho, |q| < \rho\} \subset \Bbb R^{2n-2}$,
\item Let $J_\rho := \{p = 0\} \subset B_\rho$. 
\end{enumerate}
Recall that the size parameter associated to the pair $(C \times B_\rho, \mathcal{Z}_a \times J_\rho)$ is defined to be $\rho^2/a$, and we say the pair is a loose chart if the quantitative condition $\rho^2/a > 1/2$ is satisfied. We introduce the following terminology for the sake of convenience during subsequent discussion.

\begin{definition}$(C \times B_{\rho}, \mathcal{Z}_a \times J_{\rho})$ will be called a \emph{pseudo-loose chart} if $\rho^2/a < 1/2$.\end{definition}

\begin{proposition}\label{prop-loosenonsq}A loose chart does not admit a contact relative embedding in a pseudo-loose chart. More precisely, let $(C' \times B_{\rho'}, \mathcal{Z}_{a'} \times J_{\rho'})$ with size parameter $\rho'^2/a' > 1/2$ be a loose chart and $(C \times B_{\rho}, \mathcal{Z}_a \times J_{\rho})$ with size parameter $\rho^2/a < 1/2$ be a pseudo-loose chart. Then, there does not exist a contact embedding 
$$\varphi : C' \times B_{\rho'} \to C \times B_\rho,$$
such that $\varphi(\mathcal{Z}_{a'} \times J_{\rho'}) \subset \mathcal{Z}_a \times J_\rho$.\end{proposition}

\begin{proof}Let $\varepsilon > 0$. Consider the front diagram in $\Bbb R^{n+1} = \Bbb R^2_{xz} \times \Bbb R^{n-1}_q$ defined by the following:
$$F := \left \{(x, z, q) \in \Bbb R^{n+1} : 0 < |q| < 2\rho, \left (x, z - \frac{a}{2\rho}|q| \right ) \in \pi(\mathcal{Z}_a) \right \}.$$
We may imagine $F$ as a parametric family of front projections in $\Bbb R^2_{xz}$ of a doubled link of $\mathcal{Z}_a \subset (\Bbb R^3, \alpha_0)$ consisting of $\mathcal{Z}_a$ as well as a pushoff of it given by a vector field proportional to the Reeb vector field $\partial/\partial z$. The family is radially parametrized by $D^{n-1}(0; 2\rho) \subset \Bbb R^{n-1}_q$, with magnitude of the aforementioned normal vector field being precisely $|q|$. See, Figure \ref{fig-pseudoloose}.

\begin{figure}[h]
\centering
\includegraphics[scale=0.55]{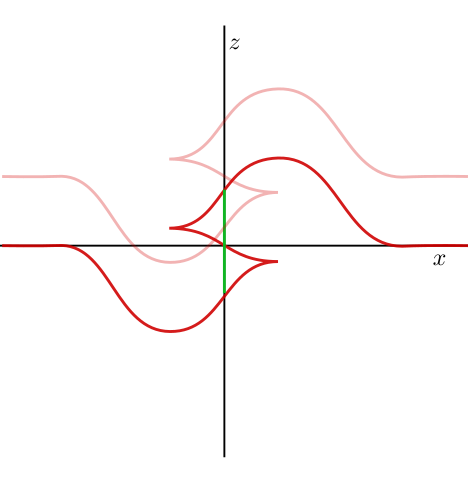}
\caption{Front projection of a Legendrian zig-zag $\mathcal{Z}_a$ (red) of action $a$, along with its Reeb chord (green), and a $z$-translated copy of $\pi(\mathcal{Z}_a)$ (pink).}
\label{fig-pseudoloose}
\end{figure}

Let $\sigma < 1/2$. We claim that the Legendrian lift $\Lambda \subset (\Bbb R^{2n+1}, \alpha_{\rm{std}})$ of $F$ contains a pseudo-loose chart of size parameter $\sigma$. To see this, first observe that $\Lambda \subset \Bbb R^{2n+1}$ is a two-component Legendrian link $\Bbb R^n \times \{0, 1\} \hookrightarrow \Bbb R^{2n+1}$, see Figure \ref{fig-interweave}.

\begin{figure}[h]
\centering
\includegraphics[scale=0.2]{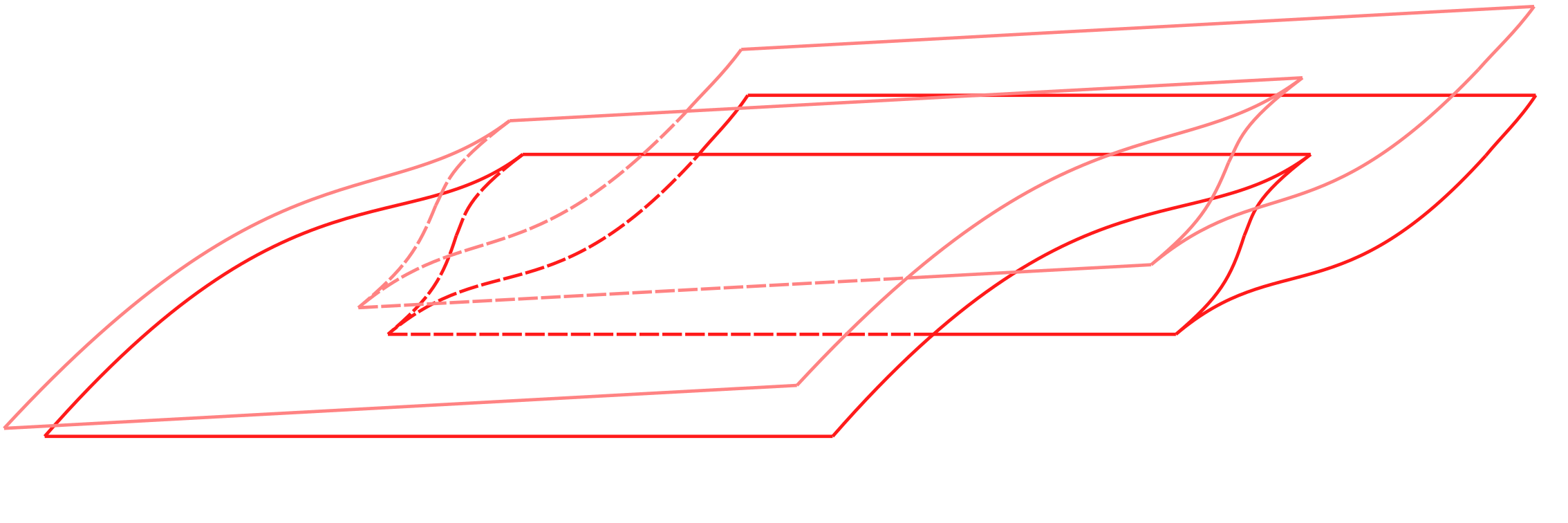}
\caption{Front projection $F \subset \Bbb R^{n+1}$ of the two-component Legendrian link $\Lambda \subset \Bbb R^{2n+1}$. Along the $q$--direction, the red component has $0$ slope while the pink component has slope $a/(2\rho)$.}
\label{fig-interweave}
\end{figure}

Choose $\rho := (a \sigma)^{1/2}$. Let $C \subset \Bbb R^3$ be a cube containing the lift of $\mathcal{Z}_a$, indicated as the red zig-zag in Figure \ref{fig-pseudoloose}. Also, let $B_\rho := \{|q| < \rho, |p| < \rho\} \subset \Bbb R^{2n-2}$. Notice that along the component of $\Lambda$ indicated by the pink component of the front in Figure \ref{fig-interweave}, we have 
$$|p| = \left |\frac{\partial z}{\partial q} \right | = \frac{a}{2\rho} > \rho,$$
since $\sigma < 1/2$. Therefore, $C \times B_\rho \subset \Bbb R^{2n+1}$ does not intersect the pink component of $\Lambda$. Consequently, $(C \times B_{\rho}, \mathcal{Z}_a \times J_{\rho})$ is an adapted chart containing the component $\mathcal{Z}_a \times J_{\rho}$ of $\Lambda$ indicated by the red component of the front in Figure \ref{fig-interweave}. This is our required pseudo-loose chart with size parameter $\sigma < 1/2$.

Next, we show that $F$ can be smoothly embedded in a front projection of a Legendrian submanifold of $\Bbb R^{2n+1}$ which is Legendrian isotopic to the flying saucer (Example \ref{eg-saucer}). To this end, we apply the toroidal Legendrian Reidemeister move $\rm{I}$ (Definition \ref{def-reid1torus}) twice to the smooth front $\{z = 0\} \subset \Bbb R^{n+1}_{xz}$, see Figure \ref{fig-r12surface}.

\begin{figure}[h]
\centering
\includegraphics[scale=0.4]{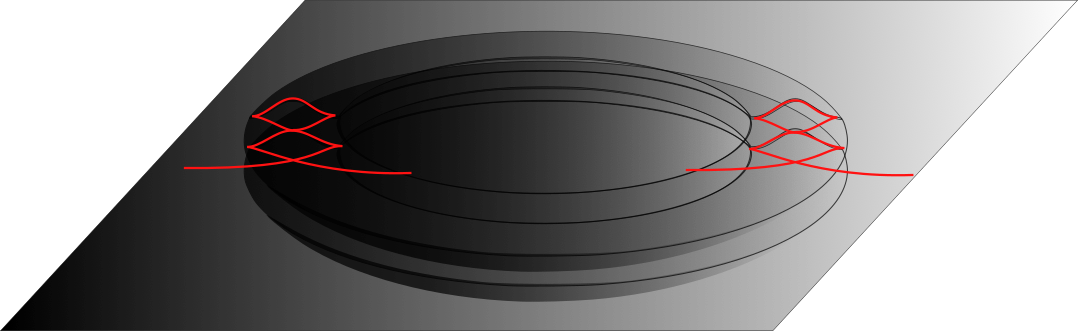}
\caption{Toroidal Legendrian Reidemeister move $\rm{I}$ applied twice.}
\label{fig-r12surface}
\end{figure}

We may smoothly isotope this front to contain a copy of $F$ as indicated in a radial slice in Figure \ref{fig-r1weave}. 
By applying the toroidal Legendrian Reidemister move $\rm{I}$ twice on a smooth portion of the front for the Legendrian flying saucer illustrated in Figure \ref{fig-legsaucer}, we thereby obtain the desired embedding. Thus, the Legendrian flying saucer admits a pseudo-loose chart with given any given size parameter $\rho^2/a < 1/2$. 

If a pseudo-loose chart with a given size parameter $\rho^2/a < 1/2$ contained a contactly embedded copy of a loose chart, it would thus imply that the Legendrian flying saucer is a loose Legendrian. This is a contradiction, by Proposition \ref{prop-saucernloose}. Therefore, a loose chart does not contactly embed in a pseudo-loose chart, as required.\end{proof}

\begin{figure}[h]
\centering
\includegraphics[scale=0.5]{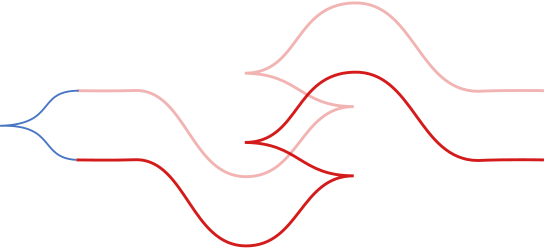}
\caption{Legendrian Reidemeister move $\rm{I}$ containing a copy of the doubling link of a zig-zag.}
\label{fig-r1weave}
\end{figure}

\begin{remark}The proof of Proposition \ref{prop-loosenonsq} shows that any Legendrian submanifold admits a pseudo-loose chart. Thus, admitting a pseudo-loose chart \emph{a posteriori} does not impose any restriction on the Legendrian submanifold, unlike admitting a loose chart.\end{remark}

\begin{remark}The proof of Proposition \ref{prop-loosenonsq} shows that a Legendrian link in a contact manifold whose components are loose Legendrian submanifolds, need not be a loose Legendrian embedding itself. Indeed, the submanifold $\Lambda \subset (\Bbb R^{2n+1}, \alpha_{\rm{std}})$ defined in the proof of Proposition \ref{prop-loosenonsq} consists of a loose embedded $\Bbb R^n$ in $(\Bbb R^{2n+1}, \alpha_{\rm{std}})$ as one component and a push-off it by the Reeb vector field as another component. Both of these components are individually loose, but nonetheless $\Lambda$ is non-loose. Intuitively, this is because the loose charts for one of the components may intersect the other component, which prevents it from being an adapted chart for the entire link. In general, there may be no contact isotopy making these charts disjoint from the other components, as is the case for $\Lambda$.
\end{remark}

\begin{remark}Eliashberg-Kim-Polterovich \cite[Theorem 1.2]{ekp} prove a version of the contact nonsqueezing theorem for the domains $D^{2n}(0; r_1) \times S^1$ and $C_{r_2} \times S^1$, in the contact manifold $\Bbb R^{2n} \times S^1$ equipped with the contact form $d\theta - \lambda$, where $\omega_{\rm{std}} = -d\lambda$. Moreover, in \cite[Theorem 1.3]{ekp} they prove a squeezing theorem for $D^{2n}(0; r_1) \times S^1$ in $D^{2n}(0; r_2) \times S^1$ for all $r_1, r_2 < 1$. These results perhaps appear analogous to Proposition \ref{prop-loosenonsq} and Propositions \ref{prop-manyloose}, respectively. Note, however, that a loose chart is a pair of spaces, consisting of an open chart of dimension $2n+1$ and a Legendrian chart of dimension $n$. Thus, the nonsqueezing and squeezing results in \cite{ekp} are absolute, whereas the corresponding results in the aforementioned propositions are relative.
\end{remark}

\end{document}